\documentclass[10pt,a4paper,leqno]{amsart}
\usepackage{amsmath}
\usepackage{amssymb}
\usepackage{amsfonts}
\usepackage{amsthm}
\usepackage{mathrsfs}
\usepackage{dsfont}
\usepackage{mathtools}
\usepackage{bm}
\usepackage{caption}
\usepackage[backref=page]{hyperref}
\usepackage{xcolor}

\definecolor{darkblue}{rgb}{0,0,1}
\hypersetup{colorlinks,breaklinks,
linkcolor=darkblue,urlcolor=darkblue,
anchorcolor=darkblue,citecolor=darkblue}

\newcommand*{\mailto}[1]{\href{mailto:#1}{\nolinkurl{#1}}}

\theoremstyle{plain}
\newtheorem{theorem}{Theorem}

\newtheorem{lemma}[theorem]{Lemma}
\newtheorem{corollary}[theorem]{Corollary}
\newtheorem{proposition}[theorem]{Proposition}

\newtheorem{definition}[theorem]{Definition}
\theoremstyle{plain}
\newtheorem{remark}[theorem]{Remark}

\numberwithin{theorem}{section}
\numberwithin{equation}{section}

\DeclareMathOperator*{\esssup}{\rm ess\,sup}

\def\mff{{\mathfrak f}}
\def\mx{{\bf x}}
\def\my{{\bf y}}

\def\pa{\partial}
\def\cal{\mathcal}
\let\mib=\boldsymbol

\def\C{\mathbb{C}}
\def\R{\mathbb{R}}

\def\N{\mathbb{N}}
\def\S{\mathbb{S}}
\def\mE{\mathscr{E}}

\def\mxi{{\mib \xi}}
\def\mx{{\mib x}}
\def\my{{\mib y}}

\def\mzeta{{\mib \zeta}}
\def\meta{{\mib \eta}}

\def\sgn{\operatorname{sgn}}

\newcommand{\cD}{\mathcal{D}}
\newcommand{\cDp}{\mathcal{D}^\prime}
\newcommand{\cC}{\mathcal{C}}
\newcommand{\cS}{\mathcal{S}}
\newcommand{\cA}{\mathcal{A}}
\newcommand{\cB}{\mathcal{B}}
\newcommand{\cF}{\mathcal{F}}
\newcommand{\cI}{\mathcal{I}}
\newcommand{\bbH}{\mathbb{H}}
\newcommand{\bbK}{\mathbb{K}}
\newcommand{\bbT}{\mathbb{T}}

\newcommand{\norm}[1]{\left \lVert#1 \right\rVert}

\newcommand{\cM}{\mathcal{M}}
\newcommand{\cL}{\mathcal{L}}

\newcommand{\abs}[1]{\left\lvert#1\right\rvert}
\newcommand{\absb}[1]{\bigl\lvert#1\bigr\rvert}

\newcommand{\inn}[1]{\left\langle#1\right\rangle}
\newcommand{\innb}[1]{\bigl\langle#1\bigr\rangle}

\newcommand{\action}[2]{\left\langle #1, #2 \right\rangle}
\newcommand{\actionb}[2]{\bigl\langle #1, #2 \bigr\rangle}
\newcommand{\loc}{\operatorname{loc}}
\newcommand{\supp}{\operatorname{supp}}
\newcommand{\Div}{\operatorname{div}}
\newcommand{\diag}{\operatorname{diag}}

\newcommand{\seq}[1]{\left\{#1\right\}}
\newcommand{\seqb}[1]{\bigl\{#1\bigr\}}

\newcommand{\meas}{\operatorname{meas}}
\newcommand{\En}{\mathbf{1}}

\makeatletter
\@namedef{subjclassname@2020}{%
  \textup{2020} Mathematics Subject Classification}
\makeatother

\begin{document}

\title[Regularity of solutions to degenerate equations]
{On the regularity of entropy solutions to stochastic 
degenerate parabolic equations}

\author[Erceg]{M. Erceg}
\address[Marko Erceg]
{Faculty of Science, Department of Mathematics, 
University of Zagreb, 
Bijeni\v cka cesta 30, 
10000 Zagreb, Croatia}
\email{\mailto{maerceg@math.hr}}

\author[Karlsen]{K. H. Karlsen}
\address[Kenneth H. Karlsen]
{Department of Mathematics, 
University of Oslo, 
P.O. Box 1053 Blindern, 
0316 Oslo, Norway}
\email{\mailto{kennethk@math.uio.no}}

\author[Mitrovi\'{c}]{D. Mitrovi\'{c}}
\address[Darko Mitrovi\'{c}]
{University of Montenegro, 
Faculty of Mathematics, Cetinjski put bb,
81000 Podgorica, Montenegro}
\email{\mailto{darkom@ucg.ac.at}}



\begin{abstract}
We study the regularity of entropy solutions for quasilinear 
parabolic equations with anisotropic degeneracy and stochastic forcing. 
Building on previous work, we establish space-time regularity under
a non-degeneracy condition that does not require an assumption on the 
derivative of the symbol of the corresponding kinetic equation---a 
restriction imposed in earlier studies. This allows us to 
obtain regularity results for certain equations 
not accounted for by prior theory, albeit with reduced 
regularity exponents. Our approach uses a kinetic formulation with two
transport equations---one of second order and 
one of first order---leveraging a form of ``parabolic regularity" 
inherent in these equations that was 
not utilized in previous studies.
\end{abstract}

\date{\today}

\maketitle


\section{Introduction}
We consider anisotropic degenerate parabolic equations 
subject to stochastic forcing in the form
\begin{equation}\label{d-p}
	\begin{split}
		\partial_t u + \Div_{\!\mx} \mff(u) 
		&= \Div_{\!\mx} ( a(u) \nabla_\mx u) 
		+B(u)\dot{W}(t)\quad \text{in} \ \cD'(\R_+^{d+1}),
	\end{split}
\end{equation}
where $(t, \mx) \in \R_+^{d+1}:=\R_+\times \R^d$, $W$ 
is a cylindrical Wiener process \cite{DaPrato:2014aa} with 
noise amplitude $B$, such that 
$B(u)\dot{W} = \sum_\ell b_\ell(u)\dot{W}_\ell$ 
for some real-valued functions $\seq{b_\ell}$ 
and independent real-valued Wiener processes 
$\seq{W_\ell}$ (we refer to Section \ref{sec:stoch-framework} 
for the precise assumptions), and $u: \Omega \times \R_+^{d+1} \to \R$ 
is the unknown function.

The nonlinear coefficients $\mff : \R\to \R^d$ and 
$a:\R \to \R^{d\times d}$ 
are given vector- and matrix-valued functions. 
We impose the following assumptions on $\mff,a$:
\begin{itemize}
	\item[({\bf a})] $\mff \in C^2(\R;\R^d)$ and 
	set $f:= \pa_\lambda \mff \in C^1(\R;\R^d)$.

	\item[({\bf b})] $a \in C^1(\R; \R^{d\times d})$ and 
	there exists a square-root matrix 
	$\sigma\in L^\infty(\R; \R^{d\times d})$ 
	such that $\sigma=\sigma^T$ and $a=\sigma^2$.
	We take $\sigma=a^{1/2}$ to be the principal square
	root (which is continuous).
	Thus, for
	$a(\lambda)=\seq{a_{ij}(\lambda)}_{i,j=1}^d$, we have
	\begin{align}
		\label{square-root}
		a_{ij}(\lambda)&
		=\sum\limits_{k=1}^d \sigma_{ik}(\lambda)
		\sigma_{kj}(\lambda),
		\\ \notag
		\bigl \langle a(\lambda) \mxi\,|\,\mxi \bigr \rangle & 
		=\sum\limits_{k=1}^d \left(\sum\limits_{i=1}^d 
		\sigma_{ik}(\lambda)\xi_i \right)^2
		=\abs{\sigma(\lambda)\mxi}^2 \geq 0, 
		\quad \mxi \in \R^d,
\end{align}
i.e., $a$ is a symmetric, positive 
semi-definite, matrix-valued function.
\end{itemize} 

The first-order term in \eqref{d-p} captures the 
nonlinear convection effects intrinsic to the 
physical process being modelled, whereas the second-order 
term represents the diffusion-driven nonlinear 
evolution of the unknown quantity $u$. 
This equation models various physical processes and appears 
in a wide range of applications, including flow in porous 
media (e.g., \cite{Espedal:2000sr,Nordbotten:2011aa}), 
sedimentation-consolidation processes \cite{Bustos:1999gb}, 
and finance \cite{Polidoro:2023aa}, among others.

The stochastic component of equation \eqref{d-p} facilitates the
modelling of systems subject to random fluctuations in sources or
external forcing. Such stochastic perturbations arise in models
from biology, physics, and engineering, see, e.g.,
\cite{Lord:2014aa,LototskyRozovsky2017} and the references therein.
The Wiener forcing idealizes rapidly varying temporal fluctuations,
while the coefficients $b_\ell(t,\mx,u)$ encode spatial correlations
and state-dependent noise intensity.

Our analysis requires no particular algebraic form of these
coefficients. The square-summable bounds in \eqref{eq:B-main-ass}
below provide the growth and derivative control needed for the stochastic
integrals and the It\^o correction. Within this noise class, the
detailed coefficients affect the constants in the estimates,
whereas the stated regularity exponents depend only on the
non-degeneracy parameter $\alpha$ and the dimension $d$.

The equation \eqref{d-p} is strongly 
\textit{degenerate parabolic}. 
Specifically, there may exist directions 
$\mxi \in \R^d \backslash \{0\}$ and 
states $\lambda \in \R$ such that 
$\bigl\langle a(\lambda) \mxi\,|\, \mxi \bigr\rangle=0$.
This reduces the equation locally to a hyperbolic equation, 
implying that solutions cannot be expected 
to remain regular for all times. 
Hence, ensuring the uniqueness of weak solutions 
necessitates the imposition of appropriate 
entropy conditions to manage potential discontinuities 
in the solutions. These conditions were initially 
introduced in the deterministic setting 
of \cite{Volpert:1969fk}. Partial uniqueness results for 
the initial value (Cauchy) problem were 
established in \cite{Volpert:1969fk,Wu:1989ve} 
(see also \cite{Wu:2001fk}). A general uniqueness 
result was first proven in \cite{Carrillo:1999hq} 
for the isotropic diffusion case, 
where $a(\cdot)$ is a scalar function. Subsequently, 
uniqueness was established in \cite{Chen:2003td} for the 
general anisotropic diffusion case considered in this work.  

A general well-posedness theory for the stochastic 
case was first examined in \cite{Feng:2008ul} and
\cite{Debussche:2010fk} for first-order equations.
An extension of this theory to cover degenerate 
second-order operators, following the approach 
in \cite{Chen:2003td}, can be found 
in \cite{Debussche:2016aa,Gess:2018ab}. 
We note that while the deterministic 
setting in \cite{Chen:2003td} applies to $\R^d$, 
the stochastic framework of \cite{Debussche:2016aa,Gess:2018ab} is 
restricted to the torus $\mathbb{T}^d$. 
For more detailed discussions of the well-posedness 
theory, see Section \ref{sec:stoch-framework}.

Given a convex (entropy) function $S \in C^2(\R)$, define 
the vector- and matrix-valued entropy fluxes 
$Q_S: \R \rightarrow \R^d$, $R_S: \R \rightarrow \R^{d\times d}$  
by $Q_S'(u) = S'(u)f(u)$, $R_S'(u) = S'(u)a(u)$. 
We call $(S, Q_S,R_S)$ an \emph{entropy/entropy-flux triple} 
and write $(S, Q_S,R_S) \in \mE$. 
The entropy inequalities, which serve to differentiate 
between non-unique weak solutions, are expressed as follows:
\begin{equation}\label{eq:dp-entropy-ineq}
	\begin{split}
		&\partial_t S(u) + \Div_{\!\mx} Q_S(u) 
		-\Div_{\!\mx} \Div_{\!\mx} R_S(u)
		\\ & \quad 
		\le -S''(u)\sum_{i,j=1}^d a_{ij}(u)
		\partial_{x_i}u\partial_{x_j}u
		+\sum_{\ell \ge 1} S'(u) b_\ell(u)
		\, \dot{W}_\ell+\frac{1}{2} S''(u) B^2(u),
	\end{split}
\end{equation}
a.s.~in $\cD'_{t,\mx}:=\cD'(\R_+^{d+1})$, 
$\forall (S,Q_S,R_S) \in \mE$. The first term on 
the right-hand side can be more compactly 
written as $-S''(u)\abs{\sigma(u)\nabla_{\!\mx}u}^2$.
The rigorous interpretation of the dissipation term involving
$\partial_{x_i}u$ in \eqref{eq:dp-entropy-ineq} is given in
Subsection~\ref{sec:chain-rule} below.
In this context, we refer to $u$ as an 
\textit{entropy solution} of \eqref{d-p}. 
To ensure that this notion is well-defined, we require $u$ 
and $\sigma(u) \nabla_{\!\mx} u$ to be sufficiently 
integrable. Specifically, we assume 
$u \in L^p_\omega(L^\infty_t L^p_x)$ for $p \in [2, \infty)$, 
potentially with $L^p_x$ replaced by a suitable 
weighted space, depending on the assumptions on the coefficients 
$\mff, a, B$. Additionally, we assume that $\sigma(u) \nabla_{\!\mx} u 
\in L^2_{\omega,t,\mx}$ (see Section 
\ref{sec:stoch-framework} for details).

As initially observed in \cite{Lions:1994qy} within the 
deterministic context, the satisfaction of all entropy 
inequalities can be equivalently expressed through 
a second-order kinetic equation. 
This kinetic formulation has been utilized for the analysis of 
well-posedness in \cite{Chen:2003td} and \cite{Gess:2018ab}. 
When incorporating the stochastic forcing term, as first done 
in \cite{Debussche:2010fk} (when $a\equiv 0$), this 
equation may be written as 
\begin{equation}\label{diffusive-intro}
	\begin{split}
		& \pa_t h+\sum_{i=1}^d\pa_{x_i}
		\bigl( f_i(\lambda) h \bigr)
		-\sum_{i,j=1}^d\pa_{x_i x_j}^2
		\bigl(a_{ij}(\lambda) h \bigr)
		\\ & \qquad\quad 
		+ \sum_{\ell \ge 1} b_\ell(\lambda)
		\pa_\lambda h \, \dot{W}_\ell(t)
		= \pa_\lambda \left(\tfrac{1}{2}B^2(\lambda)
		\pa_\lambda h \right) 
		+\pa_\lambda\zeta, 
	\end{split}
\end{equation}
a.s.~in $\cD'_{t,\mx,\lambda}:=\cD'(\R_+^{d+1}\times \R)$. 
Here, $h=\En_{\lambda<u(t,\mx)}$ and $B^2=\sum_\ell b_\ell^2$. 
The term $\zeta = \zeta(t, \mx, \lambda)$ is 
a random variable valued in the space of non-negative, 
locally finite measures, denoted as $\zeta \in 
\cM_{\loc}(\R_+^{d+1} \times \R)$ 
(see Section \ref{sec:stoch-framework}).

Since the foundational work \cite{Lions:1994qy}, it has been 
known that velocity averaging lemmas for kinetic 
equations are effective tools for 
establishing compactness and regularity properties of entropy 
solutions to hyperbolic and mixed hyperbolic-parabolic equations. 
These lemmas demonstrate that, under certain conditions, 
the $\lambda$-averages 
$\overline{h}(t,\mx):=\int h(t,\mx,\lambda)\rho(\lambda)
\,d\lambda$, $\rho \in C^1_c(\R)$, 
exhibit enhanced properties---often characterised in 
terms of strong compactness of sequences of solutions or 
bounds in fractional Sobolev or Besov spaces. 
The area of velocity averaging has a long 
and rich history, with pioneering contributions 
such as \cite{Agoshkov:84,Golse-etal:88} in the $L^2$ setting,
followed by extensions to the $L^p$ ($p > 1$) framework in 
\cite{Bezard:94,DiPerna-etal:91}. More recent works, such 
as \cite{Tadmor:2006vn} and 
\cite{Arsenio:2021aa,Jabin:2022aa}, 
further explore these concepts and 
provide a list of relevant references.

The original work \cite{Lions:1994qy} established a 
velocity averaging result for \eqref{diffusive-intro}
when $B\equiv 0$ and 
demonstrated the $L^1_{\loc}$ compactness of sequences 
of bounded entropy solutions to \eqref{d-p}. This compactness 
result holds under the non-degeneracy condition:
\begin{equation}\label{lp-111-intro}
	\sup\limits_{(\xi_0,\mxi) \in \S^d} 
	\meas\Bigl\{ \lambda \in I  \mid 
	\xi_0+\bigl\langle \mff'(\lambda) \mid \mxi\bigr \rangle=0, 
	\, \bigl\langle a(\lambda)\mxi 
	\mid \mxi \bigr\rangle=0\Bigr\}=0,
\end{equation}
for a compact set $I \subset \R$. Here, $\S^d \subset \R^{d+1}$ 
represents the unit sphere, i.e., the set of points 
$(\xi_0,\mxi)\in\R^{d+1}$ such that $\abs{(\xi_0,\mxi)}=1$. 
Recently, \cite{Erceg:2023aa} extended this result 
by employing alternative techniques based on microlocal 
defect measures, applying them to degenerate parabolic equations 
akin to \eqref{d-p}, but incorporating a 
heterogeneous convection flux $\mff(t,\mx,u)$. 
We refer the reader to \cite{Erceg:2023aa}
for additional references
(see also \cite{Nariyoshi:2020aa}).

The work \cite{Tadmor:2006vn} 
was the first to establish quantitative  estimates---more 
precisely $W^{s,p}_{\loc}$ estimates---for entropy 
solutions to degenerate parabolic equations like \eqref{d-p}
with $B\equiv 0$. An extension of these estimates to 
equations with a stochastic forcing 
term $B(u)\dot{W}$ can be found in \cite{Gess:2018ab}.
In \cite{Tadmor:2006vn}, fractional Sobolev regularity of 
entropy solutions is derived by introducing a novel velocity 
averaging lemma for higher-order kinetic equations, such 
as \eqref{diffusive-intro}. However, this lemma is applied under 
more stringent conditions than those specified in \eqref{lp-111-intro}.  
The conditions in \cite[(2.19) \& (2.20)]{Tadmor:2006vn} 
consist of two parts, which we will now recall, appropriately 
adapted to the notation of this paper. Consider the set
$$
D_{\cL}(\xi_0,\mxi;\delta)
:=\Bigl\{ \lambda \in I \mid
\abs{\xi_0+\bigl\langle \mff'(\lambda)
\mid \mxi \bigr\rangle}^2
+\bigl\langle a(\lambda)\mxi \mid \mxi
\bigr\rangle^2 \leq\delta^2\Bigr\},
$$
where $I \subset\subset \R$. 
The first condition \cite[(2.19)]{Tadmor:2006vn} 
asserts that there exist $\alpha, \beta > 0$ such that for every 
$J \gtrsim 1$ and $\delta > 0$, the following holds:
\begin{equation}\label{c-TT-intro}
	\sup_{\abs{(\xi_0,\mxi)}\sim J}
	\meas\bigl(D_{\cL}(\xi_0,\mxi;\delta)\bigr)
	\lesssim \left(\frac{\delta}{J^\beta}\right)^\alpha.
\end{equation}
Let $\cL(\xi_0,\mxi,\lambda):=i\bigl(\xi_0+\bigl\langle
\mff'(\lambda)\mid
\mxi \bigr\rangle\bigr)+\bigl\langle a(\lambda)\mxi \mid \mxi\bigr\rangle$ 
denote the symbol of the kinetic equation \eqref{diffusive-intro} 
evaluated at $(\xi_0,\mxi,\lambda)$. Then condition \eqref{c-TT-intro} 
quantifies the measure of the set of $\lambda$ for 
which $\abs{\cL(\xi_0,\mxi,\lambda)}^2 \leq \delta^2$ 
with $\abs{(\xi_0,\mxi)} \sim J$. The second condition 
\cite[(2.20)]{Tadmor:2006vn} imposes a requirement 
also on the $\lambda$-derivative of the 
symbol, $\cL_\lambda(\xi_0,\mxi,\lambda)$. 
Specifically, there exist 
$\kappa \geq 0$ and $\mu \in [0,1]$ such that
\begin{equation}\label{c-TT-Lder-intro}
	\sup_{\abs{(\xi_0,\xi)} \sim J}
	\, \sup_{\lambda \in D_{\cL}(\xi_0,\mxi;\delta)}
	\abs{\cL_\lambda(\xi_0,\mxi,\lambda)}
	\lesssim J^{\beta \kappa} \delta^\mu,
\end{equation}
which asks for a nondegeneracy condition on the 
derivatives $\mff''(\cdot)$ and $a'(\cdot)$. 
The assumptions in the stochastic forcing 
case are similar \cite{Gess:2018ab}.

The objective of this paper is to develop regularity 
estimates for entropy solutions without imposing a condition 
like \eqref{c-TT-Lder-intro} on the derivative of the symbol. 
In \cite[Remark 2.5]{Tadmor:2006vn}, the authors observe that since 
the order of $\cL(\xi_0,\mxi,\lambda)$ is at most 2, their second 
condition (2.20) (see \eqref{c-TT-Lder-intro}) is always 
satisfied with $\mu=0$ and $\beta \kappa=2$. 
Assuming $h \in L^p_{\loc}(\R_+^{d+1} \times \R)$ for $1 < p \leq 2$, the 
resulting averaging regularity is $\overline{h} 
\in W^{s,q}_{\loc}(\R_+^{d+1})$, with the regularity exponent 
$s$ being any number in the interval $(0, s_{\alpha,\beta})$ 
and the integrability exponent $q$ 
being a number in $(1, p)$, where
$$
s_{\alpha,\beta} = 2\theta_\alpha(\beta-1),
\quad
\theta_\alpha = \frac{\alpha}{\alpha + 2p'}, 
\quad 
\frac{1}{p}+\frac{1}{p'}=1.
$$
Clearly, under only condition (2.19) from \cite{Tadmor:2006vn}, 
the regularization effect is lost 
if $\beta=1$, which would correspond to 
a convection dominated regime. 
Consider, for example, the two-dimensional 
(deterministic) kinetic equation 
\begin{equation}\label{eq:intro-kinetic-2D}
	\pa_t h+\pa_{x_1} \bigl( \lambda \, h \bigr)
	-\pa_{x_2x_2}^2\bigl(\abs{\lambda} h \bigr)
	=\pa_\lambda\zeta,
\end{equation}
in which case the symbol becomes 
$\cL(\xi_0,\mxi,\lambda)=i(\xi_0+\lambda \xi_1)
+\abs{\lambda} \xi_2^2$. One can easily verify that condition 
\eqref{c-TT-intro} \cite[(2.19)]{Tadmor:2006vn} 
holds with $\beta=1$ and $\alpha=1$. Regarding the second condition 
\eqref{c-TT-Lder-intro} \cite[(2.20)]{Tadmor:2006vn}, it 
holds with $\beta=1$, $\kappa=2$, and $\mu=0$.
Thus, for this choice of parameters, the exponent in
\cite[Averaging Lemma 2.3]{Tadmor:2006vn} is $s_{\alpha,\beta}=0$.

Our result encompasses degenerate equations 
like \eqref{eq:intro-kinetic-2D}, though with the trade-off 
of achieving regularity exponents that are 
lower than those predicted by \cite{Tadmor:2006vn}. 
The work \cite{Gess:2018ab} extends \cite{Tadmor:2006vn} to 
the stochastic setting of \eqref{d-p}; similarly, our results apply to 
a broader class of stochastic equations, albeit with lower regularity 
exponents than \cite{Gess:2018ab}. 
Moreover, we establish \textit{space-time} fractional Sobolev regularity 
in the stochastic case, whereas \cite{Gess:2018ab} 
only addresses spatial regularity. 

Our primary result is as follows (a more precise statement 
will be provided in Section \ref{sec:regularity}):

\begin{theorem}\label{thm:main-intro}
Suppose the coefficients $\mff$ and $a$ 
from the mixed hyperbolic-parabolic SPDE \eqref{d-p} satisfy 
assumptions ({\bf a}) and ({\bf b}) as well as the following 
non-degeneracy condition:
\begin{equation}\label{non-deg-stand}
	\sup\limits_{|(\xi_0,\mxi)|=1} 
	\meas\Bigl\{\lambda \in I \, \Big | \, 
	\big| \xi_0+\bigl\langle \mff'(\lambda) 
	\, | \, \mxi\bigr\rangle \big|^2
	+\bigl\langle a(\lambda)\mxi \,|\, \mxi\bigr\rangle 
	\leq \delta \Bigr\}\lesssim \delta^{\alpha},
\end{equation}  
for all $\delta \in (0, \delta_0)$, where $\alpha, \delta_0 > 0$ 
and $I \subset \R$ is a compact interval. 
Suppose that the noise amplitude $B$ fulfills 
the forthcoming conditions \eqref{eq:b-ell-def} 
and \eqref{eq:B-main-ass}. Let $u$ be an entropy solution
of \eqref{d-p} and set $F_\rho(v):=\int_{-\infty}^v\rho(\lambda)\,d\lambda$.
There exist $s\in(0,1)$ and $q\in(1,2)$ such that, for every
$\rho\in C_c^1(I)$, $F_\rho(u)\in W^{s,q}_{\loc}(\R_+\times\R^d)$
a.s. If the essential range of $u$ on a bounded
Lipschitz domain $Q\Subset\R_+^{d+1}$ lies a.s.~in a fixed compact
interval $J\Subset I^\circ$, then $u\in W^{s,q}(Q)$ a.s.
\end{theorem}

\begin{remark}
Although Theorem \ref{thm:main-intro} is stated for
$a,f\in C^1$, the estimates in Step III extend to locally
Lipschitz $a,f$ by regularizing the velocity test functions in
\eqref{eq:ThetaJ-def}. No Lipschitz regularity of
$\sigma=a^{1/2}$ is required. In particular, this includes
the example \eqref{eq:intro-kinetic-2D}.
\end{remark}

\begin{remark}\label{rem:sobolev_intro}
The regularity results presented 
in \cite{Gess:2018ab,Tadmor:2006vn} 
are formulated in terms of Bessel potential spaces, 
which we will denote by $H^{\kappa,r}(\R^d)$ or 
$H^{\kappa,r}_{\loc}(\R^d)$ for the local versions, 
as per \cite{Hytonen:2016aa}. 
On the other hand, our result is expressed in terms
of the Sobolev-Slobodetskii spaces $W^{\kappa,r}(\R^D)$, 
where $D\in\N$. However, this distinction is irrelevant 
because, under the assumptions on $s$ and $q$ stated 
in Theorem \ref{thm:main-intro}, for any $\bar s>s$, we have
\begin{equation*}
	H^{\bar s,q}(\R^d) \hookrightarrow W^{s,q}(\R^d)
	\hookrightarrow H^{s,q}(\R^d),
\end{equation*}
see \cite[Corollary 14.4.25, Theorem 14.7.9, 
and p.~370]{Hytonen:2023aa}. 
For any Hilbert space $\bbH$ and $1<q<\infty$, we also have
$$
W^{s',q}(\R^D;\bbH)\hookrightarrow H^{s,q}(\R^D;\bbH),
\qquad 0<s<s'<1,
$$
see \cite[Corollary 14.4.25, (14.23), and
Proposition 14.4.18]{Hytonen:2023aa}.
Thus, the emphasis should be placed on the values 
of the parameters $s$ and $q$, rather than on 
the particular definition employed.
\end{remark}

Theorem \ref{thm:main-intro} 
assumes a quantitative form \eqref{non-deg-stand} 
of the non-degeneracy condition \eqref{lp-111-intro}, 
without requiring a derivative condition 
like \eqref{c-TT-Lder-intro}, which imposes a growth rate assumption 
on the $\lambda$-derivative of the symbol of \eqref{diffusive-intro} 
over Littlewood-Paley dyadic blocks (see \cite[(2.3)]{Gess:2018ab}, 
\cite[(2.19), (2.20)]{Tadmor:2006vn}). 
The assumption \eqref{non-deg-stand} 
avoids the restrictiveness and complexity of 
verifying \eqref{c-TT-Lder-intro}, a condition 
discussed for some special cases 
of \eqref{d-p} (see \cite[Subsection 2.4]{Gess:2018ab}, 
\cite[Corollaries 4.1-4.5]{Tadmor:2006vn}).
The symbol in \eqref{non-deg-stand} 
is homogeneous of degree 2. However, even if the square 
is removed from \eqref{non-deg-stand}, the result 
remains valid. For a discussion, see 
Section \ref{sec:discussions}.

For the rank-one convection--diffusion equation \eqref{TT-1} below,
considered in \cite[Corollary 4.5]{Tadmor:2006vn}, our Corollary~\ref{TT-ex}
gives local space--time regularity for every pair of integers $n,\ell\ge1$. A 
similar result appears in the work \cite{Nariyoshi:2020aa}. 
For the porous medium equation (i.e., \eqref{d-p}
with $\mff,B\equiv 0$),
optimal regularity has been established in \cite{Gess:2021aa,Gess:2020aa}. 
Although our regularity exponent $s$ is lower than in previous works, 
it applies more generally (see Theorem \ref{t-regularity} for details). 
In the stochastic setting, quantitative regularity estimates, 
even with reduced exponents, facilitate the stochastic compactness 
method, enabling convergence proofs at the entropy solution level 
and evading the intricate kinetic-level analysis required by noise terms, 
as exemplified in \cite{Karlsen:2022aa}.

Our proof bypasses the derivative condition 
\eqref{c-TT-Lder-intro} by leveraging a 
kinetic formulation with two transport equations---one of second 
order and one of first order. This formulation harnesses a form 
of ``parabolic regularity" linked to the positive 
semidefiniteness of the diffusion matrix $a$, 
a feature not used in prior works. 
For a related independent idea that has been used 
in the deterministic setting, see 
\cite{Nariyoshi:aa}. 

In \cite{Gess:2018ab,Tadmor:2006vn}, the positive 
semidefiniteness of $a$ was only moderately used (to verify 
the so-called truncation property). No explicit sign constraint on $a$ 
was required, as illustrated by examples like 
$a(\lambda)=\diag(\lambda,\lambda^3)$, 
where $\lambda \in \R$ (which adheres to the truncation property). 
In essence, the nonnegativity of the parabolic 
dissipation measure on the right-hand side 
of \eqref{eq:dp-entropy-ineq} is not crucial 
in these works. See \cite{Gess:2019ab} for a discussion 
on not assuming non-negativity of entropy defect 
measures in the context of scalar conservation laws.

We utilize the condition $a \geq 0$ 
(see \eqref{square-root}), which implies that 
$\sigma(u) \nabla_{\!\mx} u \in L^2_{\omega, t, \mx}$, 
to derive an additional kinetic equation.
Here, the second-order spatial differential 
operator in \eqref{diffusive-intro} 
is replaced by the first-order operator 
$\Div_{\!x} \bigl(\sigma(\lambda)\beta\bigr)$, where 
$\beta = \sigma(\lambda)\nabla_{\!x} h$ can informally 
be understood via the relation
$\nabla_{\!\mx}h
=\delta(\lambda-u)\nabla_{\!\mx}u$.
Therefore, the vector $\beta$ belongs a.s.~to the space 
$L^2_{w\star}(K;\cM(\R))$ (weak-star measurable 
and square-integrable mappings 
with values in the space of finite Radon measures), 
for any $K\subset\subset \R_+^{d+1}$. 
See Section \ref{sec:stoch-framework} 
for the rigorous interpretation, 
which involves the chain rule property 
from \cite{Chen:2003td}. 
Given these two kinetic equations, 
our overall proof strategy is similar to 
that in \cite{Tadmor:2006vn} 
and \cite{Gess:2018ab}, involving a Fourier 
space decomposition of the kinetic solution $h$. 
This approach examines the Fourier transform of 
$h$ in the region where the symbol of 
the kinetic equation is non-zero, while 
ensuring that the complementary region, 
controlled by the non-degeneracy assumption, 
remains sufficiently small. However, in
\cite{Gess:2018ab, Tadmor:2006vn} the authors had to 
simultaneously account for the flux and diffusion terms, 
which introduces significant technical challenges that 
could only be resolved by imposing additional 
non-degeneracy conditions (see the discussion above 
and compare, in particular, \eqref{c-TT-intro} 
and \eqref{c-TT-Lder-intro} with \eqref{non-deg-stand}). 

As a final element of the approach, we mention 
the formulation of the stochastic part of the kinetic equation. 
The last term on the left-hand side of \eqref{diffusive-intro} 
can be expressed formally as: 
$\sum_{\ell \ge 1} b_\ell(\lambda)
\pa_\lambda h \dot{W}_\ell(t)
=\pa_t \cI(B\pa_\lambda h)$, 
where $\cI(\Phi)$ denotes the It\^{o} integral 
of a process $\Phi$ with respect to $W$. 
Hence, we interpret the kinetic 
equation \eqref{diffusive-intro} a.s.~in the 
sense of distributions with respect 
to both time $t$ and space-velocity 
variables $(\mx, \lambda)$. 
This approach represents the stochastic forcing term 
as a distributional time derivative. This allows us to 
directly bound terms like $\inn{\pa_t \cI(B\pa_\lambda h),
\varrho}_{\cD'_\lambda,\cD_\lambda}$, for 
$\varrho\in \cD_\lambda:=\cD(\R)$, in spaces like 
$L^2(\Omega; W^{-\nu,2}([0, T]; \bbH))$ 
for $\nu > 1/2$ and $T > 0$, provided that $\cI(B(u))$ 
takes values in a Hilbert space $\bbH$ 
(see Section \ref{sec:stoch-framework} for further details).  
This also enables a concurrent analysis of 
both spatial and temporal regularity.
The regularity analysis will further rely 
on expressing $\pa_t \cI(B\pa_\lambda h)$ 
in divergence form as $\pa_\lambda\pa_t \cI(B h)
-\pa_t \cI(\pa_\lambda B h)$.
Let us also mention that working in 
an unbounded domain ($\R^d$) 
introduces additional complications in handling 
these stochastic integral terms compared to the case 
of the torus \cite{Gess:2018ab}.

Our regularity results are also new in the deterministic case.
The terms $\pa_\lambda\pa_t\cI(Bh)$ and
$\pa_t\cI((\pa_\lambda B)h)$ are not, in general, finite measure
source terms. Their control requires
Burkholder--Davis--Gundy inequalities and fractional temporal regularity
estimates with exponent $\nu<1/2$. This restriction constrains the
Fourier parameter optimization: for velocity averages, our argument
gives $s<s_*$ with stochastic forcing and $s<2s_*$ when $B\equiv0$ 
(see Section \ref{sec:discussions}). These exponents describe the
regularity obtained by our method, without asserting optimality.

The remainder of this paper is organized as follows: 
Section \ref{sec:stoch-framework} presents 
the stochastic framework, 
solution concepts, and relevant kinetic equations. 
Section \ref{sec:regularity} states and 
proves the main regularity result.
Finally, in Section \ref{sec:discussions}, we 
discuss the regularity result and its 
underlying assumptions in the deterministic case.

\section{Background material}\label{sec:stoch-framework}

\subsection{Stochastic framework}

Recall that we consider degenerate parabolic SPDEs 
of the form \eqref{d-p} and kinetic SPDEs of the form 
\eqref{diffusive-intro}. Therefore we should further 
clarify the assumptions regarding the stochastic 
components of these equations. 
For background information on stochastic analysis and SPDEs, 
including the framework of It\^{o} integration 
with respect to cylindrical Wiener processes, we 
refer to \cite{DaPrato:2014aa}. For properties of Bochner spaces 
like $L^p(\Omega;X)= L^p\bigl(\Omega,\cF,P;X\bigr)$, 
where $X$ is a Banach space, we refer to \cite{Hytonen:2016aa}.

We start by fixing a stochastic basis 
\begin{equation}\label{eq:stoch-basis}
	\cS=\bigl(\Omega, \cF,\seq{\cF_t}_{t\ge 0},P\bigr),
\end{equation}
along with an expectation operator $E$. This basis 
comprises a complete probability space $(\Omega,\cF,P)$ and a 
right-continuous, complete filtration $\seq{\cF_t}_{t\ge 0}$. Now, 
introduce two separable Hilbert spaces $\bbH$ and $\bbK$, 
each with its own norm and inner product structure. 
We define a cylindrical Wiener process $W(t)$ on 
the stochastic basis \eqref{eq:stoch-basis}, 
taking values in $\bbK$, as follows:
\begin{equation*}
	W(t) = \sum_{\ell=1}^{\infty} 
	W_{\ell}(t) e_\ell, \quad t\ge 0,
\end{equation*}
where $\seq{W_{\ell}}_{\ell=1}^\infty$ 
are independent standard $\R$-valued Wiener processes, 
and $\seq{e_\ell}_{\ell=1}^\infty$ forms 
an orthonormal basis for $\bbK$. Notably, the series 
defining $W(t)$ does not converge in $\bbK$, but rather 
in a larger Hilbert space $\tilde{\bbK}\supset \bbK$, 
equipped with a Hilbert-Schmidt embedding.

Denote by $L_2(\bbK,\bbH)$ the space of 
Hilbert-Schmidt operators 
mapping from $\bbK$ to $\bbH$. We say that a 
process $\Phi:\Omega \times [0, T] \to L_2(\bbK,\bbH)$ 
is progressively measurable if it remains measurable with 
respect to the $\sigma$-algebra $\cF_t \times \cB([0,t])$ for 
all $t \in [0, T]$. For a progressively measurable 
process $\Phi\in L^2(\Omega \times [0, T];L_2(\bbK,\bbH))$, we 
define the \textit{It{\^o} stochastic integral} $\cI(\Phi)$:
\begin{equation}\label{eq:stoch-int}
	\cI(\Phi)(t) = \int_0^t \Phi(s) 
	\, dW(s), \quad t\in [0,T].
\end{equation}
Here, $\cI(\Phi)$ is a progressively 
measurable process in $L^2(\Omega \times [0, T];\bbH)$,  
which is also a continuous $\bbH$-valued martingale. 
It satisfies the integral identity
$$
\bigl( \cI(\Phi)(t), h \bigr)_{\bbH} 
=\sum_{\ell=1}^\infty \int_0^t
\bigl(\Phi_\ell(s), h\bigr)_{\bbH} \, dW_{\ell}(s),
\quad \Phi_\ell(\cdot):=\Phi(\cdot)e_\ell,
\quad h\in \bbH,
$$
where each term in the series is interpreted 
in the It{\^o} sense. This series converges 
in $L^2(\Omega,\cF_t,P;C[0, t])$ 
for every $t \in [0, T]$. As we proceed, the 
representative examples we have in mind are
\begin{equation}\label{eq:H-space}
	\bbH=L^2(\R^d), \quad \bbH=L^2(w_Nd\mx),
\end{equation}
where $L^2(w_Nd\mx)$ is a weighted $L^2$ space 
over $\R^d$, see \eqref{eq:Lp-weight} below. 
In fact, since our results are local, we can 
safely fix $\bbH$ to be precisely $L^2(\R^d)$, a 
choice we adopt from Subsection \ref{subsec:compact-supp} 
until the end of the paper. 

Our relevant choice of $\Phi$ is the operator-valued 
mapping $B(u)$, the nonlinear noise amplitude 
in the degenerate parabolic SPDE \eqref{d-p}. 
For each $u\in L^2_{\loc}(\R^d)$,
we define $B(u)$ by its action on each $e_\ell$:
\begin{equation}\label{eq:b-ell-def}
	B(u)e_\ell:=b_\ell(t,\mx,u(\cdot)), 
	\qquad b_\ell\in C(\R_+;C^1(\R^d\times\R)),
	\qquad \ell\in \N.
\end{equation}
We then obtain
\begin{equation}\label{eq:dp-infinite-Wiener}
	B(u)\, dW(t)=\sum_{\ell\ge1} 
	b_\ell(t,\mx,u)\, dW_\ell(t).
\end{equation}
We assume that the sequence $\seq{b_\ell}_{\ell\ge 1}$ 
satisfies the following condition (which is 
similar to \cite{Gess:2018ab}): there 
is a sequence $\seq{\alpha_\ell}_{\ell\geq1}$ of positive 
numbers that satisfies $\sum_{\ell \geq 1}
\alpha_\ell^2 < \infty$ such that
\begin{equation}\label{eq:B-main-ass}
	\abs{b_\ell(t,\mx,0)} + \abs{\nabla_\mx b_\ell(t,\mx,\lambda)}
	+\abs{\pa_\lambda b_\ell(t,\mx, \lambda)}\leq \alpha_\ell, 
\end{equation}
for all $(t,\mx,\lambda)\in \R_+\times \R^d\times \R$.
This condition implies that 
\begin{align}
	\label{eq:dp-def-noise-B}
	& B^2=B^2(t,\mx,\lambda):=\sum_{\ell\geq 1}
	\left(b_\ell(t,\mx,\lambda)\right)^2
	\lesssim 1+\abs{\lambda}^2,
	\\ &
	\label{eq:dp-noise-bk-reg}
	\sum_{\ell\geq 1}\abs{b_\ell(t,\mx,\lambda_1)
	-b_\ell(t,\my,\lambda_2)}^2
	\lesssim 
	\abs{\mx-\my}^2+\abs{\lambda_1-\lambda_2}^2,
\end{align}
for all $t\in \R_+$, $\mx,\my \in \R^d$, 
and $\lambda_1,\lambda_2\in \R$. 
We define $\pa_\lambda B$ analogously to $B$, utilizing 
the sequence $\seq{\pa_\lambda b_\ell}_{\ell \ge 1}$. 
The sequence $\seq{\pa_\lambda b_\ell}_{\ell \ge 1}$ 
complies with the condition
\begin{equation}\label{eq:dp-def-noise-pa-B}
	(\pa_\lambda B)^2
	=(\pa_\lambda B)^2(t,\mx,\lambda):=
	\sum_{\ell\geq 1}
	\left(\pa_\lambda b_\ell(t,\mx,\lambda)\right)^2
	\lesssim 1.
\end{equation}

For the global stochastic integral under
\eqref{eq:dp-def-noise-B} and \eqref{eq:dp-noise-bk-reg}, take
$\bbH=L^2(w_Nd\mx)$, with the integrable weight in
\eqref{eq:weight-def} below. For $u\in\bbH$,
\begin{align*}
	\norm{B(u)}_{L_2(\bbK,\bbH)}^2
	&=\int_{\R^d}\sum_{\ell\ge1}
	\abs{b_\ell(t,\mx,u(\mx))}^2w_N(\mx)\,d\mx\\
	&\lesssim\norm{w_N}_{L^1}+\norm{u}_{\bbH}^2.
\end{align*}
The bound \eqref{eq:dp-noise-bk-reg} gives a Lipschitz map
$B:\bbH\to L_2(\bbK,\bbH)$. Hence, for every progressively
measurable $u\in L^2(\Omega\times[0,T];\bbH)$, the stochastic
integral $\cI(B(u))$ is a well-defined $\bbH$-valued martingale.
The global unweighted alternatives are discussed in
Remark \ref{rem:unweighted-noise}. After spatial localization,
for $\chi_x\in C_c^\infty(\R^d)$,
$$
\norm{\chi_xB(u)}_{L_2(\bbK,L^2(\R^d))}^2
\lesssim_{\chi_x}1+\norm{u}_{L^2(\supp\chi_x)}^2.
$$
The estimate for $\norm{\chi_xB(u)}_{L_2(\bbK,L^2(\R^d))}$ is used in
Subsection \ref{subsec:compact-supp} and thereafter.

\subsection{A priori estimates and the kinetic formulation}
\label{subsec:kinetic}
Let $u$ be an entropy solution of \eqref{d-p}. 
By applying the entropy inequality \eqref{eq:dp-entropy-ineq} 
with $S(u) = \abs{u}^p$ for $p \geq 2$, and utilizing 
\eqref{eq:dp-def-noise-B} along with a standard martingale 
argument, we deduce that
\begin{equation}\label{eq:u-in-Lp}
	u \in L^p(\Omega;L^\infty([0,T];L^p(w_N d\mx))),
	\quad \forall p\in [2,\infty),\,\, T>0,
\end{equation}
where $L^p(w_Nd\mx)$ denotes the weighted
 $L^p$ space of functions $v: \R^d \to \R$ such that
\begin{equation}\label{eq:Lp-weight}
	\int_{\R^d} |v|^p\, w_N \,d\mx < \infty.
\end{equation}
An example of a weight function is
\begin{equation}\label{eq:weight-def}
	w_N(\mx)=(1+|\mx|^2)^{-N}, \quad N > d/2.
\end{equation}
This function is integrable on 
$\R^d$ and satisfies, for all $i, j = 1, \ldots, d$,
\begin{equation}\label{eq:weight-prop}
	\abs{\partial_{x_i} w_N(\mx)}
	\lesssim \frac{w_N(\mx)}{1 + |x|} \leq w_N(\mx),
	\quad 
	\abs{\partial_{x_i x_j}^2 w_N(\mx)}
	\lesssim \frac{w_N(\mx)}{1 + |x|^2} \leq w_N(\mx).
\end{equation}
Any function $w\ge0$ that meets these 
conditions can be used as a weight function. 

\begin{remark}\label{rem:unweighted-noise}
Suppose $B(0) = 0$. In this scenario, the standard $L^p$ 
spaces provide a natural framework for addressing \eqref{d-p} 
(on $\R^d$). Specifically, we can 
omit the weight $w_N$ and establish that 
$u \in L^p(\Omega; L^\infty(0,T;L^p(\R^d)))$ 
for all $p \in [2, \infty)$, provided the 
initial condition satisfies $u_0 \in 
L^\infty_\omega\left(L^2_\mx \cap L^\infty_\mx\right)$. 
Alternatively, the weight $w_N$ can be omitted by 
imposing a stronger condition on $B^2$ as 
$\abs{\mx}\to \infty$, namely, 
\begin{equation}\label{eq:B2-xdep}
	B^2(t,\mx,\lambda)\leq (b(\mx))^2
	\left(1+\abs{\lambda}^2 \right),
	\quad 
	b \in \bigl(L^2 
	\cap L^\infty\bigr)(\R^d),
\end{equation}
for $\mx \in \R^d$ and $\lambda \in \R$.
However, if these assumptions do not hold, then
weighted $L^p$ spaces become a more appropriate choice 
for addressing various well-posedness questions. 
For further discussion on this point, see 
\cite[Remark 5.3]{Frid:2021us}.	
\end{remark}

\begin{remark}
The claim \eqref{eq:u-in-Lp} is derived from a Gronwall-type 
energy argument that relies on a standard estimate 
obtained from the Burkholder-Davis-Gundy (BDG) inequality. 
Let us consider the weight-free case where $\bbH = L^2(\R^d)$. 
Assume that the function $B$ is compactly supported in 
$\mx$, or $B$ must depend on $\mx$ and satisfy the 
bound \eqref{eq:B2-xdep}. Alternatively, $B(0) \equiv 0$, which 
eliminates the “1-term” on the right-hand 
side of \eqref{eq:dp-def-noise-B}.
Set
\begin{align*}
	\cM(t):=\sum_{\ell\ge1}\int_0^t\int_{\R^d}
	b_\ell(s,\mx,u)\varrho(u)
	\, d\mx\,dW_\ell(s),
\end{align*}
for any continuous function $\varrho:\R\to \R$ with 
$\abs{\varrho(\lambda)}\lesssim 1+ \abs{\lambda}$.
Here, $u$ is a solution such that all the norms appearing 
below are finite. Applying the BDG inequality 
then gives ($\forall q\ge 2$):
\begin{align*}
	& E\sup_{t\in [0,T]}\abs{\cM(t)}^{\frac{q}{2}}
	\\ \notag 
	& \quad \lesssim_T\frac12
	E\left( \esssup_{t\in [0,T]} 
	\norm{u(t)}_{L^2(\R^d)}^{q}\right)
	+\int_0^T E \left(\norm{u(t)}_{L^2(\R^d)}^{q}
	\right) \,dt+1<\infty,
\end{align*}
which combined with a Gronwall-type argument yields 
\eqref{eq:u-in-Lp} (with $w_N\equiv 1$).
\end{remark}

As mentioned in the introduction, we 
will instead of the entropy formulation 
rely on the more refined ``kinetic" 
interpretation of the entropy inequalities 
\cite{Chen:2003td,Debussche:2010fk,Gess:2018ab,Lions:1994qy}. 
Define the mapping $\chi: \R^2 \to \R$ by
$$
\chi(\lambda, u) = \En_{\lambda < u} - \En_{\lambda < 0},
$$
which is known as a $\chi$-function. 
The one-sided variants 
of $\chi$ are defined by $\chi_+(\lambda, u)
:=\En_{\lambda < u}$ and $\chi_-(\lambda, u)
:=-\En_{\lambda \ge u}$.  
Notice that $\chi(\lambda, u)$ can be written as 
$\chi(\lambda, u) = \chi_+(\lambda, u) - \En_{\lambda < 0}$ 
or equivalently $\chi(\lambda, u) = \chi_-(\lambda, u)
+\En_{\lambda \ge 0}$. Moreover, the identity 
$\chi_+(\lambda, u) + \chi_-(\lambda, u) 
= \sgn(u - \lambda)$ holds. 
Another relevant function is
$\frac{1}{2} \left(\sgn(u-\lambda)
-\sgn(\lambda)\right)$, which can 
also be expressed in terms of $\chi_\pm$. 
The function
$$
h = h(\omega, t, \mx) := \chi_+(\lambda, u(\omega, t, \mx)) 
= \En_{\lambda < u(\omega, t, \mx)},
$$
where $u$ is an entropy solution of \eqref{d-p}, 
satisfies the kinetic equation \eqref{diffusive-intro}.
Similar kinetic equations hold 
for $\chi_-$ and $\chi$. 

The nonnegative, measure-valued 
random variable $\zeta$ on the right 
side of \eqref{diffusive-intro} is given by
\begin{equation}\label{eq:zeta-def}
	\zeta = m + n, \quad
	n = \delta(\lambda - u) \sum_{i, j = 1}^d 
	a_{ij}(u) \partial_{x_i} u \partial_{x_j} u,
\end{equation}
where $n$ represents the parabolic 
dissipation measure, and $m$ is referred to 
as the kinetic (entropy defect) measure. 
For further details, 
see \cite[Definition 2.2]{Chen:2003td} and 
\cite[Definition 2.1]{Gess:2018ab}, and 
the discussion in the following subsection.

\begin{definition}\label{def:kinetic-sol}
Given initial data $u_0 \in 
L^\infty\left(\Omega, \mathcal{F}_0; L^\infty(\R^d)\right)$, set 
$h_0 := \En_{\lambda < u_0}$. A measurable function 
$u: \Omega \times [0, T] \times \R^d \to \R$ is called 
a \textit{kinetic solution} of the Cauchy problem for 
the degenerate parabolic equation \eqref{d-p} if $u$ is 
a predictable, $L^2(w_N d\mx)$-valued process 
satisfying \eqref{eq:u-in-Lp} and 
there is a kinetic measure $m$ such that 
the subgraph $h := \En_{\lambda < u}$ obeys the kinetic 
equation \eqref{diffusive-intro}, along with 
the (weak) satisfaction of the 
initial condition $h(0)=h_0$.	
\end{definition}

Note that if $u$ is predictable, then $h$ 
is also predictable. Incorporating the weight 
$w_N$ from \eqref{eq:weight-def} 
into \eqref{eq:zeta-def}, we define 
$\zeta_N := w_N \zeta 
= w_N m + w_N n =: m_N + n_N$.  
It can be shown that
\begin{equation}\label{eq:zetaN-pbound}
	E\int_{[0,T]\times \R^d\times \R} 
	\abs{\lambda}^p \, \zeta_N(dt,d\mx,d\lambda)
	\le C, \qquad p\in [0,\infty),
\end{equation}
where $C$ depends on $T,N$ and $\norm{u}_{L^{p+2}\left(\Omega;
L^\infty\left(0,T;L^{p+2}(w_Nd\mx)\right)\right)}$.
Thus, the random variable
$\omega \mapsto \int_{[0, T] \times \R^d \times \R} 
\abs{\lambda}^p \, \zeta_N(dt, d\mx, d\lambda)$ 
belongs to $L^1(\Omega)$. 
One can improve this to $L^q(\Omega)$ for any 
finite $q \ge 1$. This follows from the bound
\begin{equation}\label{eq:moments}
	E\Bigl(\esssup_{t \in [0, T]} 
	\norm{u(t)}_{L^{p+2}(w_Nd\mx)}^r\Bigr) 
	+E\Bigl(\int_{[0, T] \times \R^d \times \R} 
	\abs{\lambda}^p \, \zeta_N(dt, d\mx, d\lambda)
	\Bigr)^{\frac{r}{p+2}} \le C,
\end{equation}
which holds provided  $u_0$ satisfies
$E \bigl(\norm{u_0}_{L^{p+2}(w_Nd\mx)}^r\bigr) < \infty$ 
for $r > p + 2$. The case $r = p + 2$ is 
covered by the definition of kinetic solution. 
The constant $C$ depends solely on $r,T,N$ (and $u_0$).
This bound follows from \eqref{eq:dp-entropy-ineq}, 
see, e.g., \cite[Remark 5.8]{Frid:2021us}.

The well-posedness of kinetic solutions 
(for the Cauchy problem) of \eqref{d-p} is 
established in \cite{Chen:2003td,Gess:2018ab} (see 
\cite{Bendahmane:2004go} for entropy solutions), 
which also covers the case of purely $L^1$ initial data. 
The deterministic setting in 
\cite{Bendahmane:2004go,Chen:2003td} 
applies to $\R^d$, whereas the 
stochastic framework of \cite{Gess:2018ab} 
is restricted to the torus $\mathbb{T}^d$. 
However, by combining the arguments from 
\cite{Gess:2018ab} with those in \cite[Section 5]{Frid:2021us}, 
while keeping in mind \eqref{eq:weight-prop}, the well-posedness 
of kinetic solutions $u$ (with high 
integrability) can be extended to $\R^d$ using the 
weighted $L^p$ framework described above. 
The first-order case ($a \equiv 0$ in \eqref{d-p}) was 
initially studied in \cite{Feng:2008ul} (on $\R^d$ 
with entropy solutions) and \cite{Debussche:2010fk} 
(on $\bbT^d$ with kinetic solutions).

\medskip

The noise-related term in 
\eqref{diffusive-intro} can 
be formally interpreted as
$$
\sum_{\ell \ge 1} b_\ell(\lambda)
\pa_\lambda h \, \dot{W}_\ell
= B \pa_\lambda h \, \dot{W}
= \pa_t \cI(B \pa_\lambda h),
$$
where $\pa_t \cI(B \pa_\lambda h)$ denotes the 
distributional time derivative 
of an It\^{o} integral. 
Since $\pa_\lambda h=-\delta(\lambda-u)$ is a nonpositive
finite measure of total mass $-1$ in the kinetic variable, we
interpret $\cI(B\pa_\lambda h)$ through the duality formula
\begin{equation}\label{eq:I-B-pah}
	\innb{\cI(B \pa_\lambda h),
	\varrho}_{\cD'_\lambda, \cD_\lambda}
	= -\int_0^t (B \varrho)(u) \,dW(s)
	= -\cI\bigl((B\varrho)(u)\bigr),
	\quad \varrho \in \cD_\lambda,
\end{equation}
so that $\innb{\pa_t \cI(B \pa_\lambda h),
\varrho}_{\cD'_\lambda, \cD_\lambda}
=-\pa_t\cI\bigl((B\varrho)(u)\bigr)$ 
becomes the distributional 
time derivative of the It\^{o} integral of 
the process $\Phi=(B \varrho)(u)=B(u)\varrho(u)$
(see \eqref{eq:stoch-int}). 
This $\Phi$ is defined using \eqref{eq:b-ell-def}, with 
each $b_\ell(u)$ multiplied by the 
scalar function $\varrho(u)$. Furthermore, we define 
the term $B^2\pa_\lambda h$ by
\begin{equation}\label{eq:zeta-B2-pah}
	\innb{B^2\pa_\lambda h,
	\varrho}_{\cD'_\lambda,\cD_\lambda}
	=-(B^2\varrho)(u)=-B^2(u)\varrho(u),
	\quad \varrho \in \cD_\lambda.
\end{equation}
The regularity analysis will further depend on 
expressing the stochastic term $\cI(B\pa_\lambda h)$ in 
divergence form as follows:
\begin{equation}\label{eq:I-B-pah-div}
	\cI(B \pa_\lambda h) 
	= \pa_\lambda \cI (B h) - \cI (\pa_\lambda B h)
	\quad \text{in $\cDp_\lambda$},
\end{equation}
where $\pa_\lambda B$ is well-defined and satisfies 
\eqref{eq:dp-def-noise-pa-B}.

Using the interpretations \eqref{eq:I-B-pah}, 
\eqref{eq:zeta-B2-pah}, and 
\eqref{eq:I-B-pah-div}, we reformulate 
the kinetic equation \eqref{diffusive-intro} as follows:
\begin{equation}\label{diffusive-intro-ver2}
	\begin{split}
		& \pa_t h+\sum_{i=1}^d\pa_{x_i} 
		\bigl( f_i(\lambda) h \bigr)
		-\sum_{i,j=1}^d\pa_{x_i x_j}^2
		\bigl(a_{ij}(\lambda) h \bigr)
		\\ & \qquad\quad
		=-\pa_t \cI(B\pa_\lambda h)
		+\pa_\lambda \bigl(\zeta
		+\tfrac12 B^2\pa_\lambda h\bigr)
		\\ & \qquad\quad
		=-\pa_\lambda\pa_t \cI (B h)
		+\pa_t \cI (\pa_\lambda B h)
		+\pa_\lambda \bigl(\zeta
		+\tfrac12 B^2\pa_\lambda h\bigr)
		\quad 
		\text{in $\cD'_{t,\mx,\lambda}$, a.s.}
	\end{split}
\end{equation}
This equation, which emphasizes the representation 
of the stochastic forcing term in 
\eqref{diffusive-intro} as the distributional 
time derivative of the It\^{o} integral, forms 
our starting point for deriving the regularity result. 
This interpretation was also used in \cite{Gess:2018ab} and,
for a different purpose, in our prior work \cite{Erceg:2023ab}, where
we drew inspiration from the treatment of stochastic forcing in the
Navier--Stokes setting in \cite{Langa:2003aa}.

Let $\cI(\Phi)$ be the It\^{o} 
integral of $\Phi$ as defined in \eqref{eq:stoch-int}. 
Note that as the time derivative map $\pa_t$ 
is continuous from  $C$ (and thus $L^\infty$) 
to $W^{-1,\infty}$ and $\cI(\Phi)$ 
belongs to $L^2(\Omega,\cF_t,P;C([0,t];\bbH))$ 
for every time $t\in [0,T]$, it follows that
\begin{equation}\label{eq:stoch-int-deriv}
	\pa_t \cI(\Phi)\in 
	L^2(\Omega,\cF_t,P;W^{-1,\infty}([0,t];\bbH)).
\end{equation}
We will require more refined bounds 
on $\partial_t I(\Phi)$ beyond the one 
given by \eqref{eq:stoch-int-deriv}.
In particular, by \cite[Lemma 2.1]{Flandoli:1995aa}
we have
$$
\cI(\Phi) \in L^q\bigl(\Omega; 
W^{\nu,q}([0,T];\bbH)\bigr), 
\quad q\ge 2, \,\, 
\nu < \frac{1}{2},
$$
and there exists a constant $C(q,\nu) > 0$ such that
\begin{equation}\label{eq:frac-Sob-Ito-ver0}
	E\left[ \norm{\cI(\Phi)}_{W^{\nu,q}([0,T];\bbH)}^q\right] 
	\leq C(q,\nu) E\left[
	\int_0^T \norm{\Phi(t)}_{L_2(\bbK,\bbH)}^q \,dt \right].
\end{equation}
Here, $W^{\nu,q}([0,T];\bbH)$ denotes a vector-valued 
Sobolev-Slobodetskii space, as defined in 
\cite[Definition 2.5.16]{Hytonen:2016aa}. However, in 
Section \ref{sec:regularity}, we will transition 
to utilizing vector-valued 
Bessel potential spaces $H^{\nu,q}(\R;\bbH)$ 
\cite[Definition 5.6.2]{Hytonen:2016aa}. 
Therefore, let us justify 
the application of \eqref{eq:frac-Sob-Ito-ver0} within 
this framework. For a fixed $\vartheta\in C^\infty_c((0,T))$, 
consider $\vartheta \cI(\Phi)$. 
The zero-extension of $\vartheta \cI(\Phi)$ with 
respect to $t$---from $[0,T]$ to 
the entire $\R$---is denoted by
$\widetilde{\vartheta\cI(\Phi)}$ and adheres to the
same estimate as mentioned previously 
\eqref{eq:frac-Sob-Ito-ver0} 
(see, e.g., \cite[Section 5]{Di-Nezza:2012wi}), that is,
\begin{equation*}
	E\left[ \norm{\widetilde{\vartheta\cI(\Phi)}}
	_{W^{\nu,q}(\R;\bbH)}^q\right]
	\lesssim_\vartheta
	E\left[ \norm{\cI(\Phi)}_{W^{\nu,q}([0,T];\bbH)}^q\right].
\end{equation*}  
For $0<\nu<1/2$, choose $\nu'\in(\nu,1/2)$ and apply
\eqref{eq:frac-Sob-Ito-ver0} with $\nu'$.
The embedding
$W^{\nu',q}(\R;\bbH)\hookrightarrow H^{\nu,q}(\R;\bbH)$
from Remark \ref{rem:sobolev_intro} then gives
\begin{equation}\label{eq:frac-Sob-Ito-ver0-H}
	E\left[ \norm{\widetilde{\vartheta\cI(\Phi)}}
	_{H^{\nu,q}(\R;\bbH)}^q\right]
	\lesssim_{q,\nu,\vartheta} E\left[
	\int_0^T \norm{\Phi(t)}_{L_2(\bbK,\bbH)}^q \,dt \right].
\end{equation}

\subsection{The chain rule property and consequences}\label{sec:chain-rule}
Our analysis employs the kinetic solution concept 
from \cite{Chen:2003td}, which hinges on the entropy 
inequalities \eqref{eq:dp-entropy-ineq} 
and incorporates a chain rule property 
\cite[Definition 2.1 (ii)]{Chen:2003td}. 
By examining \eqref{eq:dp-entropy-ineq} with 
the entropy $S(\lambda) = \lambda^2/2$, the 
right-hand side provides a degree of parabolic 
regularity, implying that 
$\sigma(u)\nabla_{\!\mx} u\in 
L^2(\Omega;L^2(w_Nd\mx))$ (see 
\eqref{eq:zetaN-pbound} with $p=0$), 
or $\sigma(u)\nabla_{\!\mx} u\in 
L^2(\Omega;L^2_{\loc}(\R_+^{d+1}))$. 
This regularity is a natural requirement 
within the entropy solution 
framework \cite{Volpert:1969fk,Chen:2003td}. Strictly speaking, 
the parabolic dissipation term 
$S''(u)\abs{\sigma(u)\nabla_{\!\mx} u}^2$ 
on the right-hand side of \eqref{eq:dp-entropy-ineq} 
must be interpreted through the lens of the 
chain rule property described 
in \cite{Chen:2003td}. To that end, we introduce 
the (matrix-valued) functions
$$
\Sigma(\lambda)
= \int_0^\lambda \sigma(w) \, dw,
\quad
\Sigma^{\psi}(\lambda)
= \int_0^\lambda 
\sqrt{\psi(w)}\sigma(w) \,dw,
$$
for $0\le \psi\in C(\R)$, and replace
the parabolic dissipation term on the right-hand side of
\eqref{eq:dp-entropy-ineq} by
$-\absb{\Div_{\!\mx} \Sigma^{S''}(u)}^2$, 
where the divergence operator is applied row-wise.  
The chain rule property \cite[Definition 2.1 (ii)]{Chen:2003td}
requires that:
\begin{equation}\label{eq:chain-rule}
	\Div_{\!\mx} \Sigma^{\psi}(u) 
	= \sqrt{\psi(u)} \Div_{\!\mx} \Sigma(u),
	\quad \forall \psi\in C(\R), \, \, \psi\ge 0.
\end{equation}
This property, when combined with the 
entropy inequalities \eqref{eq:dp-entropy-ineq}, 
not only ensures uniqueness but, 
as we will demonstrate, 
also offers new possibilities for proving velocity 
averaging results for second-order equations. 
This approach contrasts with related works 
(discussed before) that have not previously 
exploited this property. 

The chain rule \eqref{eq:chain-rule} also expresses the entropy
diffusion operator in first-order divergence form:
\begin{equation}\label{regularity}
	\begin{split}
		&\Div_{\!\mx}\Div_{\!\mx}R_S(u)
		=\Div_{\!\mx}\bigl(S'(u)\sigma(u)V\bigr),\\
		&V:=\Div_{\!\mx}\Sigma(u)
		\in L^2\bigl(\Omega;L^2_{\loc}(\R_+^{d+1})^d\bigr).
	\end{split}
\end{equation}
Indeed, apply \eqref{eq:chain-rule} in row $k$ with the continuous
weights $(S'\sigma_{ik})_+^2$ and $(S'\sigma_{ik})_-^2$, subtract,
and sum over $k$, using $R_S'=S'a$ and $a=\sigma^2$.
For $S(r)=r$, $R_S'=a$, so this is the diffusion identity in
\eqref{d-p}.
It is worth noting that the chain rule 
property \eqref{eq:chain-rule} is automatically 
satisfied in the isotropic case (where $a(\cdot)$ 
is a scalar or diagonal matrix) and, therefore, 
does not need to be explicitly 
included in the solution concept.

As mentioned before, the primary approach 
to address \eqref{d-p} is through the kinetic 
formulation, see \eqref{diffusive-intro} and 
\eqref{diffusive-intro-ver2}. 
The measure-valued random variable 
$\zeta=\zeta(t,\mx,\lambda)$ 
in \eqref{diffusive-intro-ver2}, see 
\eqref{eq:zeta-def} and \eqref{eq:zetaN-pbound}, 
can be more precisely expressed as $\zeta = n+m$, 
where the kinetic entropy defect $m$ is a 
non-negative measure-valued random variable, and
$$
n:=\delta(\lambda-u)
\abs{\Div_{\!\mx} \Sigma(u)}^2.
$$
Using duality and \eqref{eq:chain-rule}, 
we can define the 
parabolic dissipation measure:
$$
n^{\psi}:=\actionb{n}{\psi} 
= \abs{\Div_{\!\mx} \Sigma^{\psi}(u)}^2
\in L^1(\Omega\times K), 
\, \, K\subset\subset \R_+^{d+1},
\,\, 0 \leq \psi \in C_c(\R).
$$

Finally, using \eqref{eq:chain-rule} and \eqref{square-root} 
we can reformulate \eqref{diffusive-intro-ver2} 
as a first-order equation in $\cD'_{t,\mx,\lambda}$, a.s.:
\begin{equation}\label{diffusive-ver2}
	\begin{split}
		&\pa_t h + \Div_{\!\mx} \bigl(f(\lambda) h\bigr) 
		- \Div_{\!\mx} \bigl(\sigma(\lambda)\beta\bigr) 
		\\ & \qquad 
		=-\pa_\lambda\pa_t \cI (B h)
		+\pa_t \cI (\pa_\lambda B h)
		+\pa_\lambda \bigl(\zeta+\tfrac12 B^2\pa_\lambda h\bigr),
	\end{split}
\end{equation}
where, informally, the vector $\beta$ can be interpreted 
as $\beta = \sigma(\lambda)\nabla_{\!x} h$, with 
$\nabla_{\!\mx}h
=\delta(\lambda-u)\nabla_{\!\mx}u$.
However, rigorously, via  \eqref{eq:chain-rule},
\begin{equation}\label{eq:beta-def}
	\begin{split}
		&\beta=\delta(\lambda-u)V,
		\quad V:=\Div_{\!\mx} \Sigma(u) 
		\in L^2\bigl(\Omega;L^2_{\loc}(\R_+^{d+1})^d\bigr),
		\\ & 
		\quad \text{so that for $P$-a.e.~$\omega$, 
		$\beta(\omega,\cdot)\in 
		L^2_{w\star}(K;\cM(\R)^d)$,
		$K \subset\subset \R_+^{d+1}$}.
	\end{split}	
\end{equation}

\begin{remark}
To illustrate the generality of the equations, 
assume additionally in this remark
that $\sigma\in C^1$, and note that
$\beta=-V\pa_\lambda h$.
Applying a ``differentiation by parts" in $\lambda$, 
the equation \eqref{diffusive-ver2} can 
be alternatively formulated as:
\begin{align*}
	&\pa_t h
	+ \Div_{\!\mx}
	\Bigl(
	\bigl[f(\lambda)-V(t,\mx)\sigma'(\lambda)\bigr]
	h\Bigr)
	\\ &  \qquad 
	=-\pa_\lambda\pa_t \cI (B h)
	+\pa_t \cI (\pa_\lambda B h)
	+\pa_\lambda \bigl(\zeta+\tfrac12 B^2\pa_\lambda h\bigr)
	-\Div_{\!\mx}\pa_\lambda
	\bigl(V(t,\mx)\sigma(\lambda)h\bigr).
\end{align*}
Without the stochastic components ($B\equiv 0$), this 
equation is a special case of the 
more general heterogeneous equation
\begin{align*}
	&\pa_t h + \Div_{\!\mx} \bigl( f(t,\mx,\lambda) h\bigr)
	=\pa_{\lambda}\zeta+\Div_{\!x}\pa_{\lambda}g,
\end{align*}
where $g\in L^q$ for some $q>1$, 
with a full spatial derivative in the source term. 
In the $(t,\mx)$-independent drift case, this corresponds 
to the \textit{limiting case} of velocity averaging for 
compactness, as described 
in \cite{PerthameSouganidis:98}.
\end{remark}

\subsection{Final form of the kinetic equations}
\label{subsec:compact-supp}
In the next section, we use both the second-order kinetic 
equation \eqref{diffusive-intro-ver2} and the 
corresponding first-order one \eqref{diffusive-ver2} 
to establish the regularity of kinetic solutions. 
For this analysis, we assume that the solution $h$, the 
auxiliary variable $\beta$, the 
stochastic integral terms $\cI(Bh)$, $\cI(\pa_\lambda Bh)$, 
and the ``source" $\zeta$ are compactly 
supported in $t$, $\mx$ and $\lambda$, for each $\omega$. 

Let us explain why making such an assumption 
involves no loss of generality. Due to the linearity of 
the kinetic equation, we can multiply it by a cutoff 
function $\varphi$, resulting in an equation of 
essentially the same type, but with solutions $\varphi h$ 
that are compactly supported in $t$, $\mx$, 
and $\lambda$. To be a bit more precise, let $\varphi \in 
C^\infty_c(\R_+^{d+1} \times \R)$ be a fixed cutoff
function of the form
$$
\varphi(t,\mx,\lambda)
= \vartheta(t) \phi(\mx)\varrho(\lambda),
\quad
\vartheta \in C^\infty_c((0,T)),\,\,
\phi \in C^\infty_c(\R^d),\,\,
\varrho \in C^\infty_c(\R).
$$
Fix the compact velocity interval $I$ used for averaging.
Choose a bounded open interval $U$ with $I\Subset U\Subset\R$
and take $\varrho=1$ on a neighbourhood of $\overline U$.
Choose $\vartheta\phi=1$ on the compact space-time region under
consideration. These choices are independent of the particular
averaging weight $\rho\in C_c^1(I)$.
We define
$$
\tilde{h} := h \varphi, 
\quad 
\tilde{\beta} := \beta \varphi, 
\quad 
\tilde{\zeta} := \zeta \varphi.
$$ 
As a result, $\tilde h$, $\tilde\beta$, and $\tilde\zeta$
are compactly supported. Global integrability of
$\tilde h,\tilde\beta,\tilde\zeta$ follows from
\eqref{eq:zetaN-pbound} and \eqref{eq:moments}
(see Lemma \ref{lem:localization} below).
Multiplication of the kinetic equation
by $\varphi$ gives the correction
$\tfrac12\varphi B^2\pa_\lambda h$, with one factor of $\varphi$;
this measure is interpreted as in \eqref{eq:zeta-B2-pah}.

The localized kinetic identities are understood in
$\cD'(\R^{d+1}\times U)$, where $\pa_\lambda\varphi=0$.
For the deterministic temporal cutoff $\vartheta$, the identity
$$
\vartheta\,\pa_t\cI(\Phi)=\pa_t\cI(\vartheta\Phi)
$$
holds in the sense of distributions. Set
$$
Y_0:=\cI(B\tilde h),\qquad
Y_1:=\cI((\pa_\lambda B)\tilde h).
$$
Since $\varphi=\vartheta\phi\varrho$ and
$\pa_\lambda\varrho=0$ on $U$, the stochastic terms in
\eqref{diffusive-intro-ver2} and \eqref{diffusive-ver2} satisfy
\begin{align*}
	&\varphi\bigl(-\pa_\lambda\pa_t\cI(Bh)
	+\pa_t\cI((\pa_\lambda B)h)\bigr)
	\\ &\qquad=-\pa_\lambda\pa_tY_0+\pa_tY_1.
\end{align*}
Applying the product rule to the remaining terms of
\eqref{diffusive-intro-ver2} and \eqref{diffusive-ver2} gives:
\begin{equation}\label{eq:kinetic-compact-supp-1}
	\begin{split}
		&\pa_t \tilde h+\sum_{i=1}^d \pa_{x_i}
		\bigl(f_i\tilde h\bigr)
		-\sum\limits_{i,j=1}^d \pa_{x_i x_j}^2
		\bigl( a_{ij} \tilde h\bigr)
		\\ &\qquad
		=-\pa_\lambda\pa_tY_0+\pa_tY_1
		\\ &\qquad\qquad
		+\pa_\lambda\bigl(\tilde\zeta
		+\tfrac12\varphi B^2\pa_\lambda h\bigr)+S_1,
	\end{split}
\end{equation}
and
\begin{equation}\label{eq:kinetic-compact-supp-2}
	\begin{split}
		&\pa_t \tilde h+\sum\limits_{i=1}^d \pa_{x_i}
		\bigl(f_i\tilde h\bigr)
		-\sum_{i,k=1}^d \pa_{x_i}
		\Bigl(\sigma_{ik}\tilde\beta_k\Bigr)
		\\ &\qquad
		=-\pa_\lambda\pa_tY_0+\pa_tY_1
		\\ &\qquad\qquad
		+\pa_\lambda\bigl(\tilde\zeta
		+\tfrac12\varphi B^2\pa_\lambda h\bigr)+S_2,
	\end{split}
\end{equation}
where
\begin{align*}
	S_1&:=h\pa_t\varphi+\sum_{i=1}^d f_ih\pa_{x_i}\varphi
	\\ &\qquad
	-\sum_{i,j=1}^d\pa_{x_i}
	\bigl(a_{ij}h\pa_{x_j}\varphi\bigr)
	-\sum_{i,j=1}^d\pa_{x_j}(a_{ij}h)\pa_{x_i}\varphi,
\end{align*}
and
\begin{align*}
	S_2&:=h\pa_t\varphi+\sum_{i=1}^d f_ih\pa_{x_i}\varphi
	-\sum_{i,k=1}^d\sigma_{ik}\beta_k\pa_{x_i}\varphi.
\end{align*}

We now complete the localization by putting
\eqref{eq:kinetic-compact-supp-1} and
\eqref{eq:kinetic-compact-supp-2} into their final form and
establishing the bounds needed in the averaging argument.

Write $z=(t,\mx)$ and $dz=dt\,d\mx$.
The chain rule \eqref{eq:chain-rule}, together with
\eqref{eq:beta-def}, gives the measure-valued diffusion flux
$$
\mathcal J:=a(\lambda)\nabla_{\!\mx}h
=\sigma(\lambda)\beta
=\sigma(u)V\,\delta_{u(z)}(d\lambda)\,dz.
$$
Using $\mathcal J=a\nabla_{\!\mx}h$, the product rule gives
$S_1=\cC_\varphi$, where
$$
\cC_\varphi
:=h\bigl(\pa_t\varphi+f\cdot\nabla_{\!\mx}\varphi
-a:D_{\mx}^2\varphi\bigr)
-2\nabla_{\!\mx}\varphi\cdot\mathcal J.
$$
Here and below integrable functions are identified with their
associated measures. Similarly,
$S_2=h(\pa_t\varphi+f\cdot\nabla_{\!\mx}\varphi)
-\nabla_{\!\mx}\varphi\cdot\mathcal J$.
Consequently,
$$
S_1=S_2-\Div_{\!\mx}
\bigl(h\,a(\lambda)\nabla_{\!\mx}\varphi\bigr).
$$
Define the corrected localized flux, as a velocity measure for
almost every $z$, by
$$
\beta_{\loc}
:=\tilde\beta+
h\,\sigma(\lambda)\nabla_{\!\mx}\varphi\,d\lambda.
$$
Then
$$
\sigma(\lambda)\beta_{\loc}
=a(\lambda)\nabla_{\!\mx}\tilde h.
$$
Consequently, replacing $\tilde\beta$ by $\beta_{\loc}$ in
\eqref{eq:kinetic-compact-supp-2} replaces $S_2$ by $S_1$.
Both equations therefore contain the same remainder
$\pa_\lambda(\tilde\zeta+\tfrac12\varphi B^2\pa_\lambda h)
+\cC_\varphi$.

For compact time support, choose
$\vartheta_1\in C_c^\infty((0,T))$ equal to one on a neighbourhood
of $\supp\vartheta$. With $W$ as in
\eqref{eq:dp-infinite-Wiener}, we use the cutoff stochastic integral
$$
\cI_{\loc}(\Phi)(t)
:=\vartheta_1(t)\int_0^t\Phi(s)\,dW_s,
$$
extended by zero outside $(0,T)$.

\begin{lemma}\label{lem:localization}
With the cutoffs and localized quantities defined above, there is
a finite signed random measure $\zeta_{\loc}$ such that
\eqref{diffusive1-final-ver} and \eqref{diffusive2-final-ver} 
below hold almost surely in 
$\cD'(\R^{d+1}\times U)$ after the replacements
$$
(h,\beta,\zeta,\cI)
=(\tilde h,\beta_{\loc},\zeta_{\loc},\cI_{\loc}).
$$
The function $\tilde h$ is bounded, and, for every $1\leq q<\infty$,
$$
E\left[
\norm{\beta_{\loc}}_{L^2_{w\star}(\R^{d+1};\cM(\R)^d)}^q
+\norm{\zeta_{\loc}}_{\cM(\R^{d+2})}^q
\right]<\infty.
$$
The quantities $\tilde h,\beta_{\loc},\zeta_{\loc}$ and the stochastic integrals
$\cI_{\loc}(B\tilde h)$,
$\cI_{\loc}((\pa_\lambda B)\tilde h)$ have support in a fixed
deterministic compact set in $(t,\mx,\lambda)$.
In particular, all finite velocity moments of the total variation
of $\zeta_{\loc}$ have all finite probability moments.
These stochastic integrals satisfy \eqref{eq:frac-Sob-Ito-ver1}
and \eqref{eq:BDG-Phi} below, with $h=\tilde h$ and
$\cI=\cI_{\loc}$. The constants are uniform in $\lambda$ and the
subsequent Fourier parameters, and may depend on the fixed cutoffs.
\end{lemma}

\begin{proof}
With $Y_0,Y_1$ defined above, the integrands $B\tilde h$ and
$(\pa_\lambda B)\tilde h$ vanish wherever $\vartheta_1\ne1$.
Hence
$$
\pa_t(\vartheta_1Y_j)=\pa_tY_j+\vartheta_1'Y_j,
\qquad j=0,1.
$$
Using the formulas for $\pa_t(\vartheta_1Y_j)$, $j=0,1$, and
$\sigma\beta_{\loc}=a\nabla_{\!\mx}\tilde h$ in
\eqref{eq:kinetic-compact-supp-1} and
\eqref{eq:kinetic-compact-supp-2}, we obtain
the common remainder $\pa_\lambda M+S$, where we choose the compact
extensions
\begin{align*}
	M&:=\varphi\bigl(\zeta+\tfrac12B^2\pa_\lambda h\bigr)
	+\vartheta_1'Y_0,
	\qquad
	S:=\cC_\varphi-\vartheta_1'Y_1.
\end{align*}
Here $\pa_\lambda\varphi=\vartheta\phi\pa_\lambda\varrho=0$ on $U$.

Let $K\subset\subset(0,T)\times\R^d$ contain
$\supp\vartheta_1\times\supp\phi$, and let $J_0$ be a compact
interval containing $\supp\varrho$.
The bounds \eqref{eq:zetaN-pbound} and \eqref{eq:moments}
give all finite probability moments of $\zeta(K\times\R)$,
because $w_N$ is bounded below on $K$.
With $V$ as in \eqref{eq:beta-def}, the inequality $n\leq\zeta$ gives
$\norm{V}_{L^2(K)}^2\leq\zeta(K\times\R)$, and
$$
\norm{\mathcal J}_{\cM(K\times\R)}
\leq\norm{\sigma}_{L^\infty}
\meas(K)^{1/2}\norm{V}_{L^2(K)}.
$$
Also, $\abs{\pa_\lambda h}=\delta_{\lambda=u}$, and $f,a,\sigma$
are bounded on $J_0$. The bounds \eqref{eq:dp-def-noise-B} and
\eqref{eq:dp-def-noise-pa-B} give uniform bounds for $B^2$ and
$(\pa_\lambda B)^2$ for $\lambda\in J_0$.
The BDG inequality then controls $\vartheta_1'Y_j$ in
$L^q(\Omega;L^1(K\times J_0))$ for every finite $q\geq1$.
It follows that $\supp M\cup\supp S\subset K\times J_0$ and
$E[\norm{M}_{\cM}^q+\norm{S}_{\cM}^q]<\infty$
for every $1\leq q<\infty$.
The definition of $\beta_{\loc}$ and the boundedness of $\sigma$ give
$$
\norm{\beta_{\loc}}_{L^2_{w\star}(\R^{d+1};\cM(\R)^d)}
\lesssim_{\varphi,\sigma}\norm{V}_{L^2(K)}
+\norm{h}_{L^\infty(K\times J_0)}.
$$

Define
$$
P(dz,d\lambda)
:=\varrho(\lambda)S\bigl(dz,(-\infty,\lambda]\bigr)\,d\lambda,
\qquad
\zeta_{\loc}:=M+P.
$$
For each $\lambda$, $S(dz,(-\infty,\lambda])$ is a measure in $z$
(no density of $S$ in space-time is assumed).
Equivalently, for a smooth test function $g$,
$$
\action{P}{g}
=\int\left(\int_\kappa^\infty
\varrho(\lambda)g(z,\lambda)\,d\lambda\right)S(dz,d\kappa).
$$
Hence
$$
\norm{P}_{\cM}
\leq\norm{\varrho}_{L^1(\R)}\norm{S}_{\cM},
\qquad
\supp P\subset K\times\supp\varrho.
$$
For test functions supported in $\R^{d+1}\times U$, Fubini's theorem and
$\varrho=1$ on $U$ give $\pa_\lambda P=S$.
Thus $\pa_\lambda\zeta_{\loc}=\pa_\lambda M+S$ on $U$.
Substituting $\pa_\lambda\zeta_{\loc}=\pa_\lambda M+S$ into
\eqref{eq:kinetic-compact-supp-1} and
\eqref{eq:kinetic-compact-supp-2}, with $\beta_{\loc}$ and
$\cI_{\loc}$, gives \eqref{diffusive1-final-ver} and
\eqref{diffusive2-final-ver} below on $U$.

The estimates \eqref{eq:frac-Sob-Ito-ver1} and \eqref{eq:BDG-Phi}
below follow from \eqref{eq:frac-Sob-Ito-ver0-H} and the BDG inequality,
with outer cutoff $\vartheta_1$, $\bbH=L^2(\R^d)$, and
$\Phi=B\tilde h$ or $\Phi=(\pa_\lambda B)\tilde h$.
Both choices satisfy
$\norm{\Phi(t,\cdot,\lambda)}_{L_2(\bbK,L^2(\R^d))}^2
\lesssim\norm{\tilde h(t,\cdot,\lambda)}_{L^2(\R^d)}^2$
by \eqref{eq:dp-def-noise-B} and \eqref{eq:dp-def-noise-pa-B}
(see the calculation below after \eqref{eq:frac-Sob-Ito-ver1}).
Moreover, $\supp\cI_{\loc}(\Phi)\subset
\supp\vartheta_1\times\supp\phi\times\supp\varrho$.
Since $\vartheta_1$ is independent of $\lambda$, the
distributional identity \eqref{eq:I-B-pah-div} also holds with
$\cI_{\loc}$ and $\tilde h$. This justifies the compact stochastic
notation in \eqref{diffusive1-final-ver} and \eqref{diffusive2-final-ver}
below.
\end{proof}

\begin{remark}
\label{rem:localized-equations}
The localized equations \eqref{diffusive1-final-ver} and
\eqref{diffusive2-final-ver} below, and all identities derived from them
in the subsequent averaging argument, are
understood almost surely in $\cD'(\R^{d+1}\times U)$.
Global norms of $\tilde h,\beta_{\loc},\zeta_{\loc}$ and
$\cI_{\loc}(\Phi)$ use their compact extensions to $\R_\lambda$.
The Fourier transform acts only in $(t,\mx)$. The subsequent
symbol operations and velocity differentiation are performed on $U$.
For every $\rho\in C_c^1(I)$, the velocity test functions, including their
derivatives, have support in $\supp\rho\subset I\Subset U$.
Consequently, the extensions outside $U$ do not affect any tested
identity, and velocity integration by parts has no boundary
contribution.
We suppress this restriction in the subsequent notation and
continue to write velocity integrals over $\R$.
Since $\pa_\lambda P=S$ on $U$, the extension of $P$ introduces
no additional source term against velocity test functions supported in
$\supp\rho\Subset U$.

The same localization applies to every $\rho\in C_c^1(I)$,
so the averaging weight can still be fixed in Step III.
The non-degeneracy condition is used only on $I$.
\end{remark}

For use in the next section, we now suppress the
localization notation and write
$(h,\beta,\zeta,\cI)$ for
$(\tilde h,\beta_{\loc},\zeta_{\loc},\cI_{\loc})$.
The two kinetic equations take the following form:
\begin{align}
	\label{diffusive1-final-ver}
	& \pa_t h + \sum_{i=1}^d \pa_{x_i}
	\bigl( f_i(\lambda) h \bigr)
	- \sum_{i,j=1}^d \pa_{x_i x_j}^2 
	\bigl(a_{ij}(\lambda) h \bigr)
	\\ & \notag \qquad 
	=-\pa_t \cI\bigl(B\pa_\lambda h)
	+ \pa_\lambda \zeta
	=-\pa_\lambda\pa_t \cI (B h)
	+\pa_t \cI (\pa_\lambda B h)
	+ \pa_\lambda \zeta,
	\\ 
	\label{diffusive2-final-ver}
	& \pa_t h + \sum_{i=1}^d \pa_{x_i} 
	\bigl( f_i(\lambda) h \bigr)
	-\sum_{i,k=1}^d \pa_{x_i} 
	\bigl(\sigma_{ik}(\lambda)\beta_k\bigr)
	\\ & \notag \qquad 
	= -\pa_t \cI\bigl(B\pa_\lambda h)
	+ \pa_\lambda \zeta
	=-\pa_\lambda\pa_t \cI (B h)
	+\pa_t \cI (\pa_\lambda B h)
	+ \pa_\lambda \zeta,
\end{align}
both of which hold a.s.~in $\cD'(\R^{d+1}\times U)$,
with the convention of Remark \ref{rem:localized-equations}.
It is worth mentioning that the specific 
form of $h$ is not important, whether it is derived 
from $\chi_\pm$, $\chi$, or $\frac{1}{2}\big(\sgn(u-\lambda)
-\sgn(\lambda)\big)$ (see Subsection \ref{subsec:kinetic}). 
By Lemma \ref{lem:localization}, $\zeta$ is a finite
signed measure and
$E[\norm{\zeta}_{\cM(\R^{d+2})}^q
+\norm{\beta}_{L^2_{w\star}(\R^{d+1};\cM(\R)^d)}^q]<\infty$
for every $1\leq q<\infty$.

Regarding the stochastic integrals $\cI\bigl(Bh)$ 
and $\cI (\pa_\lambda B h)$, we will 
later apply the fractional Sobolev estimate 
\eqref{eq:frac-Sob-Ito-ver0-H}
with 
\begin{equation}\label{eq:our-Phi}
	\Phi(t,\mx,\lambda)
	=B(\lambda)h(t,\mx,\lambda),
	\qquad
	\Phi(t,\mx,\lambda)
	=\pa_\lambda B(\lambda) 
	h(t,\mx,\lambda), 
\end{equation}
where $\lambda$ is treated as a parameter. In view of the 
compact support assumption, the 
relevant example of $\bbH$ is provided 
by $L^2(\R^d)$ (see \eqref{eq:H-space}). 
Consequently, the estimate presented 
in \eqref{eq:frac-Sob-Ito-ver0-H} 
is applicable, resulting in
\begin{align}
	&E\left[
	\norm{\cI(\Phi(\cdot,\cdot,
	\lambda))}_{H^{\nu,q}(\R;L^2(\R^d))}^q
	\right] \lesssim_{q,\nu} E\left[
	\int_{\R} 
	\norm{\Phi(t,\cdot,\lambda)}_{L_2(\bbK,L^2(\R^d))}^q
	\,dt \right]
	\notag \\ & 
	\qquad
	\overset{\eqref{eq:dp-def-noise-B},
	\eqref{eq:dp-def-noise-pa-B}}{\lesssim} 
	E\int_\R\norm{h(t,\cdot,\lambda)}_{L^2(\R^d)}^q\,dt,
	\qquad \forall \lambda \in \R, 
	\quad q\ge 2, 
	\quad \nu < \frac{1}{2},
	\label{eq:frac-Sob-Ito-ver1}
\end{align}
where the implicit constant in the inequality 
$\lesssim$ depends on the 
compact support ($\subset (0,T)$) of $\cI$ in $t$,
and $h$ in $\lambda$. To elaborate, given 
\eqref{eq:our-Phi}, \eqref{eq:dp-def-noise-B}, 
and the compact support assumptions, 
\begin{align*}
	\norm{(Bh)(t,\cdot,\lambda)}_{L_2(\bbK,L^2(\R^d))}^2
	& =\int_{\R^d}\sum_{\ell\geq 1}
	\abs{b_\ell(t,\mx,\lambda)}^2
	\abs{h(t,\mx,\lambda)}^2\, d\mx
	\\ & 
	\overset{\eqref{eq:dp-def-noise-B}}{\lesssim}
	\norm{h(t,\cdot,\lambda)}_{L^2(\R^d)}^2.
\end{align*}
Similarly, using \eqref{eq:dp-def-noise-pa-B}, 
$\norm{(\pa_\lambda Bh)(t,\cdot,\lambda)}_{L_2(\bbK,L^2(\R^d))}^2
\lesssim \norm{h(t,\cdot,\lambda)}_{L^2(\R^d)}^2$.
We will also require the 
following (for $\Phi$ as in \eqref{eq:our-Phi}): 
\begin{align}\label{eq:BDG-Phi}
	&E \sup_{t\in \R} 
	\norm{\cI(\Phi(\cdot,\cdot,\lambda))(t)}_{L^2(\R^d)}^q
	\lesssim
	E\left[\left(\int_{\R} 
	\norm{\Phi(t,\cdot,\lambda)}_{L_2(\bbK,L^2(\R^d))}^2
	\,dt\right)^{q/2}\right]
	\\ & \qquad\quad \lesssim 
	E\left[\left(\int_{\R} 
	\norm{h(t,\cdot,\lambda)}_{L^2(\R^d)}^2
	\,dt\right)^{q/2}\right],
	\qquad \forall \lambda \in \R, 
	\quad q\ge 1,
	\notag
\end{align}
which comes from the BDG inequality, where the 
final line crucially depends on 
the compact support assumptions.

\section{The regularity result}\label{sec:regularity}

Before presenting the main theorem, we will establish 
a few results concerning $L^1$-type 
Fourier multiplier operators. We recall 
that a comprehensive characterization of 
$L^p$-Fourier multipliers 
for arbitrary $p\in [1,\infty]$ remains elusive. 
Nevertheless, in the specific cases of $p\in \{1,2\}$, 
the situation has been fully resolved, as 
demonstrated in \cite[Theorems 2.5.8 
and 2.5.10]{Grafakos:2014aa} (see 
also \cite[Theorem 6.1.2]{Bergh:1976aa}):

\begin{proposition}\label{L1-multiplier}
Let $\psi : \R^D \to \C$ be a given function, and 
let $\mathcal{A}_\psi$ represent the 
corresponding Fourier multiplier operator. 
The following properties hold:
\begin{itemize}
	\item[i)] $\mathcal{A}_\psi$ is an $L^1$-multiplier 
	operator if and only if ${\cal F}^{-1}(\psi)$ is 
	a finite Borel measure, and the norm is equal to 
	the total variation of the measure ${\cal F}^{-1}(\psi)$. 
	Furthermore, if ${\cal F}^{-1}(\psi) \in L^1(\R^D)$, 
	$$
	\norm{\mathcal{A}_\psi}_{L^1 \to L^1} 
	=\norm{{\cal F}^{-1}(\psi)}_{L^1(\R^D)}.
	$$

	\item[ii)] $\mathcal{A}_\psi$ is an $L^2$-multiplier 
	operator if and only if $\psi$ is a bounded function, 
	with its norm given by 
	$$
	\norm{\mathcal{A}_\psi}_{L^2 \to L^2} 
	= \norm{\psi}_{L^\infty(\R^D)}.
	$$
\end{itemize}
\end{proposition}

The following lemma will be used repeatedly.

\begin{lemma}\label{L-ocjena-L1}
Let $\psi=\psi(\mxi,\lambda)\in C(\R^D\times\R)$,
$\rho\in C_c(\R)$ and
$\beta\in L^1_{w\star}(\R^D;\mathcal M(\R))$,
so that $\beta$ is weak-star measurable with integrable total
variation. Suppose there exists a compact set $K\subset\R^D$
such that $\supp\psi(\cdot,\lambda)\subset K$ for all
$\lambda\in\supp\rho$, and assume
$$
M_\psi:=\sup_{\lambda\in\supp\rho}
\norm{{\cal F}^{-1}\bigl(\psi(\cdot,\lambda)\bigr)}_{L^1(\R^D)}
<\infty.
$$
We denote
\begin{align*}
	&\left(\, \int_{\R}{\cal A}_{ \psi(\cdot,\lambda)} 
	\rho(\lambda)\,d\beta(\cdot,\lambda)\right)(\mx)
	\\ & \qquad 
	:=\int_{\R^{2D}} e^{2\pi i (\mx-\my)\cdot \mxi}
	\int_{\R}\psi(\mxi , \lambda) \rho(\lambda)
	\,d\beta(\my,\lambda)  \, d\my\,d\mxi
	\\ & 
	\qquad ={\cal F}^{-1}\left(\, \int_{\R} \psi(\cdot,\lambda) 
	\rho(\lambda)\,d\widehat{\beta}(\cdot,\lambda)\right)(\mx),
\end{align*}
where $\widehat{\beta}(\cdot,\lambda)$ is the 
Fourier transform of $\beta$ with respect to the first variable.

Then
\begin{equation}\label{ocjena-L1}
	\norm{\int_{\R}{\cal A}_{\psi(\cdot,\lambda)}
	\rho(\lambda)\,d\beta(\cdot,\lambda)}_{L^1(\R^D)}
	\le M_\psi\,
	\norm{\int_{\R}|\rho(\lambda)|
	\,d|\beta|(\cdot,\lambda)}_{L^1(\R^D)},
\end{equation}
where, for $\mx\in\R^D$, $|\beta|=|\beta|(\mx,\cdot)$
denotes the total variation measure of $\beta=\beta(\mx,\cdot)$.
The norm $M_\psi$ is also invariant under
$\mxi\mapsto A(\lambda)\mxi$, where $A(\lambda)$ is an invertible
$D\times D$ matrix for each $\lambda\in\supp\rho$ 
(see \eqref{eq:linear-kernel} below).
\end{lemma}

\begin{proof}
Throughout the proof, we have $\lambda \in \R$ 
and $\mx, \mxi \in \R^D$. Here, $\mx$ refers to the 
physical variable and $\mxi$ to its 
dual (Fourier) counterpart.

Let us take $\psi$ and $\rho$ as in the statement and let 
$M\in \N$ be such that $\supp(\rho)\subset [-M/2,M/2]=:L_M$. 
For $n \in \N$, let $\seq{K_j^n}_{j=1}^{Mn}$ denote 
a sequence of disjoint intervals that partition $L_M$ 
into $Mn$ subintervals of equal length ($1/n$), 
arranged symmetrically about the origin. 
Note that the closure of $\bigcup\limits_{j=1}^{Mn}K_j^n$ 
equals $L_M$.

Choose $\lambda_j^n\in K_j^n\cap\supp\rho$ for each
interval meeting $\supp\rho$, and omit the other intervals
from the sums below. Put
$$
\psi_n(\mxi,\lambda)=\sum\limits_{j=1}^{Mn}
\psi(\mxi,\lambda_j^n)\chi_j^n(\lambda),
$$
where $\chi_j^n$ is the characteristic function of $K_j^n$.
Applying Proposition \ref{L1-multiplier}(i), we get
\begin{align*}
	&\norm{\int_{\R}{\cal A}_{\psi_n(\cdot,\lambda)}
	\rho(\lambda)\,d\beta(\cdot,\lambda)}_{L^1(\R^D)}
	\\ &\quad
	=\norm{\sum\limits_{j=1}^{Mn}
	{\cal A}_{\psi(\cdot,\lambda_j^n)}
	\int_{\R}\chi_j^n(\lambda)\rho(\lambda)
	\,d\beta(\cdot,\lambda)}_{L^1(\R^D)}
	\\ &\quad
	\leq\sum\limits_{j=1}^{Mn}
	\norm{{\cal F}^{-1}\bigl(\psi(\cdot,\lambda_j^n)\bigr)}_{L^1(\R^D)}
	\norm{\int_{\R}\chi_j^n(\lambda)\rho(\lambda)
	\,d\beta(\cdot,\lambda)}_{L^1(\R^D)}
	\\ &\quad
	\le M_\psi\,
	\norm{\int_{\R}|\rho(\lambda)|
	\,d|\beta|(\cdot,\lambda)}_{L^1(\R^D)}.
\end{align*}
By uniform continuity of $\psi$ on $K\times\supp\rho$,
$\psi_n\to\psi$ uniformly there. In the Fourier representation
above, replacing $\psi$ by $\psi_n$ changes the integral by at most
$$
\meas(K)\sup_{K\times\supp\rho}\abs{\psi_n-\psi}\,
\norm{\int_{\R}|\rho(\lambda)|\,d|\beta|(\cdot,\lambda)}_{L^1(\R^D)}
$$
at every $\mx\in\R^D$. Fatou's lemma gives \eqref{ocjena-L1}.

If $\psi$ is also of class $C^{D+1}$ in $\mxi$,
integration by parts gives, for each $\lambda\in\supp\rho$,
\begin{align}
	\norm{{\cal F}^{-1}\bigl(\psi(\cdot,\lambda)\bigr)}_{L^1(\R^D)}
	&\lesssim_D \norm{\psi(\cdot,\lambda)}_{L^1(\R^D)}
	+\sum_{k=1}^D
	\norm{\pa_{\xi_k}^{D+1}\psi(\cdot,\lambda)}_{L^1(\R^D)}
	\notag
	\\ &\lesssim_{D,K}
	\norm{\psi(\cdot,\lambda)}_{C^{D+1}(K)}.
	\label{eq:compact-kernel}
\end{align}
Thus $M_\psi\lesssim_{D,K}\sup_{\lambda\in\supp\rho}
\norm{\psi(\cdot,\lambda)}_{C^{D+1}(K)}$.
For the invariance assertion, let $A$ be an invertible $D\times D$
matrix and $b$ satisfy $\mathcal F^{-1}b\in L^1(\R^D)$. Then
\begin{equation}\label{eq:linear-kernel}
	\begin{aligned}
		{\cal F}^{-1}[b(A\,\cdot)](\mx)
		&=\abs{\det A}^{-1}({\cal F}^{-1}b)(A^{-T}\mx),\\
		\norm{{\cal F}^{-1}[b(A\,\cdot)]}_{L^1}
		&=\norm{{\cal F}^{-1}b}_{L^1}.
	\end{aligned}
\end{equation}
Apply \eqref{eq:linear-kernel} with $b=\psi(\cdot,\lambda)$ and
$A=A(\lambda)$, then take $\sup_{\lambda\in\supp\rho}$ to obtain
the invariance of $M_\psi$.
\end{proof}

\begin{remark}\label{Lp-bound}
We observe that the same bound \eqref{ocjena-L1} holds 
if we replace $d\beta$ by any function in
$L^1(\R^{D+1})$.
\end{remark}

\begin{remark}\label{rem:L1Linfty}
Since $L^1$-multiplier operators also serve as
$L^\infty$-multiplier operators with identical norms
(see \cite[Theorem 6.1.2]{Bergh:1976aa} and
\cite[Subsection 2.5.4]{Grafakos:2014aa}),
under the assumption ${\cal F}^{-1}(\psi)\in L^1(\R^D)$
we also have $\norm{{\cal A}_{\psi}}_{L^\infty\to L^\infty}
=\norm{{\cal F}^{-1}(\psi)}_{L^1(\R^D)}$.
\end{remark}

We will also require the fact that 
multiplier operators can be defined 
on the space of Radon measures. 
The following lemma holds:

\begin{lemma}\label{l-meas}
Assume that $\psi:\R^D\to \C$ is such that 
${\cal F}^{-1}(\psi)\in L^1(\R^D)$. 
Let $\mu\in {\cal M}(\R^D)$ be a bounded Radon measure. 
Then for the action of the multiplier operator ${\cal A}_\psi$ 
on the measure $\mu$ defined for all 
$\varphi \in C_0(\R^D)$ by
$$
\bigl\langle {\cal A}_{\psi}(\mu),
\bar\varphi \bigr\rangle 
:= \left\langle \mu, 
\overline{{\cal A}_{\bar\psi}(\varphi)}\right\rangle
=\int_{\R^D}\overline{{\cal A}_{\bar\psi}(\varphi)}\,d\mu,
$$ 
the following bound holds:
\begin{equation*}
	\norm{{\cal A}_{\psi}(\mu)}_{{\cal M}(\R^D)} 
	\leq \norm{\mu}_{{\cal M}(\R^D)}
	\norm{{\cal F}^{-1}(\psi)}_{L^1(\R^D)}.
\end{equation*} 
\end{lemma}

\begin{proof}
Observe initially that the function $\bar\psi$ serves 
as the symbol of an $L^1$-multiplier operator 
(see  Proposition \ref{L1-multiplier}). Consequently, as 
indicated in Remark \ref{rem:L1Linfty}, $\bar\psi$ also acts 
as the symbol of an $L^\infty$-multiplier operator, with 
the same norm. Hence, we deduce
\begin{equation*}
	\begin{split}
		\abs{\langle {\cal A}_{\psi}(\mu), 
		\, \bar\varphi \rangle} 
		= \abs{\int_{\R^D} \overline{{\cal A}_{\bar\psi}(\varphi)}\,d\mu}
		&\leq  \norm{{\cal A}_{\bar\psi}(\varphi)}_{L^\infty(\R^D)} 
		\, \norm{\mu}_{{\cal M}(\R^D)} 
		\\& 
		\leq \norm{\mu}_{{\cal M}(\R^D)} 
		\norm{{\cal F}^{-1}(\bar\psi)}_{L^1(\R^D)} 
		\norm{\varphi}_{L^\infty(\R^D)},
	\end{split}
\end{equation*} 
i.e.,
$$
\norm{{\cal A}_{\psi}(\mu)}_{{\cal M}(\R^D)} 
\leq \norm{\mu}_{{\cal M}(\R^D)} 
\norm{{\cal F}^{-1}(\psi)}_{L^1(\R^D)} .
$$ 
By Fubini's theorem, $\mathcal A_\psi(\mu)$ has the $L^1$ density
$\bigl(\mathcal F^{-1}\psi\bigr)*\mu$.
\end{proof}

We now state and prove our main result,
Theorem \ref{t-regularity} below, on the regularity of velocity averages
for the stochastic kinetic equation \eqref{diffusive-intro-ver2}.
To deduce Theorem \ref{thm:main-intro}, observe that,
for $h=\En_{\lambda<u}$,
$$
\overline h=\int_\R h(\cdot,\lambda)\rho(\lambda)\,d\lambda
=F_\rho(u).
$$
Under the range assumption in Theorem \ref{thm:main-intro},
choose $\rho\in C_c^1(I)$ equal to one on $J$.
Then $F_\rho(u)=u+C$ a.e.~on $Q$, giving the asserted regularity
of $u$ in Theorem \ref{thm:main-intro}.

\begin{theorem}\label{t-regularity}
Suppose the coefficients $\mff$ and $a$ 
from the stochastic degenerate parabolic 
equation \eqref{d-p} satisfy 
assumptions ({\bf a}) and ({\bf b}) 
from the introduction as well as the  
non-degeneracy condition \eqref{non-deg-stand}
for all $\delta \in (0, \delta_0)$, where 
$\alpha, \delta_0 > 0$, and for a compact 
interval $I \subset\R$. Furthermore, 
suppose that the noise amplitude $B$ fulfills 
the conditions \eqref{eq:b-ell-def} 
and \eqref{eq:B-main-ass}.

Let $h$ be the weak solution 
of \eqref{diffusive-intro-ver2}, 
associated with the kinetic solution $u$ 
of \eqref{d-p} as defined in 
Definition \ref{def:kinetic-sol}. 
Then for every $\rho \in C^1_c(I)$ 
and $s \in (0, s_*)$, it holds that 
$$
\overline{h}(\omega,t,\mx):=\int_{\R} 
h(\omega,t,\mx, \lambda) \rho(\lambda)\,d\lambda 
\in L^{q_*}\bigl(\Omega;W_{\loc}^{s,q_*}(\R_+^{d+1})\bigr), 
$$
where
\begin{equation*}
	q_*=q_*(\alpha,d):=
	\frac{\alpha+2(d+4)}{\alpha+d+4} ,
	\qquad
	s_* =s_*(\alpha,d):=
	\frac{\alpha}{6\alpha+12(d+4)} .
\end{equation*}
\end{theorem}

\begin{proof}
The proof of Theorem \ref{t-regularity} is 
organized into five steps. The first step serves 
as a preliminary stage, where we leverage the linear 
structure of the problem to localize the coefficients 
and the solution on a compact set, as discussed 
in the previous section. In the second step, we utilize the 
regularity condition \eqref{regularity} to derive 
a new kinetic formulation that is inherently homogeneous.

The next two steps focus on obtaining bounds: first for the case 
where the symbol is regular (Step III), and then for the
degenerate case, where the symbol vanishes (Step IV).
A key aspect of Step IV is the application of the
non-degeneracy condition, which is crucial and, 
in fact, the only step where it is used.

The last part of the proof addresses interpolation 
and the derivation of the final estimate. 
This concluding step is well-established and 
follows fairly standard 
techniques (cf.~\cite{Tadmor:2006vn}).

\medskip

\noindent \textbf{Step I.}

\medskip

\noindent Apply Lemma \ref{lem:localization} on the compact
space-time region under consideration and use the notation of
Remark \ref{rem:localized-equations}.
The localized $h$ satisfies \eqref{diffusive1-final-ver} and
\eqref{diffusive2-final-ver} with a common source $\zeta$.
The quantities $h$, $\beta$, $\zeta$, $\cI(Bh)$, and
$\cI((\pa_\lambda B)h)$ have compact support and satisfy the
bounds of Lemma \ref{lem:localization}, including
\eqref{eq:frac-Sob-Ito-ver1} and \eqref{eq:BDG-Phi}.

\medskip

\noindent \textbf{Step II.} 

\medskip

\noindent Keeping Step I in mind, 
for a.e.~$\omega\in \Omega$, 
we can take the Fourier transform of 
\eqref{diffusive1-final-ver} with respect to $(t,\mx)$. 
We denote the corresponding dual variable by 
\begin{equation}\label{eq:xi-prime-def-tmp}
	\mxi=(\xi_0,\tilde{\mxi}), 
	\qquad 
	\xi_0\in \R, 
	\qquad 
	\tilde{\mxi}\in \R^d,
\end{equation}
and divide the obtained 
relation by $\absb{(\xi_0,\tilde{\mxi})}^{1+r}
=\abs{\mxi}^{1+r}$ for some fixed $r\in (0,1)$ to 
be determined later. After denoting 
\begin{equation}\label{eq:xi-prime-def}
	\mxi':=\frac{\mxi}{|\mxi|}, 
	\qquad 
	\tilde{\mxi}'=\frac{\tilde{\mxi}}{|\mxi|}, 
	\qquad 
	\xi_0'=\frac{\xi_0}{|\mxi|},
\end{equation}
we have
\begin{equation}\label{axi}
	\begin{split}
		4\pi^2\Big{\langle} a(\lambda) \tilde\mxi'\,|
		\,\frac{\tilde\mxi}{|\mxi|^{r}}
		\Big{\rangle} \widehat{h}(\mxi,\lambda) 
		=&-2\pi i
		\frac{\xi_0'+\langle f(\lambda)
		\,|\,\tilde{\mxi}' \rangle}{|\mxi|^{r}} 
		\widehat{h}(\mxi,\lambda)+\frac{1}{|\mxi|^{1+r}}
		\pa_\lambda\widehat{\zeta}(\mxi,\lambda) 
		\\ &
		-2\pi i\frac{\xi_0'}{|\mxi|^{r}}
		\bigl(\widehat{\cI(B \pa_\lambda h) }\bigr)(\mxi,\lambda),
	\end{split}
\end{equation} 
where $\widehat{\cdot}$ denotes the Fourier 
transform with respect to $(t,\mx)$. The latter equation 
describes the behavior of $\widehat{h}$ on the support 
of $\bigl\langle a(\lambda) \tilde{\mxi}'
\,|\,\tilde{\mxi}' \bigr\rangle$.

From \eqref{diffusive1-final-ver}, we recall that 
$\pa_t \cI\bigl(B\pa_\lambda h)$ can be rewritten 
using the divergence-form as $\pa_\lambda \pa_t \cI (B h) 
- \pa_t \cI (\pa_\lambda B h)$. This form will be utilized later 
when estimating the stochastic components of the equations. 
However, for clarity in the current discussion, 
we will continue to use the form $\pa_t \cI\bigl(B\pa_\lambda h)$, 
as done in the second line of \eqref{axi}.

The next step is to see what is happening out 
of the support of 
$\langle a(\lambda) \tilde{\mxi}'\,|\,\tilde{\mxi}'\rangle$. 
According to the non-degeneracy condition \eqref{non-deg-stand}, this is 
equivalent to describing the behavior of $\widehat{h}$ on 
the support of $\xi_0'+\langle f(\lambda)\,|\,\tilde{\mxi}' \rangle$.  
To this end, we use \eqref{diffusive2-final-ver}. 
We aim to eliminate the 
influence of the term $\sum_{k,j=1}^d \pa_{x_j}
\bigl( \sigma_{kj}(\lambda) \, \beta_k \bigr)$ in the 
latter expression.
Choose an even $\Phi\in C_c^\infty(\R;[0,1])$,
nonincreasing on $[0,\infty)$, equal to one for
$\abs{s}\le5/4$ and zero for $\abs{s}\ge7/4$. Set $\Psi_0=\Phi$
and
$$
\Psi(s)=
\begin{cases}
	\Phi(s)-\Phi(2s),&s>0,\\
	0,&s\le0.
\end{cases}
$$
These cutoffs are smooth and nonnegative, with
$\supp\Psi\subset(1/2,2)$ and $\supp\Psi_0\subset[-2,2]$.
Telescoping gives
\begin{equation}\label{functions}
	\begin{split}
		&\Phi(z)=\begin{cases}
			1,&\abs{z}\le1,\\
			0,&\abs{z}>2,
		\end{cases}\\
		&\Psi_0(\abs{\mxi})
		+\sum_{J\ge1}\Psi(2^{-J}\abs{\mxi})=1.
	\end{split}
\end{equation}
Furthermore, for $J\in\N$
and $\epsilon\in (0,1)$ such that 
\begin{equation}\label{eq:r-eps-cond}
	r+\epsilon=1,
\end{equation}
we define 
\begin{equation}\label{eq:PsiJPhiJ}
	\Psi_J(\mxi):=\Psi(2^{-J}|\mxi|),
	\qquad
	\Phi_J(\mxi,\lambda):=\Phi\Bigl(2^{\epsilon J} 
	\bigl\langle a(\lambda) \tilde\mxi'\,|\,\tilde\mxi' 
	\bigr\rangle\Bigr).
\end{equation}

Then, we take the Fourier transform of 
\eqref{diffusive2-final-ver} and multiply 
the resulting equation by
\begin{equation}\label{eq:second_eg_mult}
	-2\pi i\frac{\xi_0'
	+\bigl\langle f(\lambda)\,|\, 
	\tilde{\mxi}' \bigr\rangle}{|\mxi|} 
	\, \Phi_J(\mxi,\lambda) ,
\end{equation}
for a fixed $J\in \N$.
We get
\begin{equation}\label{fxi}
	\begin{split}
		&4\pi^2 \Phi_J(\mxi,\lambda) 
		\,\big(\xi_0'+\bigl\langle f(\lambda)
		\,|\,\tilde{\mxi}'\bigr\rangle\big)^2 
		\, \widehat{h} 
		\\ &\qquad\qquad
		-4\pi^2\Phi_J(\mxi,\lambda)
		\big(\xi_0'+\bigl\langle {f}(\lambda)
		\,|\,\tilde{\mxi}' \bigr\rangle \big) 
		\sum\limits_{k,j=1}^d 
		\Bigl( \sigma_{kj}(\lambda) \frac{\xi_j}
		{|\mxi|} \, \widehat{\beta}_k \Bigr)
		\\ & \quad
		= -2\pi i \, \frac{\xi'_0
		+\bigl\langle f(\lambda)\,|\,\tilde{\mxi}' 
		\bigr\rangle}{|\mxi|}\, \Phi_J(\mxi,\lambda) 
		\pa_\lambda \widehat{\zeta}(\mxi,\lambda)
		\\ &\qquad\qquad 
		- 4\pi^2 \Phi_J(\mxi,\lambda) \, 
		\xi_0'\bigl(\xi_0'+\bigl\langle f(\lambda)
		\,|\,\tilde{\mxi}'\bigr\rangle
		\bigr)\bigl[\widehat{\cI(B \pa_\lambda h)}
		\bigr](\mxi,\lambda).
	\end{split}
\end{equation} 
We note that the Fourier transforms presented 
above are all well-defined because of the compact 
support assumptions (discussed in Step I 
and Subsection \ref{subsec:compact-supp}).

After adding \eqref{axi} and \eqref{fxi}, we multiply the resulting 
expression by $\Psi_J(\mxi)$. This leads to
\begin{equation*}
	\begin{split}
		& 4\pi^2 \Psi_J \, \Big(\big(\xi_0'
		+\bigl\langle f\, | \, \tilde{\mxi}' 
		\bigr\rangle\big)^2\Phi_J 
		+\Big{\langle} a\,\tilde{\mxi}' 
		\, | \, \frac{\tilde{\mxi}}{|\mxi|^{r}} 
		\Big{\rangle} \Big) \, \widehat{h} 
		\\ &\qquad\qquad 
		- 4\pi^2\Psi_J \, \bigl(\xi_0'+\langle f
		\, | \, \tilde\mxi'\rangle\bigr) \Phi_J 
		\sum\limits_{k,j=1}^d  \Bigl( \sigma_{kj} 
		\frac{\xi_j}{|\mxi|} \, \widehat{\beta}_k \Bigr)
		\\ & \quad 
		= -2\pi i\, \Psi_J
		\frac{\xi_0'+\bigl\langle f\,|\,\tilde{\mxi}' 
		\bigr \rangle}{|\mxi|^{r}} \widehat{h}
		+\frac{\Psi_J}{|\mxi|^{1+r}}
		\pa_\lambda \widehat{\zeta}
		-2\pi i\,\Psi_J \, \frac{\xi_0'+\bigl\langle f
		\, | \, \tilde{\mxi}'\bigr\rangle}{|\mxi|}
		\, \Phi_J \pa_\lambda \widehat{\zeta} 
		\\ &\qquad\qquad 
		-\Psi_J \bigl[\widehat{\cI(B \pa_\lambda h)}\bigr](\mxi,\lambda) 
		\left( 2\pi i\frac{\xi_0'}{|\mxi|^{r}}+4\pi^2 \xi_0' 
		\bigl(\xi_0'+\bigl\langle {f}(\lambda)
		\,|\,\tilde{\mxi}'\bigr\rangle\bigr) 
		\Phi_J\right).
	\end{split}
\end{equation*}
Hereafter, we often omit 
writing the arguments of functions. 

Next, we define  
\begin{equation}\label{eq:SymbolLJ}
	{\cal L}_J(\mxi,\lambda)
	:=\big(\xi_0'+\bigl\langle f(\lambda)
	\,|\,\tilde{\mxi}'\bigr\rangle\big)^2 
	\, \Phi_J(\mxi,\lambda) 
	+\bigl\langle a(\lambda)\tilde{\mxi}'\,|\,\tilde{\mxi}' 
	\bigr\rangle |\mxi|^{1-r}
\end{equation}
and sum the above identity with respect to $J\in \N$. 
After dividing by $4\pi^2$, we obtain the following 
result (noting that for any fixed $\mxi$, all 
sums are finite since at most three $\Psi_J$ functions 
are non-zero): 
\begin{equation}\label{isolated-h}
	\begin{split}
		\biggl(\sum\limits_{J\in \N} & 
		\Psi_J \, {\cal L}_J\biggr) \widehat{h}
		-  \sum\limits_{J\in \N} \Psi_J 
		\, (\xi_0'+\bigl\langle f\, | \, \tilde{\mxi}'\bigr\rangle) 
		\Phi_J \sum\limits_{k,j=1}^d  \Bigl( \sigma_{kj} 
		\frac{\xi_j}{|\mxi|} \, \widehat{\beta}_k \Bigr) \\
		&= -\frac{i}{2\pi} \sum\limits_{J\in \N} \Psi_J 
		\frac{\xi_0'+\bigl\langle f
		\,|\,\tilde{\mxi}'\bigr \rangle}{|\mxi|^{r}} \widehat{h}
		\\ &  \quad 
		+\frac{1}{4\pi^2}\sum\limits_{J\in \N}
		\frac{\Psi_J}{|\mxi|^{1+r}}\pa_\lambda\widehat{\zeta}
		- \frac{i}{2\pi}\sum\limits_{J\in \N}\Psi_J 
		\, \frac{\xi_0'+\bigl\langle f\,|\,
		\tilde{\mxi}'\bigr\rangle}{|\mxi|}
		\, \Phi_J \pa_\lambda \widehat{\zeta}
		\\&  \quad 
		-\sum\limits_{J\in \N}\Psi_J 
		\bigl[\widehat{\cI(B \pa_\lambda h)}\bigr](\mxi,\lambda) 
		\left( \frac{i \xi_0'}{2\pi |\mxi|^{r}}
		+ \xi_0'(\xi_0'+\bigl\langle f(\lambda)
		\,|\,\tilde{\mxi}'\bigr\rangle)
		\Phi_J\right).
	\end{split}
\end{equation}

Observe that $\widehat{h}$ is isolated on the left-hand side, 
allowing us to achieve the desired bound by dividing by 
$\sum_{J\in \N}\Psi_J \, {\cal L}_J$ and multiplying 
by an appropriate power of $|\mxi|$. Additionally, note that 
the coefficient $\sum_{J\in \N}\Psi_J \, {\cal L}_J$ can be 
viewed as a modification of the original 
symbol in \eqref{diffusive1-final-ver}. 
Finally, it is worth noting that 
\eqref{isolated-h} essentially corresponds to  
a new kinetic formulation of \eqref{d-p}.

\medskip

\noindent\textbf{Step III.}

\medskip

\noindent Let us consider $\tilde{\Phi}:=1-\Phi$, which 
serves as an approximation of the characteristic function for the 
set $\R\backslash \seq{0}$ (see \eqref{functions}). 
Additionally, let us fix $\rho\in C^1_c(I)$. We shall also 
introduce the functions
\begin{equation}\label{eq:Phidelta}
	\tilde{\Phi}_\delta
	=\tilde{\Phi}_\delta(\mxi,\lambda)
	:=\tilde{\Phi}\left(\frac{\sum_{J\in \N}\Psi_J(\mxi)
	\, {\cal L}_J(\mxi,\lambda)}{\delta}\right),
\end{equation}
for $\delta\in (0,\min\{1,\delta_0\})$, and
\begin{align}
	\label{eq:calB-KJdelta}
	{\cal B}_{K,J,\delta}(\mxi,\lambda)
	&:=\frac{\Psi_K(\mxi) \Psi_J(\mxi)
	 \rho(\lambda)\tilde{\Phi}_\delta(\mxi,\lambda)}
	{\sum\limits_{L\in \N}\Psi_L(\mxi)
	\, {\cal L}_L(\mxi,\lambda)}
	\\ & \qquad\quad \times 
	\notag
	 \left(\frac{i\xi_0'}{2\pi|\mxi|^{r}}
	+ \xi_0'\bigl(\xi_0'+\bigl\langle f(\lambda)\,|\,\tilde{\mxi}' 
	\bigr\rangle\bigr)\Phi_J(\mxi,\lambda)\right),
\end{align}
for $K,J\in \N$. For any $K$, we test the equation 
\eqref{isolated-h} with respect to $\lambda \in \R$ 
against the function
$$
\Psi_K(\abs{\mxi})
\frac{\rho(\lambda)
\tilde{\Phi}_\delta(\mxi,\lambda)}
{\sum_{J\in \N} \Psi_J(\mxi) 
{\cal L}_J(\mxi,\lambda)},
$$ 
yielding the equation
\begin{equation}\label{isolated-h-0}
	\begin{split}
		& \int_{\R} \Psi_K \tilde{\Phi}_\delta 
		\,  \rho \, \widehat{h} \,d\lambda
		\\ & \quad 
		=  \sum_{J=K-1}^{K+1} \int_{\R}\rho(\lambda)
		\frac{\tilde{\Phi}_\delta \Psi_K \Psi_J 
		\,\big(\xi_0'+ \bigl\langle f 
		\, | \, \tilde{\mxi}' \bigr\rangle\big)}
		{\sum\limits_{L\in \N}\Psi_L \, {\cal L}_L}
		\,\Phi_J \sum_{k,j=1}^d
		\sigma_{kj} \frac{\xi_j}{|\mxi|} 
		\, d\widehat{\beta}_k(\mxi,\lambda) 
		\\ & \quad \qquad
		-\frac{i}{2\pi} \sum_{J=K-1}^{K+1}
		\int_{\R}\frac{\rho\,\tilde{\Phi}_\delta \Psi_K \Psi_J 
		\big(\xi_0'+\bigl\langle f 
		\,|\,\tilde{\mxi}'\bigr \rangle\big)
		\widehat{h}}{|\mxi|^r\sum_{L\in \N}\Psi_L 
		\, {\cal L}_L}\, d\lambda 
		\\ & \quad \qquad
		- \frac{1}{4\pi^2}\sum_{J=K-1}^{K+1} 
		\int_{\R} \pa_\lambda \left(\frac{\rho(\lambda)
		\tilde{\Phi}_\delta}{\sum_{L\in \N}\Psi_L 
		\, {\cal L}_L}\right)
		\frac{\Psi_K\Psi_J}{|\mxi|^{1+r}} 
		\,d\widehat{\zeta}(\mxi,\lambda)
		\\ & \quad\qquad 
		+ \frac{i}{2\pi}\sum_{J=K-1}^{K+1}\int_{\R}
		\pa_\lambda\left(\frac{\rho\,\tilde{\Phi}_\delta 
		\, {\big(\xi_0'+\bigl\langle f 
		\,|\, \mxi' \bigr\rangle\big)}
		\, \Phi_J}{\sum_{L\in \N}\Psi_L 
		\, {\cal L}_L} \right)\frac{\Psi_K \Psi_J}{|\mxi|} 
		\,d\widehat{\zeta}(\mxi,\lambda)
		\\ & \quad \qquad 
		- \sum_{J=K-1}^{K+1}\int_{\R}
		\widehat{\cI(B \pa_\lambda h)}
		\, {\cal B}_{K,J,\delta}\, d\lambda,
		\qquad K=1,2,\ldots,
	\end{split}
\end{equation} 
where we have used that $\Psi_K\Psi_J = 0$
for $J\not\in\{K-1,K,K+1\}$.
Here and below, neighbouring-index sums include only
$J\ge1$.

The equality above is understood through Fourier 
multiplier operators. 
More precisely, we take the inverse Fourier 
transform of \eqref{isolated-h-0} 
with respect to $\mxi$, resulting in an almost sure 
equality in the sense of distributions. 
In the following steps, we aim to estimate the inverse 
Fourier transforms of each term on the right-hand side 
of \eqref{isolated-h-0}. Our goal is to demonstrate that, with 
a suitable choice of $r$ subject to
\eqref{eq:r-eps-cond}, there is
``room" to incorporate $|\mxi|^s$ for some $s > 0$, while 
ensuring that the terms remain finite as we sum over $K$. 

\medskip

\noindent\underline{The first two terms on the right-hand
side of \eqref{isolated-h-0}.}

\medskip

\noindent The estimation of the first two expressions on the
right-hand side of \eqref{isolated-h-0} uses Lemma \ref{L-ocjena-L1}
(see Remark \ref{Lp-bound} as well).
We estimate $M_\psi$ from Lemma \ref{L-ocjena-L1} by rescaling
the symbols to fixed supports. The symbols below include
$\rho$, so we apply \eqref{ocjena-L1} with a cutoff equal to one
on $I$. For vector symbols, we sum the componentwise $L^1$ norms
of the inverse Fourier transforms.
Recall that $\epsilon=1-r$ and put
$$
G_\delta(s):=\frac{1-\Phi(s/\delta)}{s},
\qquad
\sup_s\abs{G_\delta^{(j)}(s)}\lesssim_j\delta^{-j-1},
$$
with zero extension near zero. Here
$G_\delta^{(j)}=d^jG_\delta/ds^j$, $j=0,1,\ldots$. Thus
$G_\delta(\sum_L\Psi_L\mathcal L_L)
=\tilde\Phi_\delta/\sum_L\Psi_L\mathcal L_L$.

For $K=1$, the fixed-support estimate \eqref{eq:compact-kernel}
applies directly. For $K\ge2$, we have
$\sum_{L\ge1}\Psi_L=1$ on $\supp\Psi_K$.
Choose a smooth nonincreasing $\kappa_0$ on $[0,\infty)$,
equal to one on $[0,8]$ and zero on $[16,\infty)$, and set
$$
\kappa_\ell(s):=\kappa_0(2^{-\ell}s)
-\kappa_0(2^{1-\ell}s),\qquad \ell\ge1.
$$
Then $\sum_{\ell\ge0}\kappa_\ell=1$ and
$\supp\kappa_\ell\subset[8\cdot2^{\ell-1},16\cdot2^\ell]$
for $\ell\ge1$.
For each fixed $\lambda$ and $\ell$, use the linear change of variables
\begin{equation}\label{eq:freq-change}
	S_{K,\ell}:=(I_d+2^{\epsilon K-\ell}a(\lambda))^{1/2},
	\qquad
	\mxi=2^K(\eta_0,S_{K,\ell}^{-1}\tilde\meta).
\end{equation}
Since $a=\sigma^2$, one has
\begin{equation}\label{eq:S-bounds}
	\norm{S_{K,\ell}^{-1}}_{\mathrm{op}}\le1,
	\qquad
	\norm{2^{(\epsilon K-\ell)/2}\sigma S_{K,\ell}^{-1}}
	_{\mathrm{op}}\le1.
\end{equation}
For $\psi$ supported in $\supp\Psi_K$, let
$\widetilde\psi^{\,\ell}$ denote
$\kappa_\ell\bigl(2^{\epsilon K}
\langle a\tilde\mxi'\mid\tilde\mxi'\rangle\bigr)\psi(\mxi,\lambda)$ 
after the change of variables \eqref{eq:freq-change}.
Its support satisfies
\begin{align*}
	\tfrac12&\le\abs{(\eta_0,S_{K,\ell}^{-1}\tilde\meta)}\le2,\\
	\abs{\meta}^2
	&=2^{-2K}\bigl(\abs{\mxi}^2
	+2^{\epsilon K-\ell}\langle a\tilde\mxi\mid\tilde\mxi\rangle\bigr)
	\le17\cdot2^{-2K}\abs{\mxi}^2\le68.
\end{align*}
Thus $\supp\widetilde\psi^{\,\ell}$ is contained in
$\{\meta\in\R^{d+1}:1/2\le\abs{\meta}\le\sqrt{68}\}$.
Here $\norm{\cdot}_{C_\meta^{d+2}}$ is the supremum of the derivatives
$\pa_\meta^\alpha$, $\abs{\alpha}\le d+2$, on
$\supp\widetilde\psi^{\,\ell}$,
after expressing $\mxi$ in terms of $\meta$ by \eqref{eq:freq-change}.
The inequalities \eqref{eq:S-bounds} give
\begin{align*}
	&\norm{\kappa_\ell\bigl(2^{\epsilon K}
	\langle a\tilde\mxi'\mid\tilde\mxi'\rangle\bigr)}_{C_\meta^{d+2}}
	+\norm{\mxi'}_{C_\meta^{d+2}}
	+\norm{\xi_0'+\langle f(\lambda)\mid\tilde\mxi'\rangle}
	_{C_\meta^{d+2}}\lesssim1.
\end{align*}
For $\ell=0$, \eqref{eq:r-eps-cond}, \eqref{eq:S-bounds}, and
\eqref{eq:SymbolLJ} give
$$
\norm{\sigma\tilde\mxi'}_{C_\meta^{d+2}}
\lesssim2^{-\epsilon K/2},
\qquad
\norm{\sum_L\Psi_L\mathcal L_L}_{C_\meta^{d+2}}\lesssim1.
$$
The product and chain rules consequently give
$$
\norm{\widetilde\psi^{\,0}}_{C_\meta^{d+2}}
\lesssim\delta^{-d-3}
\begin{cases}
	2^{-\epsilon K/2},&\text{for the symbol in \eqref{beta} below},\\
	2^{-rK},&\text{for the symbol in \eqref{2} below}.
\end{cases}
$$
For $\ell\ge1$, all neighbouring $\Phi_L$ vanish, since
$2^{\epsilon L}\langle a\tilde\mxi'\mid\tilde\mxi'\rangle>2$.
The symbol in \eqref{beta} below therefore vanishes, while
$$
\sum_L\Psi_L\mathcal L_L
=\abs{\mxi}^{\epsilon}
\langle a\tilde\mxi'\mid\tilde\mxi'\rangle\ge4>2\delta.
$$
Hence
$$
\norm{G_\delta(\sum_L\Psi_L\mathcal L_L)}_{C_\meta^{d+2}}
=\norm{(\sum_L\Psi_L\mathcal L_L)^{-1}}_{C_\meta^{d+2}}
\lesssim2^{-\ell},
$$
and the symbol in \eqref{2} below satisfies
$\norm{\widetilde\psi^{\,\ell}}_{C_\meta^{d+2}}
\lesssim2^{-rK}2^{-\ell}$.
Using \eqref{eq:compact-kernel} and \eqref{eq:linear-kernel}
from Lemma \ref{L-ocjena-L1} for $\widetilde\psi^{\,\ell}$
and summing in $\ell$ gives
$M_\psi\lesssim\delta^{-d-3}2^{-\epsilon K/2}$ for \eqref{beta} below
and $M_\psi\lesssim\delta^{-d-3}2^{-rK}$ for \eqref{2} below.
The constants are independent of $K,J,\delta$, and $\lambda\in I$.

Lemma \ref{L-ocjena-L1} yields the following bound for the first
expression on the right-hand side of \eqref{isolated-h-0}
(denoting $\beta=(\beta_1,\dots,\beta_d)$):
\begin{align}
	\Biggl\| & {\cal F}^{-1} \biggl(\int_{\R} \rho(\lambda)
	\big(\xi_0'+\bigl\langle f(\lambda)
	\,|\,\tilde\mxi'\bigr\rangle\big)
	\,\frac{\tilde\Phi_\delta(\cdot,\lambda)}
	{\sum_{L\in\N}\Psi_L\,{\cal L}_L(\cdot,\lambda)}
	\notag
	\\ &\qquad\qquad\qquad\quad
	\times\Psi_K\,\Psi_J\,\Phi_J(\cdot,\lambda)
	\sum_{k,j=1}^d\sigma_{kj}(\lambda)\frac{\xi_j}{|\mxi|}
	\,d\widehat\beta_k\biggr)\Biggr\|_{L^1(\R^{d+1})}
	\notag
	\\ &\quad
	\lesssim\frac{1}{\delta^{d+3}}
	\left(2^{-(1-\epsilon)K}+2^{-\frac12\epsilon K}\right)
	\norm{\beta}_{L^1_{w\star}(\R^{d+1};\cM(\R)^d)}.
	\label{beta}
\end{align}
This bound holds almost surely.
Similarly, the second term on the right-hand
side of \eqref{isolated-h-0} is estimated almost surely
using Remark \ref{Lp-bound}. The final result is
\begin{equation}\label{2}
	\begin{split}
		&\Biggl\|{\cal F}^{-1}\biggl(\int_{\R}\rho(\lambda)
		\frac{\tilde\Phi_\delta(\cdot,\lambda)\,\Psi_K\,\Psi_J
		\big(\xi_0'+\bigl\langle f(\lambda)
		\,|\,\tilde\mxi'\bigr\rangle\big)
		\widehat h(\cdot,\lambda)}{|\mxi|^r
		\sum_{L\in\N}\Psi_L\,{\cal L}_L(\cdot,\lambda)}
		\,d\lambda\biggr)\Biggr\|_{L^1(\R^{d+1})}
		\\ &\qquad
		\lesssim\frac{1}{\delta^{d+3}}2^{-rK}
		\norm{h}_{L^1(\R^{d+2})}.
	\end{split}
\end{equation}

\medskip

\noindent\underline{The third and fourth terms
on the right-hand side of \eqref{isolated-h-0}.}\par\nobreak

\medskip

\noindent To address the final two non-stochastic terms, we
aggregate them. The combined symbol becomes
$\frac{-1}{4\pi^2}\sum_{J=K-1}^{K+1}\Theta_J$, where
\begin{align}
	\Theta_J(\mxi,\lambda)
	&:=\frac{\Psi_K(\mxi)\Psi_J(\mxi)}{|\mxi|}
	\label{eq:ThetaJ-def}
	\pa_\lambda\Biggl(\frac{\rho(\lambda)
	\tilde\Phi_\delta(\mxi,\lambda)}
	{\sum_{L=K-1}^{K+1}\Psi_L(\mxi)
	\,{\cal L}_L(\mxi,\lambda)}
	\\ &\qquad\quad\times
	\Bigl(|\mxi|^{-r}-2\pi i\,\Phi_J(\mxi,\lambda)
	\big(\xi_0'+\bigl\langle f(\lambda)
	\,|\,\tilde\mxi'\bigr\rangle\big)\Bigr)\Biggr).
	\notag
\end{align}
We use the cutoffs $\kappa_\ell$ and the change of variables
\eqref{eq:freq-change} to estimate $M_{\Theta_J}$.
All velocity derivatives in \eqref{eq:ThetaJ-def} are taken at fixed
$\mxi$, before the change of variables. By \eqref{eq:PsiJPhiJ}
and \eqref{eq:SymbolLJ},
\begin{align*}
	\pa_\lambda\Phi_L
	&=2^{\epsilon L}\Phi'
	\bigl(2^{\epsilon L}\langle a\tilde\mxi'\mid\tilde\mxi'\rangle\bigr)
	\langle a'\tilde\mxi'\mid\tilde\mxi'\rangle,\\
	\pa_\lambda\mathcal L_L
	&=2\Phi_L\bigl(\xi_0'+\langle f\mid\tilde\mxi'\rangle\bigr)
	\langle f'\mid\tilde\mxi'\rangle
	+\bigl(\xi_0'+\langle f\mid\tilde\mxi'\rangle\bigr)^2
	\pa_\lambda\Phi_L\\
	&\qquad+\abs{\mxi}^\epsilon
	\langle a'\tilde\mxi'\mid\tilde\mxi'\rangle.
\end{align*}
For $\ell=0$, these identities and \eqref{eq:S-bounds} give
$\norm{\sum_L\Psi_L\pa_\lambda\mathcal L_L}_{C_\meta^{d+2}}
\lesssim2^{\epsilon K}$. The product and chain rules in
\eqref{eq:ThetaJ-def} therefore give
$$
\norm{\widetilde\Theta_J^{\,0}}_{C_\meta^{d+2}}
\lesssim\delta^{-d-4}2^{(\epsilon-1)K}
=\delta^{-d-4}2^{-rK}.
$$
For $\ell\ge1$, both $\Phi_L$ and $\pa_\lambda\Phi_L$ vanish.
With $\pa_\lambda$ taken at fixed $\mxi$ before \eqref{eq:freq-change},
\begin{align*}
	\norm{(\sum_L\Psi_L\mathcal L_L)^{-1}}_{C_\meta^{d+2}}
	&\lesssim2^{-\ell},\\
	\norm{\pa_\lambda(\sum_L\Psi_L\mathcal L_L)^{-1}}
	_{C_\meta^{d+2}}
	&\lesssim2^{\epsilon K}2^{-2\ell}.
\end{align*}
Thus
$$
\norm{\widetilde\Theta_J^{\,\ell}}_{C_\meta^{d+2}}
\lesssim2^{-(1+r)K}
\left(2^{-\ell}+2^{\epsilon K}2^{-2\ell}\right).
$$
Using \eqref{eq:compact-kernel} and \eqref{eq:linear-kernel}
from Lemma \ref{L-ocjena-L1} for $\widetilde\Theta_J^{\,\ell}$
and summing in $\ell$ yields
\begin{equation}\label{eq:kernel-bounds}
	M_{\Theta_J}\lesssim\delta^{-d-4}2^{-rK}.
\end{equation}
Only the bounds for $a,f,a',f',\rho,\rho'$ on $I$ enter the constant.
No velocity derivative of $\sigma$ or $S_{K,\ell}$ is used.

The joint measure $\zeta$ is finite by Lemma \ref{lem:localization}.
For $z=(t,\mx)$, the velocity-integrated multiplier has the
convolution density
$$
\mathcal F^{-1}\left(\int_\R\Theta_J(\cdot,\lambda)
\,d\widehat\zeta(\cdot,\lambda)\right)(z)
=\int_{\R^{d+2}}
\bigl[\mathcal F^{-1}\Theta_J(\cdot,\lambda)\bigr](z-y)
\,\zeta(dy,d\lambda).
$$
To see this, note first that Tonelli's theorem gives
\begin{equation}\label{g-term}
	\begin{aligned}
		&\norm{\mathcal F^{-1}\left(
		\int_\R\Theta_J(\cdot,\lambda)
		\,d\widehat\zeta(\cdot,\lambda)\right)}_{L^1(\R^{d+1})}
		\\ &\qquad
		\le\int_{\R^{d+2}}
		\norm{\mathcal F^{-1}\Theta_J(\cdot,\lambda)}_{L^1(\R^{d+1})}
		\,d\abs{\zeta}(y,\lambda)
		\le M_{\Theta_J}\norm{\zeta}_{\cM(\R^{d+2})}.
	\end{aligned}
\end{equation}
The last inequality uses the definition of $M_\psi$ in
Lemma \ref{L-ocjena-L1}.
Fubini's theorem identifies its Fourier transform with
$\int_\R\Theta_J(\cdot,\lambda)\,d\widehat\zeta(\cdot,\lambda)$.
The total variation of the measure with the displayed convolution
density equals the $L^1$ norm in \eqref{g-term}
(no density of $\zeta$ in $\lambda$ is required).
Combining \eqref{g-term} and \eqref{eq:kernel-bounds} gives, almost surely,
\begin{equation}\label{g-term-1}
	\sum_{J=K-1}^{K+1}
	\norm{\mathcal F^{-1}\left(\int_\R\Theta_J(\cdot,\lambda)
	\,d\widehat\zeta(\cdot,\lambda)\right)}_{L^1(\R^{d+1})}
	\lesssim\frac{2^{-rK}}{\delta^{d+4}}
	\norm{\zeta}_{\cM(\R^{d+2})}.
\end{equation}

For later use, choose a smooth radial cutoff
$\chi_\mxi\in C_c^\infty(\R^{d+1})$, equal to one for
$1/2\le\abs{\mxi}\le2$ and zero for
$\abs{\mxi}\notin(1/4,4)$.
Then $\chi_\mxi(2^{-K}\mxi)=1$ on $\supp\Psi_K$.
For $\psi$ supported in $\supp\Psi_K$, we have
\begin{equation}\label{eq:weighted-kernel}
	M_{\abs{\mxi}^s\psi}\lesssim_s2^{sK}M_\psi,
	\qquad
	M_{(1+\abs{\mxi}^2)^{\gamma/2}\psi}
	\lesssim_\gamma2^{\gamma K}M_\psi.
\end{equation}
By \eqref{eq:compact-kernel} and \eqref{eq:linear-kernel}
from Lemma \ref{L-ocjena-L1},
the substitution $\mxi=2^K\meta$ gives
\begin{align*}
	\norm{\mathcal F^{-1}\bigl(\abs{\mxi}^s
	\chi_\mxi(2^{-K}\mxi)\bigr)}_{L^1}
	&\lesssim_s2^{sK},\\
	\norm{\mathcal F^{-1}\bigl((1+\abs{\mxi}^2)^{\gamma/2}
	\chi_\mxi(2^{-K}\mxi)\bigr)}_{L^1}
	&\lesssim_\gamma2^{\gamma K}.
\end{align*}
Convolving these functions with $\mathcal F^{-1}\psi$ proves
\eqref{eq:weighted-kernel}.

\medskip

\noindent\underline{The fifth (stochastic) term on the right-hand 
side of \eqref{isolated-h-0}.}\par\nobreak

\medskip

\noindent From \eqref{diffusive1-final-ver} and 
\eqref{diffusive2-final-ver}, we will now consider 
the divergence form of $\cI(B\partial_\lambda h)$, 
which is $\partial_\lambda \cI(Bh) -\cI(\partial_\lambda Bh)$. 
This causes the final term in \eqref{isolated-h-0} 
to take the form
\begin{align*}
 	&\sum_{J=K-1}^{K+1}\int_{\R}
	\widehat{\cI(B \pa_\lambda h)}
	\, {\cal B}_{K,J,\delta}\, d\lambda
	\\ & \qquad 
	= \sum_{J=K-1}^{K+1}\int_{\R}
	\partial_\lambda \widehat{\cI(Bh)}
	\, {\cal B}_{K,J,\delta}\, d\lambda
	-\sum_{J=K-1}^{K+1}\int_{\R}
	\widehat{\cI(\partial_\lambda Bh)}
	\, {\cal B}_{K,J,\delta}\, d\lambda,
\end{align*}
where ${\cal B}_{K,J,\delta}$ is defined 
in \eqref{eq:calB-KJdelta}. 
Hence, we need to estimate 
\begin{equation}\label{eq:S1-S2}
	\begin{split}
		&{\cal S}_1:=E\int_{\R}\norm{{\cal F}^{-1}\left(
		\widehat{\cI(Bh)}
		\, \partial_\lambda {\cal B}_{K,J,\delta} 
		\right)}_{L^1(\R^{d+1})}\, d\lambda,
		\\ & 
		{\cal S}_2:=E\int_{\R}\norm{{\cal F}^{-1}\left(
		\widehat{\cI(\pa_\lambda Bh)}
		\, {\cal B}_{K,J,\delta}
		\right)}_{L^1(\R^{d+1})}\, d\lambda.
	\end{split}
\end{equation}

Before we begin, let us 
mention that estimating the stochastic integral 
terms \eqref{eq:S1-S2} becomes more intricate due 
to the unbounded domain $\R^d$, primarily because it 
necessitates the use of a fractional chain rule formula in 
the style of Kato-Ponce. This analysis would be simpler 
if conducted on the torus $\mathbb{T}^d$ 
(as done in \cite{Gess:2018ab}).

Fix $0<\nu<1/2$ and recall from \eqref{eq:r-eps-cond} that $\epsilon=1-r$.
For the symbol in \eqref{eq:calB-KJdelta}, put
\begin{align*}
	\mathscr P=\mathscr P_{K,J,\delta}
	&:=\frac{\rho\,\tilde\Phi_\delta\,\Psi_K\Psi_J}
	{\abs{\mxi}\sum_L\Psi_L\mathcal L_L}
	\left(\frac{i}{2\pi}\abs{\mxi}^{-r}
	+\Phi_J\bigl(\xi_0'+\langle f\mid\tilde\mxi'\rangle\bigr)\right).
\end{align*}
By \eqref{eq:ThetaJ-def},
\begin{equation}\label{eq:kernel-factor}
	\mathcal B_{K,J,\delta}=\xi_0\mathscr P,
	\qquad
	\pa_\lambda\mathcal B_{K,J,\delta}
	=\frac{i}{2\pi}\xi_0\Theta_J,
	\qquad 
	M_{\mathscr P}\lesssim\delta^{-d-3}2^{-K}.
\end{equation}
To estimate $M_{\mathscr P}$, we use the cutoffs $\kappa_\ell$
and the change of variables \eqref{eq:freq-change}.
For $\ell=0$, $\norm{\widetilde{\mathscr P}^{\,0}}_{C_\meta^{d+2}}
\lesssim\delta^{-d-3}2^{-K}$.
For $\ell\ge1$, $\Phi_J=0$ and
$\norm{(\sum_L\Psi_L\mathcal L_L)^{-1}}_{C_\meta^{d+2}}
\lesssim2^{-\ell}$, giving
$\norm{\widetilde{\mathscr P}^{\,\ell}}_{C_\meta^{d+2}}
\lesssim2^{-(1+r)K}2^{-\ell}$. Applying \eqref{eq:compact-kernel}
and \eqref{eq:linear-kernel} from Lemma \ref{L-ocjena-L1},
then summing in $\ell$, proves
the bound for $M_{\mathscr P}$ in \eqref{eq:kernel-factor}.

Write
$$
F_{1,\lambda}=\cI(Bh),\qquad
F_{2,\lambda}=\cI((\pa_\lambda B)h),
\qquad
D_t^\nu=\cA_{\abs{\xi_0}^\nu}.
$$
The stochastic integral $\cI$ includes the outer cutoff $\vartheta_1$ from
Lemma \ref{lem:localization}. Put
$$
c=\Phi=\Psi_0,\qquad
m_1=\pa_\lambda\mathcal B_{K,J,\delta},\qquad
m_2=\mathcal B_{K,J,\delta}.
$$
The function $c$ equals one near zero and vanishes for
$\abs{\tau}\ge2$. Define, with zero extension near $\xi_0=0$,
$$
\psi_{i,1}=c(\xi_0)m_i,\qquad
\tilde\psi_{i,2}
=\abs{\xi_0}^{-\nu}(1-c(\xi_0))m_i,\qquad i=1,2.
$$
The triangle inequality gives
$$
\mathcal S_i\le\mathcal S_{i,1}+\mathcal S_{i,2},
$$
where
\begin{align*}
	\mathcal S_{i,1}
	&:=E\int_\R
	\norm{\cA_{\psi_{i,1}}F_{i,\lambda}}_{L^1(\R^{d+1})}
	\,d\lambda,\\
	\mathcal S_{i,2}
	&:=E\int_\R
	\norm{\cA_{\tilde\psi_{i,2}}D_t^\nu F_{i,\lambda}}
	_{L^1(\R^{d+1})}\,d\lambda.
\end{align*}

Choose $\chi_t\in C_c^\infty(\R)$ equal to one on $[-2,2]$
and zero outside $[-4,4]$, and set
$$
v_0(\tau)=\tau c(\tau),\qquad
v_{\nu,K}(\tau)
=\tau\abs{\tau}^{-\nu}(1-c(\tau))\chi_t(2^{-K}\tau),
$$
where $v_{\nu,K}=0$ near $\tau=0$.
The inverse Fourier transforms $\check v_0$ and $\check v_{\nu,K}$ satisfy
\begin{equation}\label{eq:time-kernel-bounds}
	\norm{\check v_0}_{L^1(\R)}\lesssim1,\qquad
	\norm{\check v_{\nu,K}}_{L^1(\R)}
	\lesssim_\nu2^{(1-\nu)K}.
\end{equation}
The bound for $\norm{\check v_0}_{L^1}$ in
\eqref{eq:time-kernel-bounds} follows from
\eqref{eq:compact-kernel} with $D=1$, since $v_0$ is smooth
and compactly supported. For $v_{\nu,K}$, use
a smooth dyadic partition on $1\lesssim\abs{\tau}\lesssim2^K$.
Choose the partition so that
$\supp[v_{\nu,K,j}(2^j\,\cdot)]
\subset\{s\in\R:1/2\le\abs{s}\le2\}$,
$0\le j\le K+3$. Then
$$
\norm{v_{\nu,K,j}(2^j\,\cdot)}_{C^2}
\lesssim_\nu2^{(1-\nu)j}.
$$
Indeed, $(2^js)\abs{2^js}^{-\nu}
=2^{(1-\nu)j}s\abs{s}^{-\nu}$, and
$\pa_s^k(s\abs{s}^{-\nu})$ is bounded for
$1/2\le\abs{s}\le2$, $k=0,1,2$.
For $k=1,2$, $\pa_s^k[1-c(2^js)]=0$ there when $j\ge3$.
Also, $\abs{\pa_s^k\chi_t(2^{j-K}s)}
\lesssim2^{k(j-K)}\lesssim1$ for $0\le j\le K+3$.
Applying \eqref{eq:compact-kernel}
with $D=1$ and \eqref{eq:linear-kernel} from Lemma \ref{L-ocjena-L1}
to each
$v_{\nu,K,j}(2^j\,\cdot)$ therefore yields
$$
\norm{\check v_{\nu,K}}_{L^1}
\le\sum_j\norm{\check v_{\nu,K,j}}_{L^1}
\lesssim_\nu\sum_{j=0}^{K+3}2^{(1-\nu)j}
\lesssim_\nu2^{(1-\nu)K}.
$$

On $\supp\Psi_K$ we have $\abs{\xi_0}\le2^{K+1}$, so
$\chi_t(2^{-K}\xi_0)=1$. The identities \eqref{eq:kernel-factor}
therefore give the exact factorizations
\begin{align*}
	\psi_{1,1}
	&=\frac{i}{2\pi}v_0(\xi_0)\Theta_J,&
	\tilde\psi_{1,2}
	&=\frac{i}{2\pi}v_{\nu,K}(\xi_0)\Theta_J,\\
	\psi_{2,1}
	&=v_0(\xi_0)\mathscr P,&
	\tilde\psi_{2,2}
	&=v_{\nu,K}(\xi_0)\mathscr P.
\end{align*}
The inverse space-time Fourier transform of $v(\xi_0)$ is the
finite measure $\check v(t)\,dt\otimes\delta_0(d\mx)$.
Convolving it with $\mathcal F^{-1}\Theta_J$ and
$\mathcal F^{-1}\mathscr P$, whose $L^1$ norms are bounded by
\eqref{eq:kernel-bounds} and \eqref{eq:kernel-factor}, and applying
Young's inequality in $t$ with \eqref{eq:time-kernel-bounds}, we obtain
\begin{equation}\label{eq:stochastic-kernel-bounds}
	\begin{aligned}
		M_{\psi_{1,1}}
		&\lesssim\delta^{-d-4}2^{-rK},&
		M_{\tilde\psi_{1,2}}
		&\lesssim_\nu\delta^{-d-4}2^{-(\nu+r-1)K},\\
		M_{\psi_{2,1}}
		&\lesssim\delta^{-d-3}2^{-K},&
		M_{\tilde\psi_{2,2}}
		&\lesssim_\nu\delta^{-d-3}2^{-\nu K}.
	\end{aligned}
\end{equation}
All constants are independent of $K$ and $\delta$.

Proposition~\ref{L1-multiplier}(i) now gives
\begin{equation}\label{eq:S11-est1}
	\mathcal S_{1,1}
	\le M_{\psi_{1,1}}
	E\int_\R\norm{\cI(Bh)}_{L^1(\R^{d+1})}\,d\lambda.
\end{equation}

The compact support of the mapping $(t,\mx)\mapsto \cI(Bh)$ 
implies that the norm of $\cI(Bh)$ in $L^1(\R^{d+1})$ is 
bounded by a constant times its norm 
in $L^2(\R^{d+1})$. 
Applying the BDG inequality \eqref{eq:BDG-Phi} (with $q=2$), 
keeping in mind that $\cI(Bh)$ has 
compact support in $t$, 
this $L^2$ term can be bounded as follows:
\begin{equation}\label{eq:S11-est1-L2}
	\begin{split}
		& E\int_{\R}\norm{\cI(Bh)}_{L^2(\R^{d+1})}
		\, d\lambda
		\lesssim \left( \int_{\R} 
		E \int_{\R}\norm{\cI(Bh)
		(t,\cdot,\lambda)}_{L^2(\R^d)}^2
		\,dt\, d\lambda \right)^{1/2}
		\\ & \qquad 
		\overset{\eqref{eq:BDG-Phi}}{\lesssim} 
		\left( \int_{\R}E\int_{\R} 
		\norm{h(t,\cdot,\lambda)}_{L^2(\R^d)}^2
		\,dt\, d\lambda \right)^{1/2}
		=\norm{h}_{L^2_{\omega,t,\mx,\lambda}}.			
	\end{split}
\end{equation}

Combining \eqref{eq:S11-est1}, \eqref{eq:S11-est1-L2}, and
\eqref{eq:stochastic-kernel-bounds} yields
\begin{equation}\label{eq:S11-est1-final}
	\mathcal S_{1,1}
	\lesssim\frac{1}{\delta^{d+4}}2^{-rK}
	\norm{h}_{L^2_{\omega,t,\mx,\lambda}}.
\end{equation}

For $\mathcal S_{1,2}$, Proposition~\ref{L1-multiplier}(i) gives
\begin{equation}\label{eq:S12-est1}
	\mathcal S_{1,2}
	\le M_{\tilde\psi_{1,2}}
	E\int_\R\norm{D_t^\nu\cI(Bh)}_{L^1(\R^{d+1})}\,d\lambda.
\end{equation}%
To proceed we must
replace the $L^1$ norm (over an unbounded set)
with the $L^2$ norm. Recall that the function 
$(t, \mx, \lambda) \mapsto \cI(Bh)$ is compactly supported, 
and thus $\cA_{\abs{\xi_0}^{\nu}} \left(\cI(Bh)\right)$ 
is compactly supported in $\mx \in \R^d$ (and in 
$\lambda$), though not in the time variable $t \in \R$. 
Choose $g\in C_c^\infty(\R)$ depending only on time,
equal to one on the fixed time support of $\cI(Bh)$. Then
$$
\cA_{\abs{\xi_0}^{\nu}}
\left(\cI(Bh)\right)=
\cA_{\abs{\xi_0}^{\nu}}
\left(g\, \cI(Bh)\right).
$$
To handle the issue of non-compact support in $t$, we 
combine the compact support in $\mx$ 
and $\lambda$ with the fractional Leibniz rule 
from \cite[Theorem 1]{Grafakos:2014ab} 
in $t$, applied with $r = 1$, $p_1 = p_2 = q_1 = q_2 = 2$, 
$s = \nu$, $f = \cI(Bh)$, and $g = g$, 
to conclude that
\begin{align*}
	&\norm{\cA_{\abs{\xi_0}^{\nu}}
	\left(g\,\cI(Bh)\right)}_{L^1(\R^{d+1})}
	\\ & \qquad 
	\lesssim
	\int_{\R^d}\norm{\cA_{\abs{\xi_0}^{\nu}}
	\left(\cI(Bh)\right)}_{L^2(\R)}
	\norm{g}_{L^2(\R)}\, d\mx
	\\ & \qquad \qquad 
	+
	\int_{\R^d}\norm{\cI(Bh)}_{L^2(\R)}
	\norm{\cA_{\abs{\xi_0}^{\nu}}(g)}_{L^2(\R)}\, d\mx
	\\ & \qquad 
	\lesssim_g 
	\norm{\cA_{\abs{\xi_0}^{\nu}}
	\left(\cI(Bh)\right)}_{L^2(\R^{d+1})}
	+
	\norm{\cI(Bh)}_{L^2(\R^{d+1})}.
\end{align*}
The last inequality uses Cauchy--Schwarz in $\mx$ and
$\supp_\mx\cI(Bh)\subset\supp\phi$ from Lemma \ref{lem:localization}.
The last term is controlled by \eqref{eq:S11-est1-L2}. 
Regarding the first term, we proceed as follows:
\begin{align*}
	& E\int_{\R}\norm{\cA_{\abs{\xi_0}^{\nu}}
	\left(\cI(Bh)\right)}_{L^2(\R^{d+1})}
	\, d\lambda
	\\ & \quad 
	\lesssim
	\left( \int_{\R}
	E \int_{\R}\norm{\cA_{\abs{\xi_0}^{\nu}}
	\left(\cI(Bh)\right)
	(t,\lambda)}_{L^2(\R^d)}^2
	\,dt\, d\lambda \right)^{1/2}
	\\ & \quad 
	=\left( \int_{\R}
	E \int_{\R} \int_{\R^d}\abs{\xi_0}^{2\nu}
	\abs{\widehat{\cI(Bh)}
	(\xi_0,\tilde\mxi,\lambda)}^2
	\, d\tilde \mxi\,d\xi_0
	\, d\lambda \right)^{1/2}
	\\ & \quad 
	=:\left( \int_{\R} A(\lambda)
	\, d\lambda \right)^{1/2},
\end{align*}
where
\begin{align*}
	A(\lambda)&=E \int_{\R^d}\int_{\R}
	\abs{\cF_\mx\Bigl(\abs{\xi_0}^{\nu} 
	\cF_t\bigl(\cI(Bh)\bigr)\Bigr)
	(\xi_0,\tilde\mxi,\lambda)}^2
	\, d\xi_0\, d\tilde \mxi
	\\ & =
	E\int_{\R^d}\int_{\R}
	\abs{\xi_0}^{2\nu} 
	\abs{\cF_t\bigl(\cI(Bh)\bigr)
	(\xi_0,\mx,\lambda)}^2
	\, d\xi_0\, d\mx
	\\ & \le
	E\int_{\R^d}\norm{\cI(Bh)
	(\cdot,\mx,\lambda)}_{H^{\nu,2}(\R)}^2\, d\mx
	\\ & 
	= E\left[
	\norm{\cI(Bh)(\cdot,\cdot,
	\lambda)}_{H^{\nu,2}(\R;L^2(\R^d))}^2\right]
	\\ & 
	\overset{\eqref{eq:frac-Sob-Ito-ver1}}{\lesssim}
	E\int_\R\norm{h(t,\cdot,\lambda)}_{L^2(\R^d)}^2\,dt,
	\qquad \forall \lambda \in \R.
\end{align*}
Consequently, 
\begin{align*}
	&E\int_{\R}\norm{\cA_{\abs{\xi_0}^{\nu}}
	\left(\cI(Bh)\right)}_{L^2(\R^{d+1})}\, d\lambda
	\lesssim\left( \int_{\R} A(\lambda)
	\, d\lambda \right)^{1/2}
	\lesssim 
	\norm{h}_{L^2_{\omega,t,\mx,\lambda}}.
\end{align*}

Combining this bound with the fractional Leibniz
estimate for $D_t^\nu(g\cI(Bh))$ and
\eqref{eq:stochastic-kernel-bounds} in \eqref{eq:S12-est1}, we obtain
$$
\mathcal S_{1,2}\lesssim_\nu
\frac{1}{\delta^{d+4}}2^{-(\nu+r-1)K}
\norm{h}_{L^2_{\omega,t,\mx,\lambda}}.
$$
Together with \eqref{eq:S11-est1-final}, this gives
\begin{equation}\label{eq:S1-finalest}
	\begin{split}
		\mathcal S_1&\mathrel{\le}
		\mathcal S_{1,1}+\mathcal S_{1,2}\\
		&\lesssim_\nu\frac{1}{\delta^{d+4}}
		\left(2^{-rK}+2^{-(\nu+r-1)K}\right)
		\norm{h}_{L^2_{\omega,t,\mx,\lambda}}.
	\end{split}
\end{equation}

The bounds \eqref{eq:BDG-Phi} and \eqref{eq:frac-Sob-Ito-ver1}
also apply to $F_{2,\lambda}=\cI((\pa_\lambda B)h)$ by
\eqref{eq:dp-def-noise-pa-B}. By Proposition~\ref{L1-multiplier}(i),
the estimates
\eqref{eq:S11-est1} and \eqref{eq:S12-est1} apply with
$(F_{1,\lambda},\psi_{1,1},\tilde\psi_{1,2})$ replaced by
$(F_{2,\lambda},\psi_{2,1},\tilde\psi_{2,2})$.
Together with the second row of \eqref{eq:stochastic-kernel-bounds},
these give
\begin{equation}\label{eq:S2-finalest}
	\begin{split}
		\mathcal S_2&\mathrel{\le}
		\mathcal S_{2,1}+\mathcal S_{2,2}\\
		&\lesssim_\nu\frac{1}{\delta^{d+3}}
		\left(2^{-K}+2^{-\nu K}\right)
		\norm{h}_{L^2_{\omega,t,\mx,\lambda}}.
	\end{split}
\end{equation}
Combining \eqref{eq:S1-finalest} and \eqref{eq:S2-finalest},
using $0<\delta<1$ and $0<\nu<1/2$, gives
\begin{equation}\label{stoch-integral-1}
	\mathcal S_1+\mathcal S_2
	\lesssim_\nu\frac{1}{\delta^{d+4}}
	\left(2^{-\nu K}+2^{-(\nu+r-1)K}\right)
	\norm{h}_{L^2_{\omega,t,\mx,\lambda}}.
\end{equation}
For the $H^{\gamma,1}$ estimate, we multiply $\Theta_J$ and
$\mathscr P$ by $(1+\abs{\mxi}^2)^{\gamma/2}$.
By \eqref{eq:weighted-kernel}, this increases their inverse Fourier
$L^1$ bounds by at most a factor $C2^{\gamma K}$.
The temporal factors $v_0,v_{\nu,K}$ are unchanged.
Thus multiplying the right-hand sides of
\eqref{eq:stochastic-kernel-bounds} by $2^{\gamma K}$ gives the $L^1$
bounds used in \eqref{final-regular} below.

\medskip

\noindent\underline{Optimization of parameters.}

\medskip

\noindent We now aim to optimize the parameters 
$r$ and $\epsilon$ to minimize the exponent of $K$ in 
the estimates \eqref{beta}, \eqref{2}, \eqref{g-term-1} 
(deterministic parts), and 
\eqref{stoch-integral-1} (stochastic part). 
This optimization will yield the best possible regularity 
result, as will be explained below. Therefore, for 
$r, \epsilon \in (0,1)$ satisfying \eqref{eq:r-eps-cond}, 
our goal is to maximize
\begin{equation*}
	\min\Bigl\{1-\epsilon, \frac{\epsilon}{2},r,\nu+r-1\Bigr\} .
\end{equation*}
Since $1-\epsilon=r$ by \eqref{eq:r-eps-cond},
and $r\geq\nu+r-1$ for $\nu<1/2$, the maximum is attained
when $\epsilon/2=(1-r)/2=\nu+r-1$.
This results in
\begin{equation}\label{eq:r-eps-chosen}
	r=1-\frac{2}{3}\nu , \qquad
	\epsilon=\frac{2}{3}\nu.
\end{equation} 
These values are permissible, with the optimal value 
being $\frac{\nu}{3}$. Therefore, in each of the estimates
\eqref{beta}, \eqref{2}, \eqref{g-term-1}, 
and \eqref{stoch-integral-1}, the 
exponent can be increased by up to $\frac{\nu}{3}$. 
Given that $\nu$ may assume any value within
$(0,\frac{1}{2})$, we deduce that the exponent can be 
increased by any $\gamma\in (0,\frac{1}{6})$.

Therefore, from \eqref{isolated-h-0}, by combining \eqref{beta}, 
\eqref{2}, \eqref{g-term-1}, 
and \eqref{stoch-integral-1},
using the weighted estimate
\eqref{eq:weighted-kernel},
for any $\gamma\in (0,\frac{1}{6})$
by choosing $\nu= \frac{3}{2}\gamma 
+\frac{1}{4}\in (3\gamma,\frac{1}{2})$, which fixes 
$r$ and $\epsilon$ by \eqref{eq:r-eps-chosen},
we get the following estimate in the Bessel potential 
space $H^{\gamma,1}(\R^{d+1})$:
\begin{align}
	& E\left[ \norm{\int_\R \mathcal{A}_{(1-\Psi_0)
	\tilde\Phi_\delta(\cdot,\lambda)}
	\bigl(\rho(\lambda)h(\cdot,\lambda)\bigr) 
	\, d\lambda}_{{H}^{\gamma,1}(\R^{d+1})}\right] 
	\notag 
	\\ & \quad
	\overset{\eqref{functions}}{=} 
	E\left[\norm{\sum_{K\geq 1} \int_{\R} 
	{\cal A}_{(1+|\mxi|^2)^{\frac{\gamma}{2}}
	\Psi_K \tilde{\Phi}_{\delta}(\cdot,\lambda)}
	\bigl( \rho(\lambda)h(\cdot,\lambda)\bigr) 
	\, d\lambda}_{L^1(\R^{d+1})}\right] 
	\notag \\ & \quad
	\lesssim 
	\frac{1}{\delta^{d+4}}
	\sum_{K\ge 1}\Biggl( 
	2^{-\left(1-\epsilon-\gamma\right) K}
	+2^{-\left(\frac{1}{2}\epsilon-\gamma\right)K}
	+2^{-\left(r-\gamma\right) K} 
	\notag\\ & \qquad\qquad\qquad\qquad \qquad
	+2^{-(\nu-\gamma) K}+2^{-(\nu+r-1-\gamma)K}
	\Biggr) A
	\notag \\ & \quad 
	\lesssim \frac{1}{\delta^{d+4}} A,
	\label{final-regular}
\end{align}
where
\begin{align*}
	A=A(h,\beta,\zeta) & :=
	\norm{h}_{L^1(\Omega\times\R^{d+2})}
	+\norm{\beta}_{L^1_{w\star}
(\Omega\times \R^{d+1};\cM(\R)^d)}
	\\ & \qquad\qquad 
	+\norm{\zeta}_{L^1(\Omega;\cM(\R^{d+2}))} 
	+ \norm{h}_{L^2(\Omega\times\R^{d+2})}.
\end{align*}
Here, we used that $1-\epsilon=r\geq\frac{1}{3}\nu$, 
$\frac12\epsilon=\frac{1}{3}\nu$, 
and $\nu + r - 1 = \frac{\nu}{3}$, and 
noticed that the exponents of the five
$2^{-\left( \, \cdots\right)}$-terms are strictly negative because of the 
restriction $\nu > 3\gamma$. 
The implicit constant in the inequality $\lesssim$ 
depends on $f$, $a$, $\rho$, and $\gamma$. 

The function $1-\Psi_0$, which is supported 
away from the origin, is introduced into \eqref{final-regular} 
to account for the fact that the equation in \eqref{isolated-h-0} 
holds for $K\geq 1$ but not $K=0$. Of course, the choice of 
the cutoff does not affect estimate \eqref{final-regular}.
In particular, when applying this estimate in the upcoming 
Step V (interpolation), we will use $1-\Psi_0-\Psi_1$ 
in place of $1-\Psi_0$.

\medskip

\noindent\textbf{Step IV.} 

\medskip

\noindent We now estimate the part where the modified
symbol $\sum_{J\ge1}\Psi_J\mathcal L_J$ is close to zero, using
\eqref{non-deg-stand}. The low-frequency term
$\cA_{\Psi_0+\Psi_1}\int_\R\rho(\lambda)h(\cdot,\lambda)\,d\lambda$
is smooth because $\Psi_0+\Psi_1$ is smooth and compactly supported.
It therefore remains to estimate the high-frequency part, for
which we set $\tilde\Psi:=1-\Psi_0-\Psi_1$.
By \eqref{eq:Phidelta}, the support of
$1-\tilde\Phi_\delta$ lies in the sublevel set
$\sum_{J\ge1}\Psi_J\mathcal L_J\le2\delta$.
This set may be unbounded in $\mxi$. The Plancherel estimate
\eqref{aroundzeroA} below only uses
$\meas\widetilde\Lambda_{2\delta}(\mxi)$, with
$\widetilde\Lambda_{2\delta}$ defined in
\eqref{eq:def-tilde-Lambda} below.

Let us recall that in the previous step, we chose and fixed the 
parameters $\gamma$, $\nu$, $r$, and $\epsilon$. Using the 
Plancherel formula and the H{\"o}lder inequality, 
we can easily obtain
\begin{equation}\label{aroundzeroA}
	\begin{split}
		&E\left[\, \norm{\int_{\R} {\cal A}_{\tilde\Psi
		(1-\tilde{\Phi}_{\delta}(\cdot,\lambda))}
		\bigl(\rho(\lambda)h(\cdot,\lambda)\bigr) 
		\, d\lambda}^2_{L^2(\R^{d+1})}\, \right]
		\\ & \qquad 
		=E\left[ \, \norm{\tilde\Psi \, \int_{\R}
		(1-\tilde{\Phi}_\delta(\cdot,\lambda))
		\rho(\lambda)\widehat{h}(\cdot,\lambda)
		\, d\lambda}^2_{L^2(\R^{d+1})}\, \right]
		\\ & \quad\quad 
		\lesssim_\rho \sup_{|\mxi|\geq 2}
		\meas\widetilde\Lambda_{2\delta}(\mxi) 
		\, E\norm{h}^2_{L^2(\R^{d+2})} ,
	\end{split}
\end{equation} 
where
\begin{equation}\label{eq:def-tilde-Lambda}
	\widetilde\Lambda_{2\delta}(\mxi)
	:=\seq{\lambda\in \supp(\rho) \, :\, 
	\sum\limits_{J\in \N} \Psi_J(\mxi) 
	{\cal L}_J(\mxi,\lambda) \leq 2\delta}.
\end{equation}
We will demonstrate that the measure of this set 
can be controlled using the non-degeneracy 
condition \eqref{non-deg-stand}. 

Let us fix $\mxi\in\R^{d+1}$ such that $|\mxi|\geq 2$. 
Then there exists $K\geq 2$ such that $\Psi_J(\mxi)=0$ 
for any $J\in\N_0\setminus\{K-1,K,K+1\}$ 
(see \eqref{functions}). In particular, 
$|\mxi|\geq 2^{K-2}$. Hence, for any 
$\lambda\in\supp(\rho)$, we have (recall the 
notations in \eqref{eq:xi-prime-def})
\begin{equation}\label{eq:nondeg-estimate}
	\begin{aligned}
		&\sum_{J\in \N} \Psi_J(\mxi){\cal L}_J(\mxi,\lambda)
		\\ & \qquad 
		= \big(\xi_0'+\bigl\langle f(\lambda)
		\,|\,\tilde\mxi'\bigr\rangle\big)^2
		\sum_{J=K-1}^{K+1} \Psi_J(\mxi)\Phi_J(\mxi,\lambda) 
		+ \bigl\langle a(\lambda){\tilde{\mxi}'}
		\,|\, \tilde{\mxi}' \bigr\rangle
		|\mxi|^{\frac{2}{3}\nu} 
		\\ & \qquad 
		\geq \big(\xi_0'+\bigl\langle f(\lambda)\,|\,\tilde\mxi'
		\bigr\rangle\big)^2 \Phi_{K+1}(\mxi,\lambda)
		\sum_{J=K-1}^{K+1} \Psi_J(\mxi) 
		+ 2^{\frac{2}{3}\nu(K-2)}
		\bigl\langle a(\lambda){\tilde{\mxi}'} 
		\,|\, {\tilde{\mxi}'} \bigr\rangle
		\\ & \qquad 
		\geq \big(\xi_0'+\bigl\langle f(\lambda)
		\,|\,\tilde{\mxi}'\bigr\rangle\big)^2 
		\Phi_{K+1}(\mxi,\lambda)
		+2^{\frac{2}{3}\nu(K-2)}
		\bigl\langle a(\lambda){\tilde{\mxi}'} 
		\,|\, {\tilde{\mxi}'}\bigr\rangle.
	\end{aligned}
\end{equation}
Now, we divide the analysis into two parts, depending 
on the value of $\lambda$. 

If $\bigl\langle a(\lambda){\tilde{\mxi}'} 
\,|\, {\tilde{\mxi}'}\bigr\rangle
>2^{-\frac{2}{3}\nu(K+1)}$, then 
$\sum_{J\in \N} \Psi_J(\mxi){\cal L}_J(\mxi,\lambda) 
\geq 4^{-\nu}>1/2$ (recall that $\nu<1/2$). 
Therefore, for $\delta<1/8$, 
we have 
\begin{equation}\label{eq:Lambda1}
	\widetilde\Lambda_{2\delta}(\mxi)\cap 
	\seq{\lambda\in\supp(\rho)\, :\, 
	\bigl\langle a(\lambda){\tilde{\mxi}'} 
	\,|\, {\tilde{\mxi}'} \bigr\rangle
	>2^{-\frac{2}{3}\nu(K+1)}}
	=\emptyset.
\end{equation}

On the other hand, if 
$\bigl\langle a(\lambda){\tilde{\mxi}'} 
\,|\, {\tilde{\mxi}'}\bigr\rangle
\leq 2^{-\frac{2}{3}\nu(K+1)}$, we have 
$\Phi_{K+1}(\mxi,\lambda)=1$.
Thus, using \eqref{eq:nondeg-estimate}, we deduce that
\begin{equation}\label{eq:Lambda2}
	\begin{aligned}
		&\widetilde\Lambda_{2\delta}(\mxi)
		\cap \seq{\lambda\in\supp(\rho) \, :\, 
		\bigl\langle a(\lambda){\tilde{\mxi}'} 
		\,|\, {\tilde{\mxi}'} \bigr\rangle
		\leq 2^{-\frac{2}{3}\nu(K+1)}}
		\\ & \quad \subseteq 
		\seq{\lambda\in\supp(\rho) \,:\, 
		\big(\xi_0'+\bigl\langle f(\lambda)\,|\,\tilde{\mxi}'
		\bigr\rangle\big)^2
		+\bigl\langle a(\lambda){\tilde{\mxi}'} 
		\,|\, {\tilde{\mxi}'}\bigr\rangle 
		\leq 2\delta}
		\\ & \quad =:\Lambda_{2\delta}(\mxi').
	\end{aligned}
\end{equation}
Therefore, in view of 
\eqref{eq:Lambda1} and \eqref{eq:Lambda2}, we obtain
$$
\sup_{|\mxi|\geq 2}\meas
\widetilde\Lambda_{2\delta}(\mxi)
\leq \sup_{|\mxi|=1}\meas
\Lambda_{2\delta}(\mxi)
\overset{\eqref{non-deg-stand}}{\lesssim}
\delta^\alpha.
$$

Referring back to \eqref{aroundzeroA}, 
when $\delta< 1/8$, we obtain
\begin{equation}\label{aroundzeroB}
	\left(E\left[ \, \norm{\int_{\R} 
	{\cal A}_{\tilde\Psi (1-\tilde{\Phi}_{\delta}(\cdot,\lambda))}
	\bigl(\rho(\lambda)h(\cdot,\lambda)\bigr) 
	\, d\lambda}^2_{L^2(\R^{d+1})} \, \right]\right)^{1/2}
	\lesssim_\rho \delta^{\alpha/2} \tilde{A}(h), 
\end{equation} 
where $\tilde{A}=\tilde{A}(h):=
\norm{h}_{L^2(\Omega\times\R^{d+2})}$.

\medskip

\noindent\textbf{Step V.}\par\nopagebreak[4]

\medskip

\noindent Given \eqref{final-regular} and \eqref{aroundzeroB}, 
we are ready to apply the $K$-interpolation method 
(see, for instance, \cite[Appendix C]{Hytonen:2016aa}). 
This method constructs intermediate spaces between 
two Banach spaces. This is done via 
the Peetre $K$-functional.  The procedure is 
quite standard (it is already used in this 
context, for example, in \cite[p.~1496]{Tadmor:2006vn}; 
see also \cite[pp.~2516-2518]{Gess:2018ab}), but for 
completeness and to ensure the correct order of 
the smoothing effect, we present the 
argument in some detail. 

The relevant functions for decomposition 
in our case are the velocity averages:
$$
H(\omega,t,\mx)
:=\int_{\R} {\cal A}_{\tilde\Psi}\bigl(\rho(\lambda) 
h(\omega,\cdot,\cdot,\lambda)\bigr)(t,\mx) 
\,d\lambda,
$$
where we recall that 
$\tilde\Psi=1-\Psi_0-\Psi_1$ 
(see \eqref{functions}). The so-called 
$K$-functional of $H$ is then defined by
$$
\mathcal{K}(\tau,H) := 
\inf\limits_{\substack{H_0\in L^2(\Omega\times\R^{d+1}),
\\ 
\, H_1\in L^1(\Omega; B^{\eta}_{1,1}(\R^{d+1})),
\\ 
H_0+H_1=H}}\Bigl( \|H_0\|_{L^2(\Omega\times\R^{d+1})}
+\tau\|H_1\|_{L^1(\Omega; B^{\eta}_{1,1}(\R^{d+1}))}\Bigr),
$$
where $\tau>0$, $\eta\in (0,\frac{1}{6})$, 
and $B^{\eta}_{1,1}(\mathbb{R}^{d+1})$ 
represents the $L^1$-based Besov space of functions 
whose smoothness is characterized by the parameter $\eta > 0$ 
(see \cite[Subsection 14.4]{Hytonen:2023aa}).
For a given fixed $\eta$, select $\gamma\in 
(\eta, 1/6)$. To effectively apply the estimate 
\eqref{final-regular}, it is crucial to recall 
the embedding $H^{\gamma,1}(\R^{d+1}) 
\hookrightarrow B^{\eta}_{1,1}(\R^{d+1})$. 
This result can be found in 
\cite[Proposition 14.4.18 and 
Theorem 14.4.19]{Hytonen:2023aa}.

By choosing $H_1=0$ (and $H_0=H$), we have a trivial 
estimate of $\mathcal{K}(\tau, H)$ 
that is independent of $\tau$:
\begin{equation*}
	\mathcal{K}(\tau,H) 
	\leq \norm{H}_{L^2(\Omega\times\R^{d+1})}
	\lesssim_\rho \norm{h}_{L^2(\Omega\times\R^{d+2})} 
	=\tilde{A} .
\end{equation*}
Thus, for any $\theta>0$,
\begin{equation}\label{eq:calK1}
	\sup_{\tau\in \left(C\frac{\tilde{A}}{A},\infty\right)} 
	\abs{\tau^{-\theta}\mathcal{K}(\tau,H)}
	\lesssim \left(\frac{A}{\tilde{A}}\right)^\theta 
	\tilde{A} = A^\theta \tilde{A}^{1-\theta},
\end{equation}
where $A$ and $\tilde A$ are provided at the end 
of Steps III and IV, respectively, and $C>0$ is
a fixed constant that we specify below.

On the other hand, for 
$\tau\in \left(0, C\frac{\tilde{A}}{A}\right)$, 
we take
$$
\delta^{\frac{\alpha}{2}+d+4}
:= \frac{\tau A}{\tilde{A}},
$$
where $C>0$ is specified such that $\delta< 1/8$, as 
it is required in Step IV (note that $C$ 
depends only on $\alpha$ and $d$). 
Thus, by \eqref{final-regular} 
and \eqref{aroundzeroB}, for
\begin{equation*}
	\theta_*=\frac{\alpha}{\alpha+2(d+4)} 
	\in (0,1),
\end{equation*}
we have
\begin{align*}
	\abs{\tau^{-\theta_*} \mathcal{K}(\tau,H)}
	& \leq 
	\tau^{-\theta_*}\norm{\int_{\R} 
	{\cal A}_{\tilde\Psi(1-\tilde{\Phi}_{\delta}(\cdot,\lambda))}
	\bigl(\rho(\lambda)h(\cdot,\lambda)\bigr)
	\, d\lambda}_{L^2(\Omega\times\R^{d+1})}
	\\ & \qquad 
	+ \tau^{1-\theta_*}
	\norm{\int_\R {\cal A}_{\tilde\Psi
	\tilde\Phi_\delta(\cdot,\lambda)}
	\bigl(\rho(\lambda)h(\cdot,\lambda)\bigr)
	\, d\lambda}_{L^1(\Omega;B^{\eta}_{1,1}(\R^{d+1}))}
	\\ & 
	\lesssim_{a,f,\rho,\eta}
	\tau^{-\theta_*}\delta^{\frac{\alpha}{2}} \tilde{A} 
	+ \tau^{1-\theta_*}\frac{A}{\delta^{d+4}}
	=2 A^{\theta_*}\tilde{A}^{1-\theta_*}.
\end{align*}
Since the right-hand side does not depend on $\tau$, we get
\begin{equation*}
	\sup_{\tau\in \left(0,C\frac{\tilde{A}}{A}\right)}
	\abs{\tau^{-\theta_*} \mathcal{K}(\tau,H)} 
	\lesssim_{a,f,\rho,\eta} A^{\theta_*}
	\tilde{A}^{1-\theta_*}.
\end{equation*}
This, combined with \eqref{eq:calK1}, yields 
\begin{equation}\label{eq:K-endpoint}
	\sup_{\tau>0}
	\abs{\tau^{-\theta_*}\mathcal{K}(\tau,H)} 
	\lesssim_{a,f,\rho,\eta} 
	A^{\theta_*}\tilde{A}^{1-\theta_*}.
\end{equation}

Fix $0<s<s_*=\theta_*/6$, and put $D=d+1$. Choose
$$
\frac{s}{\theta_*}<\eta<\frac16,
\qquad
\frac{s}{\eta}<\theta<\theta_*.
$$
The bounds \eqref{final-regular} and \eqref{aroundzeroB}
apply with this $\eta$, by choosing
$\gamma\in(\eta,1/6)$ in Step III. Set
$$
p=\frac{2}{1+\theta},
\qquad s_1=\eta\theta.
$$
Then $s_1>s$ and $p>q_*=2/(1+\theta_*)$. Let
$$
X_0=L^2(\Omega\times\R^D),
\qquad
X_1=L^1\bigl(\Omega;B^\eta_{1,1}(\R^D)\bigr).
$$
By \eqref{eq:K-endpoint} and
$\mathcal K(\tau,H;X_0,X_1)\le\norm{H}_{X_0}$,
$$
\mathcal K(\tau,H;X_0,X_1)
\leq \min\seq{N,M\tau^{\theta_*}},
$$
where $N=\norm{H}_{X_0}<\infty$ and
$M\lesssim A^{\theta_*}\tilde A^{1-\theta_*}<\infty$.
Consequently,
\begin{align*}
	&\int_0^\infty
	\bigl[\tau^{-\theta}\mathcal K(\tau,H;X_0,X_1)\bigr]^p
	\,\frac{d\tau}{\tau}
	\\
	&\quad\leq
	M^p\int_0^1\tau^{p(\theta_*-\theta)}\,\frac{d\tau}{\tau}
	+N^p\int_1^\infty\tau^{-p\theta}\,\frac{d\tau}{\tau}
	\\
	&\quad=
	\frac{M^p}{p(\theta_*-\theta)}
	+\frac{N^p}{p\theta}<\infty.
\end{align*}
Thus $H\in(X_0,X_1)_{\theta,p}$. The real-interpolation identity
for Bochner spaces \cite[Theorem 2.2.10]{Hytonen:2016aa}
(interchanging the endpoints to match the ordering of its exponents),
combined with the Besov-space identity
\cite[Theorem 6.4.5(3)]{Bergh:1976aa}, now gives
\begin{align*}
	(X_0,X_1)_{\theta,p}
	&=L^p\bigl(\Omega;B^{s_1}_{p,p}(\R^D)\bigr)
	\\
	&=L^p\bigl(\Omega;W^{s_1,p}(\R^D)\bigr),
\end{align*}
with equivalent norms. Here we used
$L^2(\R^D)=B^0_{2,2}(\R^D)$ and $0<s_1<1$.

Since $p>q_*$ and $s_1>s$, localization and real interpolation
yield the continuous restriction
$$
W^{s_1,p}(\R^D)\longrightarrow W^{s,q_*}(Q)
$$
for every bounded Lipschitz domain
$Q\subset\subset(0,\infty)\times\R^d$, see
\cite[Theorems 14.4.31 and 14.4.19(i),
and Corollary 14.4.25]{Hytonen:2023aa}.
Since $P(\Omega)=1$, H\"older's inequality in $\omega$ gives
$$
H\in L^{q_*}\bigl(\Omega;W^{s,q_*}(Q)\bigr).
$$
Finally, the low-frequency part
$\overline h-H=\cA_{\Psi_0+\Psi_1}\overline h$
belongs to $L^2(\Omega;H^m(\R^D))$ for every $m\geq0$,
because $\overline h\in L^2(\Omega\times\R^D)$ and
$\Psi_0+\Psi_1$ is smooth and compactly supported.
Since $2>q_*$ and $m>s$ can be chosen, the restriction and
H\"older estimates above also give
$\overline h-H\in L^{q_*}(\Omega;W^{s,q_*}(Q))$. This proves
Theorem \ref{t-regularity} for every $s<s_*$.
\end{proof}

\section{Discussion}
\label{sec:discussions}

In this final section, we will discuss the regularity result 
and its underlying assumptions in the deterministic case 
(where $B \equiv 0$), and also show that it applies to 
a scenario not covered by \cite{Tadmor:2006vn}.

First, if $B\equiv 0$, in the final part of 
Step III of the proof of Theorem \ref{t-regularity}, it 
is optimal to select $r=\frac{1}{3}$ and 
$\epsilon=\frac{2}{3}$. This modification improves
the regularity exponent by a factor of 2. Specifically, 
in Theorem \ref{t-regularity}, instead of $s_*$, 
we would achieve $2s_*$.

\subsection*{(i) Non-degeneracy condition and Step IV} 
In Step IV, we demonstrated that the non-degeneracy 
condition \eqref{non-deg-stand} is 
sufficient to guarantee that (see 
\eqref{eq:def-tilde-Lambda} for the definition 
of $\widetilde\Lambda_{\delta}$)
\begin{equation}\label{eq:tilde-lambda-delta}	
	\sup_{|\mxi|\geq 2}\meas 
	\widetilde\Lambda_{\delta}(\mxi)
	\lesssim \delta^\alpha,
\end{equation}
which is essential for the argument. 
For a fixed $\rho$, the converse gives \eqref{non-deg-stand}
with $I$ replaced by $\supp\rho$. Assume that there exist $\alpha>0$
and $\delta_0>0$ such that \eqref{eq:tilde-lambda-delta} 
is satisfied for $\delta\in (0,\delta_0)$. 
Consider $\meta\in\R^{d+1}$ where
$\abs{\meta}=1$, and let $\delta\in 
\bigl(0,\min\seqb{2^{-\frac{2}{3}}\delta_0,
\frac{1}{4}}\bigr)$. Then 
(see \eqref{eq:Lambda2} for the 
definition of $\Lambda_\delta$)
$$
\Lambda_\delta(\meta) 
= \Lambda_\delta(\meta) 
\cap \seq{\lambda\in\supp(\rho) 
:\bigl\langle a(\lambda)\tilde\meta
\,|\,\tilde\meta \bigr \rangle 
\leq \frac{1}{4}},
$$
i.e., for any $\lambda\in\Lambda_\delta(\meta)$ it holds
$\langle a(\lambda)\tilde\meta\,|\,\tilde\meta \bigr \rangle 
\leq \frac{1}{4}$.

Let us take $\lambda\in\Lambda_\delta(\meta)$ 
and define $\mxi:=2\meta$. Hence, $\mxi'=\meta$.
We have
\begin{align*}
	&\sum_{J\in \N} \Psi_J(\mxi) 
	{\cal L}_J(\mxi,\lambda) 
	\\ & \quad =\big(\xi_0'+ \bigl\langle f(\lambda)
	\,|\,\tilde{\mxi}' \bigr\rangle\big)^2 
	\sum_{J=1}^{3} \Psi_J(\mxi)\Phi_J(\mxi,\lambda) 
	+ \bigl\langle a(\lambda){\tilde{\mxi}'} 
	\,|\, {\tilde{\mxi}'} \bigr\rangle 2^{\frac{2}{3}\nu} 
	\\ & \quad = \big(\xi_0'+\bigl\langle f(\lambda)
	\,|\,\tilde{\mxi}'\bigr\rangle\big)^2 
	\sum_{J=1}^{3} \Psi_J(\mxi) 
	+ \bigl\langle a(\lambda){\tilde{\mxi}'} 
	\,|\, {\tilde{\mxi}'} \bigr\rangle 2^{\frac{2}{3}\nu} 
	\\ & \quad \leq 2^{\frac{2}{3}} 
	\Bigl(\big(\xi_0'+\bigl\langle f(\lambda)
	\,|\,\tilde{\mxi}'\bigr\rangle\big)^2
	+\bigl\langle a(\lambda){\tilde{\mxi}'} 
	\,|\, {\tilde{\mxi}'} \rangle \Bigr),
\end{align*}
where we used that $\Phi_J(\mxi,\lambda)=1$, $J\in \{1,2,3\}$, 
since $\langle a(\lambda)\tilde\mxi'\,|\,\tilde\mxi' \bigr \rangle
=\langle a(\lambda)\tilde\meta\,|\,\tilde\meta \bigr \rangle 
\leq \frac{1}{4}$, and $\nu<1$.
Therefore, we get 
$$
\meas\Lambda_\delta(\meta) 
\leq \meas\widetilde\Lambda_{2^{2/3}\delta}(2\meta) 
\leq \sup_{|\mzeta|\geq 2}\meas
\widetilde\Lambda_{2^{2/3}\delta}(\mzeta)
\lesssim \delta^\alpha.
$$
Since the constant on the right-hand side does 
not depend on $\meta$, by taking the 
supremum over $|\meta|=1$ we get the claim. 

\subsection*{\texorpdfstring{(ii) Homogeneous versus
nonhomogeneous symbols}{(ii) Homogeneous versus nonhomogeneous symbols}}

The symbol in \eqref{non-deg-stand} is homogeneous 
with respect to $\mxi$ (of order $2$). We chose to 
present the result with this form of the non-degeneracy 
condition because, with our method---specifically, using 
the additional condition \eqref{regularity}---we can 
transform an inhomogeneous second-order 
equation into a homogeneous one. In our view, this 
is an important point to emphasize, as techniques for 
studying homogeneous and inhomogeneous equations, 
such as those in \cite{Tadmor:2006vn}, 
often differ significantly.

However, as is evident from Corollary \ref{TT-ex} below, 
the form of the symbol in \eqref{non-deg-stand} is not optimal. 
Specifically, replacing 
$\langle (1,f(\lambda))\,|\,\mxi\rangle^2$ 
with $\abs{\langle (1,f(\lambda))\,|\,\mxi\rangle}$ 
would be preferable---and indeed, this modification is 
possible. More precisely, Theorem \ref{t-regularity} 
remains valid even with this adjustment to the 
non-degeneracy condition \eqref{non-deg-stand}. 

To see this, one only needs to remove 
$i\bigl(\xi_0'+\bigl\langle f(\lambda)
\,|\,\tilde{\mxi}' \bigr\rangle\bigr)$ 
in \eqref{eq:second_eg_mult}, and change 
$\mathcal{L}_J$, $\tilde\Phi_\delta$, 
and $\widetilde\Lambda_\delta(\mxi)$ to 
\begin{align*}
	& {\cal L}_J(\mxi,\lambda) 
	=i\bigl(\xi_0'+\bigl\langle f(\lambda)
	\,|\,\tilde{\mxi}'\bigr\rangle\bigr) 
	\, \Phi_J(\mxi,\lambda) 
	+\bigl\langle a(\lambda)\tilde{\mxi}'
	\,|\,\tilde{\mxi}'\bigr\rangle|\mxi|^{1-r},
	\\
	& \tilde\Phi_\delta(\mxi,\lambda)
	= \tilde\Phi\left(\frac{\abs{\sum_{J\in\N}
	\Psi_J(\mxi) \mathcal{L}_J(\mxi,\lambda)}^2}
	{\delta^2}\right),
	\\ 
	& \widetilde\Lambda_{\delta}(\mxi)
	=\seq{\lambda\in \supp(\rho)
	\,: \, \abs{\sum_{J\in \N} \Psi_J(\mxi)
	{\cal L}_J(\mxi,\lambda)}^2\leq \delta^2} ,
\end{align*}
and follow the proof of 
Theorem \ref{t-regularity} after that.

\subsection*{(iii) Our non-degeneracy 
condition versus
\texorpdfstring{\cite{Tadmor:2006vn}}{Tadmor and Tao}} 
In the regularity result for general symbols presented 
in \cite{Tadmor:2006vn}, the non-degeneracy conditions 
are composed of two parts: formulas (2.19) and (2.20) from the 
aforementioned paper. Our condition \eqref{non-deg-stand} 
can be related to the first one, which states that there 
exist positive constants $\alpha$ and $\beta$ and a finite 
interval $I\subseteq\R$ such that for every integer 
$J\ge 1$ and any positive number $\delta>0$, it holds that
\begin{equation}\label{c-TT}
	\sup\limits_{|\mxi|\sim J} 
	\meas\seq{\lambda \in I \, : \, 
	\abs{\xi_0+\bigl\langle f(\lambda) \,|\, \tilde{\mxi}
	\bigr\rangle}^2
	+\bigl\langle a(\lambda)\tilde{\mxi}
	\,|\, \tilde{\mxi}\bigr\rangle^2 \leq\delta^2}
	\lesssim \left(\frac{\delta}{J^\beta}
	\right)^\alpha. 
\end{equation} 

It is straightforward to see that \eqref{c-TT} 
implies \eqref{non-deg-stand}. To demonstrate this, consider 
any $\delta\in (0,1)$ and let $\mxi\in\S^d$ be 
arbitrary. We select $J\in\N$ such that \eqref{c-TT} 
holds. Given that $\delta<1$, it follows from
$\abs{\xi_0+\bigl\langle f(\lambda) \,|\, \tilde{\mxi}
\bigr\rangle}^2
+\bigl\langle a(\lambda)\tilde{\mxi}
\,|\, \tilde{\mxi}\bigr\rangle^2 \leq\delta^2$ that
\begin{align*}
	&\abs{\xi_0+\bigl\langle f(\lambda) 
	\,|\, \tilde{\mxi}\bigr\rangle}^2
	+\bigl\langle a(\lambda)\tilde{\mxi}
	\,|\, \tilde{\mxi}\bigr\rangle
	\geq \abs{\bigl\langle (1,f(\lambda)) 
	\,|\, \mxi\bigr\rangle}^2
	+\bigl\langle a(\lambda)\tilde{\mxi}
	\,|\, \tilde{\mxi}\bigr\rangle^2 
	\\ & \qquad 
	\geq \frac{1}{J^4}\left(\abs{\bigl\langle (1,f(\lambda)) 
	\,|\, (J\mxi) \bigr\rangle}^2
	+\bigl\langle a(\lambda)(J\tilde{\mxi})
	\,|\, (J\tilde{\mxi})\bigr\rangle^2\right).
\end{align*}
Thus,
\begin{align*}
	& \meas\seq{\lambda \in I \, :\, 
	\abs{\xi_0+\bigl\langle f(\lambda) 
	\,|\, \tilde{\mxi}\bigr\rangle}^2
	+\bigl\langle a(\lambda)\tilde{\mxi}
	\,|\, \tilde{\mxi}\bigr\rangle 
	\leq\delta}
	\\ & \quad \leq 
	\meas\seq{\lambda \in I \, :\, 
	\abs{\bigl\langle (1,f(\lambda)) 
	\,|\, (J\mxi)\bigr\rangle}^2
	+\bigl\langle a(\lambda)(J\tilde{\mxi})
	\,|\, (J\tilde{\mxi})\bigr\rangle^2 
	\leq J^4\delta} \lesssim \delta^{\alpha/2}.
\end{align*} 
This ensures that \eqref{non-deg-stand} is satisfied 
with coefficient $\alpha/2$, where the division by 
2 is a consequence of the homogeneous form of the symbol 
in \eqref{non-deg-stand} (see part (ii) of this section).

On the other hand, when $\delta$ is small, one can 
observe that \eqref{non-deg-stand} 
implies \eqref{c-TT} with the same 
$\alpha$ and $\beta = 2$. It is preferable for 
$\beta$ to match the order of the equation 
(here $\beta=2$ is achieved precisely because of 
the homogeneous nature of the non-degeneracy 
condition \eqref{non-deg-stand}). 
However, in \cite[Averaging Lemma 2.3]{Tadmor:2006vn}, it becomes 
necessary to also consider larger values of $\delta$ 
in \eqref{c-TT}, which is in contrast to the 
homogeneous setting discussed in Averaging Lemma 2.1 
of the same paper. Consequently, for a general $\delta > 0$, 
deriving from \eqref{non-deg-stand}, we find that 
on the right-hand side of \eqref{c-TT} we obtain 
$\bigl(\frac{\delta}{J^2}\bigr)^\alpha(1 + \delta^\alpha)$. 
This expression, however, proves insufficient for 
achieving the regularity result based 
on \cite[Averaging Lemma 2.3]{Tadmor:2006vn}.

To summarize, Theorem \ref{t-regularity} enables 
us to obtain a regularity result with more 
relaxed assumptions compared to 
\cite[Averaging Lemma 2.3]{Tadmor:2006vn} (see
also \cite[Theorem 3.1]{Gess:2018ab}). This is illustrated 
in Corollary \ref{TT-ex} below. Apart from the 
discussion above on the connection between 
\eqref{non-deg-stand} and \eqref{c-TT}, we do not 
need any information about the behavior of the 
symbol’s derivative with respect to the velocity 
variable $\lambda$, as needed in \cite[(2.20)]{Tadmor:2006vn}. 
We note that this last condition is 
not an insignificant restriction. While it can always be 
satisfied with suitable parameter choices (as 
discussed in \cite[Remark 2.5]{Tadmor:2006vn}), it is 
possible that the resulting regularity is zero. 
In our kinetic equations, this usually 
occurs when we must set $\beta=1$ in \eqref{c-TT}.  

\subsection*{(iv) The regularity exponents are not optimal.} 
The regularity order provided by 
Theorem \ref{t-regularity} is clearly not optimal and is inferior to 
\cite[Averaging Lemma 2.3]{Tadmor:2006vn}
and \cite[Theorem 3.1]{Gess:2018ab} (when the latter results 
are applicable; see part (iii) of this section). 
Our regularity exponent $s$ is 
significantly lower and dimension-dependent, compared to 
the results presented in the aforementioned papers.
The reason for this is that Step III of the proof of 
Theorem \ref{t-regularity} relies on more conventional 
``$L^1$ theory" of Fourier multiplier operators 
(as seen in Lemmas \ref{L-ocjena-L1} and \ref{l-meas}). It does not 
utilize the ``truncation property" introduced 
in \cite[Section 2.1]{Tadmor:2006vn}.  
Since our main objective was to achieve some 
degree of regularity in cases that were not addressed 
in previous studies, we chose to follow this method.

\subsection*{(v) An application} 
\leavevmode Tadmor and Tao \cite[Corollary 4.5]{Tadmor:2006vn} state a spatial
regularity result for \eqref{TT-1} below in the range $n\ge2l$.
We verify \eqref{non-deg-stand} directly for all $n,l\ge1$ and apply
Theorem~\ref{t-regularity} to obtain the following space--time regularity
result (see also \cite{Nariyoshi:aa}).

\begin{corollary}\label{TT-ex} 
Let $n,l\in\N$ and consider the two-dimensional
convection-diffusion equation 
\begin{equation}\label{TT-1}
	\begin{split}
		\pa_t u &+\left( \frac{\pa}{\pa x_1}
		+\frac{\pa}{\pa x_2} \right) 
		\left(\frac{1}{l+1} u^{l+1}\right)
		\\ & 
		-\left(\frac{\pa^2}{\pa x_1^2}
		-2\frac{\pa^2}{\pa x_1 \pa x_2}
		+\frac{\pa^2}{\pa x_2^2} \right) 
		\left( \frac{1}{n+1}|u|^{n}u \right)=0,
	\end{split}
\end{equation} 
where $u=u(t,\mx)=u(t,x_1,x_2)$. 
This equation, augmented with the initial data 
$u_0$ from the space 
$L^1(\R^2)\cap L^\infty(\R^2)$, exhibits 
a regularization effect. Specifically, for any 
$q\in (1,q_*)$ and $s\in (0,2s_*)$, it is true that 
$$
u_0 \mapsto u\in  W_{\loc}^{s,q}(\R_+\times\R^2),
$$ 
where $u$ is the entropy solution to \eqref{TT-1} 
with initial data $u|_{t=0}=u_0$, and $q_*, s_*$ 
are the exponents given in Theorem \ref{t-regularity} for 
$\alpha=\min\seq{\frac{1}{2l},\frac{1}{n}}$ and $d=2$. 
\end{corollary}

\begin{proof}
Introduce the invertible linear change of variables
$$
y_1=\frac{x_1+x_2}{2},\qquad y_2=\frac{x_1-x_2}{2}.
$$
Then
$$
\pa_{x_1}+\pa_{x_2}=\pa_{y_1},\qquad
(\pa_{x_1}-\pa_{x_2})^2=\pa_{y_2}^2,
$$
so \eqref{TT-1} becomes
\begin{equation*}
	\begin{split}
		\pa_t u +\frac{\pa}{\pa y_1}
		\left(\frac{1}{l+1} u^{l+1}\right)
		-\frac{\pa^2}{\pa y^2_2}\left(
		\frac{1}{n+1}|u|^{n}u 
		\right)=0,
	\end{split}
\end{equation*}
where $u=u(t,\my)=u(t,y_1,y_2)$. 
The invertible linear coordinate change preserves
$W^{s,q}_{\loc}$ regularity and the exponents $s,q$.

We apply Theorem \ref{thm:main-intro} in the new coordinates.
In the notation of our kinetic 
equation \eqref{diffusive-intro} 
(with $B\equiv 0$), we have
$$
f(\lambda)=\left(\lambda^l,0\right), 
\quad 
a(\lambda)=\diag\left(0,\abs{\lambda}^n\right).
$$
Let $M=\norm{u_0}_{L^\infty}$. By the maximum principle,
$u$ takes values in $[-M,M]$. We verify \eqref{non-deg-stand} on
$I=[-M-1,M+1]$, so that we may choose $\rho\in C_c^1(I)$ equal to
one on $[-M,M]$.
For $\delta>0$ and $\mxi=(\xi_0,\xi_1,\xi_2)
\in \S^2$, i.e., 
$\xi_0^2+\xi_1^2+\xi_2^2=1$, we study
\begin{align*}
	\Lambda_\delta(\mxi) &:=
	\seq{\lambda\in I \, : \, 
	\abs{\xi_0+\bigl\langle f(\lambda)
	\,|\,\tilde{\mxi}\bigr\rangle}^2
	+\bigl\langle a(\lambda)\tilde{\mxi}
	\,|\,\tilde{\mxi} \bigr\rangle \leq\delta} 
	\\ & 
	=\seq{\lambda\in I \, : \, 
	\left(\xi_0+\lambda^l\xi_1\right)^2
	+\abs{\lambda}^n\xi_2^2\leq\delta}.
\end{align*}

The aim is to prove that there is an exponent 
$\alpha>0$ such that $\meas\Lambda_\delta(\mxi)
\lesssim\delta^\alpha$ 
for small values of $\delta$, where the implicit 
constant in $\lesssim$ is independent of $\mxi$. 
We will consider two cases:
$$
({\rm a}) \quad 
{\xi_2^2} < \frac{1}{4}, 
\qquad 
({\rm b}) \quad 
\xi_2^2 \geq \frac{1}{4}.
$$ 
In the case (b), we clearly have 
\begin{equation}\label{(a)}
	\begin{split}
		\meas\Lambda_\delta(\mxi) 
		&\leq \meas\seq{\lambda \in I \, :\, 
		\abs{\lambda}^n {\xi_2^2}
		\leq \delta}
		\\ &
		\leq \meas\seq{\lambda \in I 
		\, :\, \abs{\lambda}^n\leq 4\delta}
		\lesssim \delta^{\frac{1}{n}}.
	\end{split} 
\end{equation} 	
	
In the case (a), we have
$$
\xi_0^2+\xi_1^2 > 3/4 .
$$ 
Neglecting the second-order term, we get
\begin{equation}\label{eq:cor-flux}
\meas\Lambda_\delta(\mxi) 
\leq
\meas\seq{\lambda \in I \, :\, 
	\abs{\xi_0+\lambda^l\xi_1} 
	\leq \sqrt{\delta}} .
\end{equation}
The term on the right-hand side corresponds
to the symbol of the one-dimensional 
scalar conservation law
studied in \cite[(3.7)]{Tadmor:2006vn}. 
Hence, we have 
\begin{equation}\label{(b)}
	\meas\Lambda_\delta(\mxi)
	\lesssim \bigl(\sqrt{\delta}\bigr)^{\frac 1l}
	= \delta^{\frac{1}{2l}}.
\end{equation}

Therefore, from \eqref{(a)} and \eqref{(b)}, 
we conclude that \eqref{non-deg-stand} holds with 
$\alpha=\min\seq{\frac{1}{2l},\frac{1}{n}}$. 
Using Theorem \ref{t-regularity} and the discussion
at the beginning of the section, we conclude 
the proof of the corollary.
\end{proof}

\begin{remark}
Let us comment on Corollary \ref{TT-ex}  
in a specific case where $n=l=1$ (which 
is not covered by \cite[Corollary 4.5]{Tadmor:2006vn}). 
Here, we have $\alpha = 1/2$, so the critical 
values for the parameters are given by 
$$
q_*=\frac{25}{13}, 
\qquad 
2s_*= \frac{1}{75}.
$$
On the other hand, following the approach described in {\bf (ii)}, 
one obtains a slightly better $\alpha$, namely, 
$\alpha=\min\seq{\frac{1}{l},\frac{1}{n}}$. This is due to the fact
that then in \eqref{eq:cor-flux} the parameter 
$\delta$ appears with power $1$ (i.e., without 
a square root). Thus, for this choice 
of parameters, we have $\alpha=1$, which leads to 
the new critical values
$q_*=\frac{13}{7}$ and $2s_*=\frac{1}{39}$.
\end{remark}

\section*{Acknowledgement}

This work of Darko Mitrovi\'{c} was supported 
in part by the project P 35508 of the Austrian 
Science Fund (FWF). The current address of Darko 
Mitrovi\'{c} is University of Vienna, 
Oskar Morgenstern Platz-1, 1090 Vienna, Austria. 
The work of Kenneth H. Karlsen was 
funded by the Research Council of Norway under 
project 351123 (NASTRAN).
\begin{samepage}
The research of Marko Erceg is supported by the Croatian Science
Foundation (\mbox{UIP-2025-02-1337}) and by the European Union -- \mbox{NextGenerationEU}
through the National Recovery and Resilience Plan 2021--2026
(IK IA 1.1.3. Impact4Math, institutional grant of University of Zagreb,
Faculty of Science).\par
\end{samepage}

\section*{Data availability statement}
Data sharing is not applicable to this article as no new data
were created or analyzed in this study.

\section*{Use of AI tools}
No AI tools were used in preparing the originally submitted
manuscript. During the revision of the manuscript, ChatGPT-6
and SIRI AI (macOS 27) were used 
for language improvements and text refinement.



\end{document}